\documentclass[12pt,a4paper]{amsart}

\usepackage[T1]{fontenc}
\usepackage{amssymb,amsmath,latexsym,amsthm,paralist,a4,mathrsfs,dsfont}
\usepackage{IEEEtrantools}

\usepackage{tikz-cd}
\usetikzlibrary{decorations.markings}

\makeatletter
\tikzcdset{
  open/.code     = {\tikzcdset{hook, circled};},
  closed/.code   = {\tikzcdset{hook, slashed};},
  open'/.code    = {\tikzcdset{hook', circled};},
  closed'/.code  = {\tikzcdset{hook', slashed};},
  circled/.code  = {\tikzcdset{markwith = {\draw (0,0) circle (.375ex);}};},
  slashed/.code  = {\tikzcdset{markwith = {\draw[-] (-.4ex,-.4ex) -- (.4ex,.4ex);}};},
  markwith/.code ={
    \pgfutil@ifundefined%
    {tikz@library@decorations.markings@loaded}%
    {\pgfutil@packageerror{tikz-cd}{You need to say %
      \string\usetikzlibrary{decorations.markings} to use arrows with markings}{}}{}%
    \pgfkeysalso{/tikz/postaction = {
      /tikz/decorate,
      /tikz/decoration={markings, mark = at position 0.5 with {#1}}}
    }
  },
}
\makeatother

\usepackage[headings]{fullpage}
\usetikzlibrary{babel}
\usetikzlibrary{positioning,calc}
\usepackage{varioref}
\usepackage[pdfpagelabels, pdftex]{hyperref}
\hypersetup{
  pdftitle={The adic tame site},
  pdfauthor={Katharina H\"{u}bner},
  pdfsubject={},
  pdfkeywords={},
  colorlinks=true,    
  linkcolor=red,     
  citecolor=blue,     
  filecolor=blue,      
  urlcolor=blue,       
  breaklinks=true,
  bookmarksopen=true,
  bookmarksnumbered=true,
  pdfpagemode=UseOutlines,
  plainpages=false,
  unicode=true
  }
\usepackage[capitalise]{cleveref}
\usepackage{changes}  
\usepackage{bbold} 
\newcommand{\mapsfrom}{\mathrel{\reflectbox{\ensuremath{\mapsto}}}}

\newcommand{\id}{\mathrm{id}}

\newcommand{\im}{\mathrm{im}}

\newcommand{\Gal}{\mathrm{Gal}}

\newcommand{\Hom}{\mathrm{Hom}}

\newcommand{\m}{\mathfrak{m}}
\newcommand{\p}{\mathfrak{p}}

\newcommand{\Z}{\mathds Z}
\newcommand{\N}{\mathds N}
\newcommand{\Q}{\mathds Q}

\newcommand{\bC}{\mathds C}

\newcommand{\Tor}{\mathrm{Tor}}

\renewcommand{\O}{\mathcal{O}}

\newcommand{\GG}{\mathbb{G}}

\newcommand{\et}{\textit{\'et}}
\newcommand{\sh}{\mathrm{sh}}

\newcommand{\Nis}{\mathrm{Nis}}
\renewcommand{\top}{\mathrm{top}}

\newcommand{\A}{{\mathbb A}}
\newcommand{\ZZ}{\mathbb{Z}}
\renewcommand{\P}{{\mathbb P}}

\newcommand{\Sh}{\mathit{Sh}}

\newcommand{\fX}{{\mathfrak X}}

\newcommand{\cO}{{\mathscr O}}

\newcommand{\cS}{{\mathscr S}}

\newcommand{\cU}{{\mathscr U}}
\newcommand{\cV}{{\mathscr V}}
\newcommand{\cW}{{\mathscr W}}
\newcommand{\cX}{{\mathscr X}}
\newcommand{\cY}{{\mathscr Y}}
\newcommand{\cZ}{{\mathscr Z}}

\newcommand{\liso}{\mathrel{\hbox{$\longrightarrow$} \kern-2.4ex\lower-1ex\hbox{$\scriptstyle\sim$}\kern1.7ex}}

\newcommand{\set}{\textit{s\'et}}

\makeatletter
\newcommand{\zerounderset}[3][\mathord]{%
  #1{\vtop{
    \let\\\cr
    \baselineskip\z@skip\lineskip.25ex
    \ialign{\hidewidth$##$\hidewidth\crcr
      \omit$#3$\cr
      #2\crcr
    }%
  }}%
}
\makeatother

\newtheoremstyle{alexthm}
  {}
  {}
  {\sl }
  {}
  {\bf}
  {.}
  {.5em}
  {}
\theoremstyle{alexthm}

\newtheorem{theorem}{Theorem}[section]
\newtheorem*{theorem*}{Theorem}
\newtheorem{corollary}[theorem]{Corollary}
\newtheorem{proposition}[theorem]{Proposition}
\newtheorem{lemma}[theorem]{Lemma}
\newtheorem*{lemma*}{Lemma}

\newtheoremstyle{alexdef}
  {}
  {}
  {\rm }
  {}
  {\bf}
  {.}
  {.5em}
  {}
\theoremstyle{alexdef}
\newtheorem*{example*}{Example}

\newtheorem{remark}[theorem]{Remark}

\newtheorem{definition}[theorem]{Definition}

\DeclareMathOperator{\Spec}{\mathrm{Spec}}

\DeclareMathOperator{\Spa}{\mathrm{Spa}}
\DeclareMathOperator{\RZ}{\mathrm{RZ}}
\DeclareMathOperator{\supp}{\mathrm{supp}}

\DeclareMathOperator{\ch}{char}

\DeclareMathOperator*{\colim}{colim}

\DeclareMathOperator{\mor}{\mathrm{mor}}
\DeclareMathOperator{\ob}{\mathrm{ob}}

\DeclareMathOperator{\trdeg}{\textit{trdeg}}

\definecolor{darklimegreen}{RGB}{31,142,8}

\begin{document}

\hfuzz=4pt
\title{The adic tame site}
\author{Katharina H\"{u}bner}
\email{khuebner@mathi.uni-heidelberg.de}
\date{\today}
\address{Einstein Institute of Mathematics, Hebrew University, Jerusalem}
\thanks{This research is partly supported by ERC Consolidator Grant 770922 - BirNonArchGeom}

\begin{abstract}
 For every adic space~$\cX$ we construct a site~$\cX_t$, the tame site of~$\cX$.
 For a scheme~$X$ over a base scheme~$S$ we obtain a tame site by associating with $X/S$ an adic space $\Spa(X,S)$ and considering the tame site $\Spa(X,S)_t$.
 We examine the connection of the cohomology of the tame site with \'etale cohomology and compare its fundamental group with the conventional tame fundamental group.
 Finally, assuming resolution of singularities, for a regular scheme~$X$ over a base scheme~$S$ of characteristic $p > 0$ we prove a cohomological purity theorem for the constant sheaf $\Z/p\Z$ on $\Spa(X,S)_t$.
 As a corollary we obtain homotopy invariance for the tame cohomology groups of $\Spa(X,S)$.
\end{abstract}

\maketitle
\tableofcontents

\section{Introduction}


\'Etale cohomology of a scheme with torsion coefficients away from the residue characteristics yields a well behaved cohomology theory.
For instance, there is a smooth base change theorem, a cohomological purity theorem, and the cohomology groups are $\mathbb{A}^1$-homotopy invariant.
This breaks down, however, if we take the coefficients of the cohomology groups to be $p$-torsion, where~$p$ is a residue characteristic of the scheme in question.
The problem can be seen already when looking at the cohomology group $H^1_{\et}(\mathbb{A}^1_k,\Z/p\Z)$ for some algebraically closed field~$k$.
If the characteristic of~$k$ is not~$p$, this cohomology group vanishes.
But if the characteristic of~$k$ is~$p$, $H^1_{\et}(\mathbb{A}^1_k,\Z/p\Z)$ is infinite due to wild ramification at infinity.

In order to address these problems we introduce a tame site of a scheme~$X$ over some base scheme~$S$ which does not allow this wild ramification at the boundary.
The rough idea is to consider only \'etale morphisms $Y \to X$ which are tamely ramified (in an appropriate sense) along the boundary $\bar{X}-X$ of a compactification~$\bar{X}$ of~$X$ over~$S$.
The concept of tameness is a valuation-theoretic one.
This makes it more natural to work in the language of adic spaces rather than in the language of schemes.
For an \'etale morphism of adic spaces it is straightforward to define tameness:
An \'etale morphism $\varphi: \cY \to \cX$ is tame at a point $y \in \cY$ with $\varphi(y) = x$ if the valuation on $k(y)$ corresponding to~$y$ is tamely ramified in the finite separable field extension $k(y)|k(x)$.
Note that these valuations neither need to be discrete nor of rank one.
In this context tameness of $k(y)|k(x)$ is defined by requiring that the extension of the strict henselizations $k(y)^{\sh}|k(x)^{\sh}$ be of degree prime to the residue characteristic of the corresponding valuation rings.
Defining coverings to be the surjective tame morphisms, we obtain the tame site~$\cX_t$ for every adic space~$\cX$.
In addition, we define the strongly \'etale site~$\cX_{\set}$ by replacing \lq\lq tame\rq\rq with \lq\lq unramified\rq\rq.

This construction also provides a tame site for a scheme~$X$ over a base scheme~$S$ by associating with $X \to S$ the adic space $\Spa(X,S)$ (see \cite{Tem11}) and considering the tame site $\Spa(X,S)_t$.
Note that $\Spa(X,S)$ is not an analytic adic space:
If $X = \Spec A$ and $S = \Spec R$ are affine, we have $\Spa(X,S) = \Spa(A,A^+)$, where~$A^+$ is the integral closure of the image of~$R$ in~$A$ and~$A$ is equipped with the \emph{discrete} topology.
The adic space $\Spa(X,S)$ should not be thought of an analytification of $X/S$ but rather as a means of encoding the essential information on $X \to S$ in the language of adic spaces.
We call adic spaces which are locally of this type discretely ringed.

Of course, tameness is not a new concept in algebraic geometry.
Several approaches have been made to define the notion of a tame covering space of a scheme over a base scheme.
These are summarized and compared in \cite{KeSch10}.
Having a notion of tameness for covering spaces we can define the corresponding tame fundamental group.
In \cref{tamefundamentalgroup} we show that the fundamental group of the tame site coincides with the curve-tame fundamental group constructed in \cite{Wie08}, see also \cite{KeSch10}.

Also in other respects the tame site  behaves the way it should:
For an \'etale torsion sheaf with torsion away from the characteristic, the tame cohomology groups coincide with the \'etale cohomology groups.
If $X \to S$ is proper, the tame cohomology groups of $\Spa(X,S)$ coincide with the \'etale cohomology groups for all \'etale sheaves (see \cref{comparisonwithetalecohomology}).

Having established these rather straightforward comparison results, we move on to prove our first big theorem concerning the tame site, namely absolute cohomological purity for constant sheaves in characteristic $p > 0$ (see \cref{purity2}):
Let~$S$ be an excellent quasi-compact, quasi-separated scheme of characteristic $p > 0$ and $X$ a regular scheme which is separated and essentially of finite type over~$S$.
Assume that resolution of singularities holds over~$S$.
Under these assumptions~$X$ admits a regular compactification $\bar{X} \to S$ and we have
\[
  H^i(\Spa(X,S)_t,\Z/p\Z) \cong H^i(\bar{X}_{\et},\Z/p\Z).
\]
As a by-product we obtain that $H^i(\bar{X}_{\et},\Z/p\Z)$ is independent of the choice of compactification.
Purity immediately implies that under the hypothesis of resolution of singularities the tame cohomology groups $H^i(\Spa(X,S)_t,\Z/p\Z)$ are homotopy invariant for regular schemes~$X$ of finite type over~$S$ (see \cref{homotopyinvariance}).

In order to prove the purity theorem we examine the Artin-Schreier sequence
\[
 0 \to \Z/p\Z \longrightarrow \GG_a^+ \longrightarrow \GG_a^+ \to 0,
\]
on $\Spa(X,S)_t$, where~$\GG_a^+$ is the sheaf defined by $\GG_a^+(Z) = \cO_Z^+(Z)$.
It reduces us to the study of the cohomology of $\GG_a^+$.
The core of the argument is to establish in the course of Sections~\ref{cohomologydiscrete} to \ref{tamecohomology} the following chain of isomorphisms
\begin{equation} \label{isomorphism_chain}
 H^i(S,\cO_S) \cong H^i(\Spa(X,S),\GG_a^+) \cong H^i(\Spa(X,S)_{\set},\GG_a^+) \cong H^i(\Spa(X,S)_t,\GG_a^+).
\end{equation}
In \cref{cohomologydiscrete} we prove the left hand isomorphism.
This is where we use resolution of singularities to construct a basis of the topology of $\Spa(X,S)$ consisting of open subspaces of the form $\Spa(U,Y)$, where~$Y$ is regular and $U \subseteq Y$ an open subscheme.
Another important ingredient is the vanishing of the higher direct images of the structure sheaf under a projective birational morphism of regular schemes proved in \cite{ChRu15}.

In \cref{stronglyetalecohomology} we show the middle isomorphism in (\ref{isomorphism_chain}).
In preparation to this we examine in \cref{PrueferHuber} Pr\"ufer Huber pairs, i.e. Huber pairs $(A,A^+)$ such that $A^+ \to A$ is a Pr\"ufer extension.
Pr\"ufer Huber pairs are important in the study of the cohomology groups of $\GG_a^+$ because $\GG_a^+$ is acyclic on the adic spectra of Pr\"ufer Huber pairs.

The final step is the comparison of the strongly \'etale with the tame cohomology of $\GG_a^+$, i.e., the right hand isomorphism in (\ref{isomorphism_chain}).
More precisely, we show in \cref{tamecohomology} that for any noetherian, discretely ringed or analytic adic space~$\cX$ we have natural isomorphisms
\[
 H^i(\cX_{\set},\GG_a^+) \overset{\sim}{\longrightarrow} H^i(\cX_t,\GG_a^+)
 \]
for all $i \geq 0$. \\
\\
{\bf Acknowledgments}
 First of all I am grateful to Alexander Schmidt, whose idea it was to tackle the construction of a tame site.
 He provided me with many insights concerning the properties a tame site should satisfy and was a persistent critic of my ideas.
 I would like to thank Giulia Battiston and Johannes Schmidt for helpful preliminary discussions about the definition of the tame site.
 My thanks also go to Johannes Ansch\"utz who directed my attention to adic spaces.
 Finally I want to thank the referee for his/her helpful comments that led to many improvements.

\section{Background on adic spaces} \label{section_adic_spaces}

To fix notation let us briefly recall from \cite{Hu93} and \cite{Hu94} some notions concerning adic spaces.
A \emph{Huber ring} ($f$-adic ring in Huber's terminology) is a topological ring~$A$ such that there exists an open subring $A_0$ carrying the $I$-adic topology for a finitely generated ideal $I \subseteq A_0$.
The ring~$A_0$ is called a \emph{ring of definition} of~$A$ and the ideal~$I$ an \emph{ideal of definition}.
An example of a Huber ring is $\Q_p$ with ring of definition~$\Z_p$ and ideal of definition $p\Z_p$.

An element~$a$ of a Huber ring~$A$ is \emph{power-bounded} if the set $\{a^n \mid n \in \N\}$ is bounded, i.e. for any neighborhood $U \subset A$ of~$0$ there is a neighborhood~$V$ of~$0$ such that
\[
 V\cdot\{a^n \mid n \in \N\} \subseteq U.
\]
An element~$a$ of~$A$ is called \emph{topologically nilpotent} if the sequence $a^n$ converges to~$0$.
Every topologically nilpotent element is power-bounded.
We denote the set of power bounded elements of~$A$ by $A^{\circ}$ and the set of topologically nilpotent elements by $A^{\circ\circ}$.

A \emph{ring of integral elements} of~$A$ is an open, integrally closed subring~$A^+$ of~$A$ that is contained in~$A^{\circ}$.
The rings of integral elements are precisely the integrally closed subrings~$A^+$ of~$A$ such that
\[
 A^{\circ\circ} \subseteq A^+ \subseteq A^{\circ}.
\]
A \emph{Huber pair} (affinoid ring in Huber's terminology) is a pair $(A,A^+)$ consisting of a Huber ring~$A$ and a ring of integral elements $A^+ \subseteq A$.

Given a Huber pair $(A,A^+)$ we define its \emph{adic spectrum}
\[
 \cX = \Spa(A,A^+) = \{\text{continuous valuations $v:A \to \Gamma \cup \{0\}$} \mid v(a) \le 1~\forall\,a \in A^+\}.
\]
Notice that we write valuations multiplicatively.
Furthermore, for an element $x \in \cX$ we write $f \mapsto |f(x)|$ for the valuation corresponding to~$x$.

For $f_1,\ldots,f_n,g \in A$ such that the ideal of~$A$ generated by $f_1,\ldots,f_n$ is open, we define the \emph{rational subset} $R\big(\frac{f_1,\ldots,f_n}{g}\big)$ of~$\cX$ by
\[
 R\big(\frac{f_1,\ldots,f_n}{g}\big) = \{x \in \cX \mid |f_i(x)| \le |g(x)| \ne 0~\forall\,i=1,\ldots,n\}.
\]
It is the adic spectrum of the Huber pair
\[
 (A(\frac{f_1,\ldots,f_n}{g}),A(\frac{f_1,\ldots,f_n}{g})^+),
\]
 where $A ( \frac{f_1,\ldots,f_n}{g})$ is the localization~$A_g$ of~$A$ endowed with the topology defined by the ring of definition $A_0[\frac{f_1}{g},\ldots,\frac{f_n}{g}]$
 and the ideal of definition $IA_0[\frac{f_1}{g},\ldots,\frac{f_n}{g}]$ and $A(\frac{f_1,\ldots,f_n}{g})^+$ is the integral closure of $A^+[\frac{f_1}{g},\ldots,\frac{f_n}{g}]$ in $A(\frac{f_1,\ldots,f_n}{g})$.
We endow~$\cX$ with the topology generated by the rational subsets as above.

On the topological space~$\cX$ we can define a presheaf~$\cO_{\cX}$ of complete topological rings (complete always comprises Hausdorff) such that for any rational subset $R\big(\frac{f_1,\ldots,f_n}{g}\big)$ of~$\cX$ we have
\[
 \cO_{\cX}(R\big(\frac{f_1,\ldots,f_n}{g}\big)) = A \langle \frac{f_1,\ldots,f_n}{g} \rangle,
\]
the latter ring being the completion of $A(\frac{f_1,\ldots,f_n}{g})$.
In particular,
\[
 \cO_{\cX}(\cX) = \hat{A}.
\]
Furthermore, there is a subpresheaf~$\cO_{\cX}^+$ of~$\cO_{\cX}$ with
\[
 \cO_{\cX}^+(R\big(\frac{f_1,\ldots,f_n}{g}\big)) = A \langle \frac{f_1,\ldots,f_n}{g} \rangle^+.
\]
We say that a Huber pair $(A,A^+)$ is \emph{sheafy} if the corresponding presheaf~$\cO_{\cX}$ on $\cX = \Spa(A,A^+)$ is a sheaf.
In this case we speak of the \emph{structure sheaf}~$\cO_{\cX}$.
If~$\cO_{\cX}$ is a sheaf,~$\cO_{\cX}^+$ is a sheaf, as well.
The Huber pair $(A,A^+)$ is known to be sheafy in the following cases:
\begin{enumerate}
 \item $\hat{A}$ has a noetherian ring of definition over which~$\hat{A}$ is finitely generated.
 \item $A$ is a strongly noetherian Tate ring.
 \item The topology of~$A$ is discrete.
\end{enumerate}
The main problem is caused by completion being not exact in general.
In cases~(1) and~(2) the sheaf property is a non-trivial result that has been shown in \cite{Hu94}, Theorem~2.2.
Case~(1) relies on the observation that for a noetherian adic ring~$A$, every homomorphism of finite~$A$-modules is strict and case~(2) reduces to Tate's acyclicity theorem.
In case~(3), however, due to the topology being discrete, there are no issues with completion.
In fact, the structure sheaf~$\cO_{\cX}$ can be identified with the pullback of the structure sheaf~$\cO_{\Spec A}$ on $\Spec A$ along the natural morphism
\[
\supp: \Spa(A,A^+) \longrightarrow \Spec A
\]
that maps a valuation to its support.
This is a continuous morphism as the preimage of a fundamental open $D(g) = \{\p \in \Spec A \mid g \notin \p\}$ of $\Spec A$ is the rational subset $R\big(\frac{g}{g}\big)$ of $\Spa(A,A^+)$.
It is also open as it maps a rational subset $R\big(\frac{f_1,\ldots,f_n}{g}\big)$ to $\Spec A_g$.
Therefore, it is easy to compute the presheaf pullback $\supp^{-1}(\cO_{\Spec A})$ and check that the sheaf condition is satisfied using that $\cO_{\Spec A}$ is a sheaf.

Throughout this article we will only consider Huber pairs satisfying one of the above conditions.
Mostly we will be concerned with Huber pairs of type~(3), where we may assume the topology of~$A$ to be discrete.

An adic space is a triple $(\cX,\cO_{\cX},(v_x)_{x \in \cX})$, where
\begin{itemize}
 \item $\cX$ is a topological space,
 \item $\cO_{\cX}$ is a sheaf of complete topological rings whose stalks are local rings,
 \item for every $x \in \cX$, $v_x$ is an isomorphism class of valuations on~$\cO_{\cX,x}$ whose support is the maximal ideal of~$\cO_{\cX,x}$,
\end{itemize}
which is locally isomorphic to $\Spa(A,A^+)$ for a sheafy Huber pair $(A,A^+)$.

Unfortunately, closed subsets of adic spaces do not carry the structure of an adic space in general.
Therefore, following \cite{Hu96}, \S 1.10, we define \emph{prepseudo-adic spaces} to be pairs $\cX = (\underline{\cX},|\cX|)$, where~$\underline{\cX}$ is an adic space and~$|\cX|$ is a subset of (the underlying topological space of)~$\underline{\cX}$.
The subset~$|\cX|$ of~$\underline{\cX}$ is called \emph{convex} if for any chain of specializations
\[
 x_1 \rightsquigarrow x_2 \rightsquigarrow x_3
\]
in~$\underline{\cX}$ such that $x_1$, $x_3 \in |\cX|$, it follows that $x_2 \in |\cX|$.
Moreover, $|\cX|$ is \emph{pro-constructible} if it is closed in the constructible topology of~$\underline{\cX}$ and \emph{locally pro-constructible} if it is pro-constructible in an open subset of~$\underline{\cX}$.
A prepseudo-adic space $\cX$ is called \emph{pseudo-adic space} if~$|\cX|$ is convex and locally pro-constructible.
In particular, any closed subset~$\cZ$ of an adic space~$\cY$ defines a pseudo-adic space.
If~$\cY$ is an adic space and~$\cZ$ is a subset of~$\cY$, we often use the same letter~$\cZ$ to denote the prepseudo-adic space $(\cY,\cZ)$.

For the present work it would not be essential to work with the more general pseudo-adic spaces instead of just adic spaces.
The only pseudo-adic spaces appearing naturally in the proof of cohomological purity are of the form $(\Spa(k,k^+),\{s\})$, where~$k$ is a field, $k^+$ a valuation ring of~$k$ and~$s$ the closed point of $\Spa(k,k^+)$.
We could deal with these objects without introducing the notion of a pseudo-adic space.
However, in subsequent work we plan to treat constructible sheaves and base change theorems.
We expect that even if we are only interested in adic spaces in the end, we will have to deal with pseudo-adic spaces.
For this reason many results in the first half of the present article are formulated for pseudo-adic spaces for future reference.

In this article we will be especially interested in adic spaces that are locally of the form $\Spa(A,A^+)$ where~$A$ carries the discrete topology.
We call this type \emph{discretely ringed} adic spaces.
An important construction that produces discretely ringed adic spaces is described in \cite{Tem11}, \S~3.1.
Starting with a morphism of schemes $X \to S$, Temkin constructs an adic space $\Spa(X,S)$.
The points of $\Spa(X,S)$ are triples $(x,R,\phi)$, where~$x$ is a point of~$X$,~$R$ is a valuation ring of~$k(x)$ and $\phi : \Spec R \to S$ is a morphism compatible with $\Spec k(x) \to S$.
In case $S$ is separated,~$\phi$ is uniquely determined (if it exists) by $(x,R)$.
The topology of $\Spa(X,S)$ is generated by the subsets $\Spa(X',S')$ of $\Spa(X,S)$ coming from commutative diagrams
\[
 \begin{tikzcd}
  X'	\ar[r]	\ar[d]	& X	\ar[d]	\\
  S'	\ar[r]			& S
 \end{tikzcd}
\]
with $X' \to X$ an open immersion and $S' \to S$ separated and of finite type.
This construction is compatible with Huber's definition of the adic spectrum given in \cite{Hu93}:
If $X = \Spec A$ and $S = \Spec A^+$ are affine and the homomorphism $A^+ \to A$ is injective with integrally closed image, $\Spa(X,S)$ coincides with Huber's $\Spa(A,A^+)$
 (where~$A$ is equipped with the discrete topology).

Pulling back the structure sheaf of~$X$ via the support morphism
\[
 \supp : \cX := \Spa(X,S) \to X, \quad (x,R,\phi) \mapsto x
\]
we obtain a sheaf of rings~$\mathcal{O}_{\cX}$ on $\cX = \Spa(X,S)$ making $\cX$ a locally ringed space with
\[
 \cO_{\cX,(x,R,\phi)} = \cO_{X,x}.
\]
For each point $z = (x,R,\phi)$ denote by $v_z$ the equivalence class of valuations on $k(x)$ corresponding to~$R$.
We obtain a discretely ringed adic space $(\cX,\O_{\cX},(v_z \mid z \in \cX))$.
Checking functoriality yields:

\begin{lemma}
 The above assignment defines a functor
 \begin{align*}
  \Spa : \{\textit{morphisms of schemes}\}	& \longrightarrow	 \{\textit{discretely ringed adic spaces}\}	\\
  (X \to S)					& \mapsto		 (\cX=\Spa(X,S),\O_{\cX},(v_z \mid z \in \cX)).
 \end{align*}
 mapping morphisms of affine schemes to affinoid adic spaces.
\end{lemma}

Where no confusion can arise we write $\Spa(X,S)$ for the adic space
\[
 (\cX=\Spa(X,S),\cO_{\cX},(v_z \mid z \in \cX)).
\]
An important property is the following observation.

\begin{lemma} \label{properisomorphismonSpa}
 Let $X \to S'$ be a morphism of schemes and $S' \to S$ a proper morphism of schemes.
 Then
 \[
  \Spa(X,S') \cong \Spa(X,S).
 \]
\end{lemma}

\begin{proof}
 As $S' \to S$ is of finite type and separated, the natural morphism $\Spa(X,S') \to \Spa(X,S)$ is an open immersion.
 In order to check surjectivity, take a point $(x,R,\phi) \in \Spa(X,S)$.
 The morphism $\phi: \Spec R \to S$ lifts (uniquely) to a morphism $\phi': \Spec R \to S'$ by the valuative criterion for properness.
 Hence, $(x,R,\phi')$ is a preimage in $\Spa(X.S')$ of $(x,R,\phi)$.
\end{proof}

We have two natural morphisms of locally ringed spaces attached to a morphism of schemes $\pi: X \to S$ that will appear throughout the article.
The first one is the support morphism
\[
 \supp: (\cX=\Spa(X,S),\cO_{\cX}) \to (X,\cO_X)
\]
whose underlying morphism of topological spaces is the one mentioned earlier, that sends $(x,R,\phi)$ to~$x$.
On the level of structure sheaves it is tautological because~$\cO_{\cX}$ is the pullback of~$\cO_X$ by definition.
The second morphism is the center morphism
\[
 (c,c^+): (\cX = \Spa(X,S),\cO^+_{\cX}) \to (S,\cO_S).
\]
The morphism $c$ sends $(x,R,\phi)$ to the image of the closed point of $\Spec R$ under the map $\phi: \Spec R \to S$.
It is continuous as the preimage of an open subset $S' \subseteq S$ is the open subset $\Spa(X \times_S S',S')$ of $\Spa(X,S)$.
In order to define the corresponding homomorphism of sheaves $c^+: \cO_S \to c_*\cO_{\cX}^+$, we first note that $c_*\cO_{\cX}$ is naturally identified with $\pi_*\cO_X$ as $\cO_{\cX} = \supp^{-1}\cO_X$.
Hence, the homomorphism $\cO_S \to \pi_* \cO_X$ induces a functorial homomorphism
\[
 \cO_S \to c_*\cO_{\cX}.
\]

\begin{lemma}
 The homomorphism $\cO_S \to c_*\cO_{\cX}$ factors through $c_*\cO_{\cX}^+$.
\end{lemma}

\begin{proof}
 It is equivalent to show that the adjoint homomorphism $c^{-1}\cO_S \to \cO_{\cX}$ factors through~$\cO_{\cX}^+$.
 It suffices to check this for affinoid opens $\Spa(A,A^+)$ of~$\cX$ and the presheaf pullback $c^p\cO_S$.

 The sections $c^p\cO_S(\Spa(A,A^+))$ are given as the colimit of $\cO_S(S')$ over all commutative diagrams
 \begin{equation} \label{presheafpullback}
  \begin{tikzcd}
   \Spa(A,A^+)		\ar[r]		\ar[d,open]	& S'	\ar[d,open]	\\
   \cX = \Spa(X,S)	\ar[r,"c"]				& S
  \end{tikzcd}
 \end{equation}
 with~$S'$ an open subscheme of~$S$:
 \[
  c^p\cO_S(\Spa(A,A^+)) = \colim_{S'} \cO_S(S').
 \]
 The homomorphism $c^p\cO_S(\Spa(A,A^+)) \to \cO_{\cX}(\Spa(A,A^+))$ is the colimit of the homomorphisms
 \[
  \begin{tikzcd}
   \cO_S(S')	\ar[r]	& \cO_{\cX}(\Spa(X \times_S S',S'))	\ar[r]	\ar[d,dash,shift left=.5]	\ar[d,dash,shift right=.5]	& \cO_{\cX}(\Spa(A,A^+))	\ar[d,dash,shift left=.5]	\ar[d,dash,shift right=.5]	\\
			& \cO_X(X \times_S S')		\ar[r]									& A.
  \end{tikzcd}
 \]
 We want to show that $\cO_S(S') \to A$ factors through
 \[
  A^+ = \{ a \in A \mid |a(x)| \leq 1~\forall x \in \Spa(A,A^+)\}.
 \]
 Let~$x \in \Spa(A,A^+)$.
 By the commutativity of diagram~(\ref{presheafpullback}), the valuation of~$A$ corresponding to~$x$ has center on~$S'$, which is equivalent to saying that~$|b(x)| \le 1$ for all $b \in \cO_S(S')$.
 This implies the claim.
\end{proof}

The resulting map
\[
 \cO_S \to c_*\cO_{\cX}^+
\]
is the homomorphism~$c^+$ we wanted to define.
It will play a crucial role when computing the cohomology of~$\cO_{\cX}^+$ in \cref{cohomologydiscrete}.
Moreover, it is handy for explaining the connection of $\Spa(X,S)$ with compactifications of~$X$ over~$S$.
For the rest of the section assume that $\pi: X \to S$ is a separated morphism of qcqs schemes.
Recall from \cite{Tem11}, \S~2.1 that an $X$-modification of~$S$ is a factorization
\[
 X \overset{\pi_i}{\longrightarrow} S_i \overset{g_i}{\longrightarrow} S
\]
of $\pi: X \to S$ into a schematically dominant morphism~$\pi_i$ and a proper morphism~$g_i$.
If $\pi$ is an open immersion with dense image, these are just the usual modifications of~$S$ outside~$X$.
By \cref{properisomorphismonSpa} we can identify all the spaces $\Spa(X,S_i)$ with $\cX:=\Spa(X,S)$.
The $X$-modifications of~$S$ form a cofiltered inverse system compatible with the center maps
\[
 (c_i,c_i^+) : (\cX,\cO_{\cX}^+) \longrightarrow (S_i,\cO_{S_i}).
\]
The limit of the $(S_i,\cO_{S_i})$ exists in the category of locally ringed spaces and is called the \emph{relative Riemann-Zariski space} $\fX = \RZ_X(S)$ of $\pi: X \to S$ with its sheaf of regular functions~$\cO_{\fX}$ (see \cite{Tem11}, Remark~2.1.1).
In a similar manner as for classical Riemann-Zariski spaces, the relative Riemann-Zariski space $\fX$ also has a valuation theoretic description.
With the above definitions we can phrase it in the following way.
There exists an embedding $\RZ_X(S) \to \Spa(X,S)$ of locally ringed spaces such that the composition
\[
 (\fX =\RZ_X(S),\cO_{\fX}) \to (\cX = \Spa(X,S),\cO^+_{\cX}) \to (\fX,\cO_{\fX})
\]
is the identity (where the latter morphism is the limit of the center maps).
The image of the first morphism consists of all points $(x,R,\phi)$ that do not admit a non-trivial horizontal specialization (i.e. a specialization that has the same center in every~$S_i$).
We call these points \emph{Riemann-Zariski points}.
The above stated assertions are not at all trivial and in fact are the main results of \cite{Tem11}.
Morally one should think of $\RZ_X(S)$ as a deformation retract of~$\Spa(X,S)$.

If $\pi:X \to S$ is of finite type, $\RZ_X(S)$ is the limit over all compactifications of~$X$ over~$S$.
Having moreover an interpretation in terms of valuations, it seems well suited for the definition of a tame site.
The main reason why we chose to work with $\Spa(X,S)$ instead of $\RZ_X(S)$ (even though it has many more points) is that $\Spa(X,S)$ has better functorial properties.
For a commutative square
\[
 \begin{tikzcd}
  Y	\ar[r]	\ar[d]	& X	\ar[d]	\\
  T	\ar[r]			& S
 \end{tikzcd}
\]
of schemes, the definition of the associated morphism $\Spa(Y,T) \to \Spa(X,S)$ is very natural.
However, it does not take Riemann-Zariski points to Riemann-Zariski points, in general.
In order to obtain a morphism $\RZ_Y(T) \to \RZ_X(S)$, we have to consider the composition
\[
 \RZ_Y(T) \to \Spa(Y,T) \to \Spa(X,S) \to \RZ_X(S),
\]
which is quite a complicated construction.

%

\section{The strongly \'etale and the tame site} \label{section_set_t}

Recall from \cite{Hu96}, Definition~1.6.5 i) that a morphism of adic spaces $\cY \to \cX$ is \'etale if it is locally of finite presentation and if, for any Huber ring $(A,A^+)$, any ideal~$I$ of~$A$ with $I^2 = \{0\}$, and any morphism $\Spa(A,A^+) \to \cX$, the mapping
\[
 \Hom_{\cX}(\Spa(A,A^+),\cY) \to \Hom_{\cX}(\Spa(A,A^+)/I,\cY)
\]
is bijective.
In order to obtain a ring theoretic description we make the following definition.

\begin{definition}
 A homomorphism of Huber pairs $(A,A^+) \to (B,B^+)$ is \emph{\'etale} if it is algebraically of finite type and $A \to B$ is \'etale (in the classical ring theoretic sense).
\end{definition}

If the topology of~$A$ is discrete, this boils down to $B/A$ being \'etale and $B^+$ being the integral closure in~$B$ of a finite type $A^+$-subalgebra of~$B$.
Locally, an \'etale homomorphism of adic spaces is isomorphic to a morphism $\Spa(B,B^+) \to \Spa(A,A^+)$, coming from an \'etale morphism $(A,A^+) \to (B,B^+)$ of Huber pairs (see \cite{Hu96}, Corollary~1.7.3~iii)).

We want to consider \'etale morphisms with an additional valuation theoretic constraint on the residue field extensions.
In order to define them, let us recall a bit of ramification theory for extensions of valued fields.
Remember that we are dealing with general valuations, not just discrete valuations of rank one.
We consider a finite extension of valued fields $L|K$ and fix an algebraic closure~$\bar{L}$ of~$L$ and a prolongation~$\bar{v}$ of the valuation of~$L$ to~$\bar{L}$.
We can compute the strict henselizations $L^{\sh}$ and $K^{\sh}$ with respect to~$\bar{v}$ as subfields of $\bar{L}$.

The extension $L|K$ is called \emph{unramified} if $L^{\sh} = K^{\sh}$.
It is \emph{tame} if the degree of $L^{\sh}|K^{\sh}$ is prime to the residue characteristic.
If $L|K$ is not unramified, we say that it is \emph{ramified} and if it is not tame, it is \emph{wild}.
These notions are independent of the choice of~$\bar{L}$ and~$\bar{v}$.

Unramified extensions $L|K$ of valued fields are closely related to \'etale ring homomorphisms as the following lemma shows.
Because we use it in the proof, we want to remind the reader of the concept of \emph{Pr\"ufer domains}.
They are integral rings such that all of their localizations at some prime ideal are valuation rings.
For us the most crucial property of a Pr\"ufer domain~$A$ is that any $A$-algebra~$B$ that sits between~$A$ and its quotient field~$K$, $A \subseteq B \subseteq K$, is an intersection of localizations~$A_{\p}$ at a prime ideal $\p$ of~$A$.
Moreover, if the overring is a valuation ring, it coincides with the localization of~$A$ at a unique prime ideal~$\p$ (see \cite{Gil92}, Theorem~26.1).

\begin{lemma} \label{unramified_etale}
 Let $L|K$ be a finite extension of valued fields and denote their valuation rings by $L^+$ and~$K^+$, respectively.
 Then $L|K$ is unramified if and only if $L^+$ is the localization of an \'etale $K^+$-algebra.
\end{lemma}

\begin{proof}
 If $L^+$ is the localization of an \'etale $K^+$-algebra, it is clear from the definition that $L|K$ is unramified.
 Suppose that $L|K$ is unramified.
 Then $L^+ \subseteq L^{+\sh} = K^{+\sh}$.
 Since $L|K$ is finite, there is an \'etale $K^+$-subalgebra~$A$ of $L^+$ whose quotient field is~$L$.
 It contains the integral closure of~$K^+$ in~$L$, which is a Pr\"ufer domain.
 Hence, $A$ is a Pr\"ufer domain itself and $L^+$ is a localization of~$A$ (see \cite{Gil92}, Theorem~26.1).
 \end{proof}

In case $L|K$ is Galois, the notions of unramified and tame extensions can also be defined via the action of the Galois group $G_{L|K}$.
Recall that the decomposition group of~$G_{L|K}$ with respect to the valuation $v_L$ of~$L$ is defined as follows
\[
 D_{L|K} = \{g \in G_{L|K} \mid v_L \circ g = v_L\}.
\]
It can also be interpreted as the Galois group of the extension of the henselizations $L^h|K^h$.
By the definition of the decomposition group, we obtain an induced action of~$D_{L|K}$ on the residue field $\kappa_L$ of the valuation ring~$L^+$ corresponding to~$v_L$.
Inside the decomposition group there is a normal subgroup, the inertia group
\[
 I_{L|K} = \{g \in D_{L|K} \mid gx = x~\forall x \in \kappa_L\}.
\]
If the characteristic~$p$ of~$\kappa_L$ is positive, the inertia group has a unique $p$-Sylow subgroup, the wild inertia group~$R_{L|K}$, also known as ramification group.
By \cite{Ray70}, Chapitre~X, Th\'eor\`eme~1, $L|K$ is unramified if and only if the inertia group~$I_{L|K}$ is trivial.
It is tame if and only if the wild inertia group~$R_{L|K}$ is trivial (see \cite{EP2005} Theorem~5.3.3)

\begin{definition}
  A morphism of prepseudo-adic spaces $f: \cY \to \cX$ is called \emph{strongly \'etale} (resp. \emph{tame}) at a point $y \in |\cY|$ if~$f$ is \'etale at~$y$ and the extension on residue fields $k(y)|k(f(y))$ is tamely ramified with respect to the valuation~$|\cdot(y)|$.
  The morphism~$f$ is called strongly \'etale (resp. tame) if~$f$ is so at every point of~$\cY$.
  A homomorphism of Huber pairs $(A,A^+) \to (B,B^+)$ is \emph{strongly \'etale} (resp. \emph{tame}) if it is \'etale and $\Spa(B,B^+) \to \Spa(A,A^+)$ is strongly \'etale (resp. tame).
\end{definition}

Recall that an affinoid field is a Huber pair $(k,k^+)$ such that~$k$ is a field, $k^+$ a valuation ring of~$k$, and the topology of~$k$ is either discrete or induced by the valuation corresponding to~$k^+$.
If the topology of~$k$ is non-discrete, $k^+$ has a prime ideal~$\p$ of height one and the topology is also generated by the valuation corresponding to~$k^+_{\p}$ (see \cite{Hu96}, Definition~1.1.4).
In particular, $k$ is a non-Archimedean field.
Note that when we say that $(k,k^+)$ is complete, we are referring to completeness with respect to the given topology of~$k$.
In particular, every affinoid field with discrete topology is complete.

\begin{lemma} \label{comparenotionsofetale}
 Let $(k,k^+)$ be a complete affinoid field.
 An \'etale morphism $\Spa(A,A^+) \to \Spa(k,k^+)$ is strongly \'etale if and only if $k^+ \to A^+$ is \'etale.
\end{lemma}

\begin{proof}
 By \cite{Hu96}, Cor.~1.7.3~iii) the ring homomorphism $k \to A$ is \'etale and~$A^+$ is the integral closure of an open subring~$A'$ of~$A$ which is of finite type over~$k^+$.
 (Note that since~$k$ is a field, every \'etale homomorphism $k \to B$ is finite \'etale.
 Hence,~$B$ is automatically complete).
 Therefore, we may assume that $A$ is a field and $k \to A$ is a finite separable field extension.
 
 Then~$A^+$ is a Pr\"ufer domain in~$A$.
 In this case the center map
 \[
  c: \Spa(A,A^+) \to \Spec A^+
 \]
 is an isomorphism.
 Indeed, an inverse is given by mapping $\p \in \Spec A^+$ to $(\eta,A^+_{\p},\phi)$, where~$\eta$ is the unique point of $\Spec A$ and $\phi: \Spec A^+_{\p} \to \Spec A^+$ is the localization morphism.
 
 If $k^+ \to A^+$ is \'etale, it follows from \cref{unramified_etale} that every localization of~$A^+$ at some prime is unramified over the corresponding localization of~$k^+$.
 Since~$c$ is an isomorphism, this means that $\Spa(A,A^+) \to \Spa(k,k^+)$ is strongly \'etale.
 
 Now assume that $\Spa(A,A^+) \to \Spa(k,k^+)$ is strongly \'etale.
 Let $\p$ be a prime ideal of~$A^+$ (corresponding via~$c$ to the point of $\Spa(A,A^+)$ given by the valuation ring $A^+_{\p}$).
 By \cref{unramified_etale}, there is an \'etale $k^+$-algebra~$B$ (contained in~$A$) such that $A^+_{\p} = B_{\p}$.
 Modifying~$B$ we may assume that $A^+ \subseteq B$ and $\Spec B \to \Spec A^+$ is an open immersion.
 This shows that $\Spec A^+ \to \Spec k$ is \'etale at~$\p$.
\end{proof}

Let~$\cX$ be a prepseudo-adic space.
We define the following sites over~$\cX$ called the \emph{strongly \'etale} site~$\cX_{\set}$ and the \emph{tame site}~$\cX_t$:
\begin{itemize}
 \item	The underlying categories of~$\cX_{\set}$ and~$\cX_t$ are the categories of strongly \'etale and tame morphisms $f: \cY \to \cX$, respectively.
 \item	Coverings are families $\{f_i:\cY_i \to \cY\}_{i\in I}$ of strongly \'etale, respectively tame, morphisms such that	
	\[
	 |\cY| = \bigcup_{i \in I} f_i (|\cY_i|).
	\]
\end{itemize}

In order to show that this definition makes sense, we have to convince ourselves that tameness and strong \'etaleness are stable under compositions and base change.
But this follows by combining the corresponding stability  results of \'etaleness (\cite{Hu96}, Proposition~1.6.7) and extensions of valued fields (\cite{EP2005}, \S 5).
In addition, it follows from the same references that a morphism between two objects of~$\cX_{\set}$ (resp. $\cX_t$) is automatically strongly \'etale (resp. tame).

For a morphism of schemes $X \to S$ the \emph{tame site} of $X \to S$ is defined to be the tame site of $\Spa(X,S)$.
Let us explain in what sense this models \'etale morphisms of schemes that are tamely ramified at the boundary.
Assume for simplicity that $X \to S$ is a morphism of finite type of noetherian schemes.
The tame site of $\Spa(X,S)$ is generated by objects of the form $\Spa(Y,T) \to \Spa(X,S)$ induced by commutative squares of schemes
\[
 \begin{tikzcd}
  Y	\ar[r,"f"]	\ar[d]	& X	\ar[d]	\\
  T	\ar[r,"g"]			& S,
 \end{tikzcd}
\]
where $f$ is \'etale and~$g$ is of finite type such that for every $y \in Y$ mapping to $x \in X$ and every valuation $v$ of $k(y)$ with center on~$T$, the field extension $k(y)|k(x)$ is tamely ramified with respect to~$v$.
We may further assume that $Y \to T$ is separated.
Then there exists a compactification $\bar{Y}$ of~$Y$ over~$T$ (see \cite{Con07}).
By the valuative criterion for properness, the natural morphism $\Spa(Y,\bar{Y}) \to \Spa(Y,T)$ is an isomorphism.
We may therefore take $Y \to T$ to be an open immersion.
We can view~$T$ as a sort of partial compactification of~$Y$ over~$S$.
The tameness condition heuristically says that $f$ should be tamely ramified at points of~$T$.
More precisely, tameness at a point $t \in T$ is encoded in the tameness of $\Spa(f,g)$ at all points $(y,R,\phi)$ of $\Spa(Y,T)$ with $c(y,R,\phi) = t$ ($c$ is the center morphism from above).

\section{Openness of the tame locus}

Our aim is to show that the strongly \'etale and the tame locus of an \'etale morphism of adic spaces is open.
The argument is similar to the one for Riemann-Zariski spaces given in \cite{Tem17}.
First we prove that strongly \'etale morphisms are locally of a standardized form just as \'etale morphisms of schemes are locally standard \'etale.
The proof of this statement follows the arguments given in \cite[Tag 00UE]{stacks-project}.

Before we start let us introduce some notation.
For a ring~$A$ with a valuation~$v$ we define the following valuation~$v_1$ on the polynomial ring $A[T]$.
For a polynomial $f(T) = f_nT^n + \ldots + f_0$ we set
\[
 v_1(f(T)) = \max_{i=0,\ldots,n} v(f_i).
\]
It is easy to check that~$v_1$ is indeed a valuation.
The index~$1$ stands for radius~$1$.
There is a more general construction for arbitrary radius but we do not need it here.

\begin{proposition} \label{standardetale}
 Let $\varphi:Y \to X$ be an \'{e}tale morphism of schemes, $y \in Y$ and~$w$ a valuation of $k(y)$.
 Set $x = \varphi(y)$ and $v = w|_{k(x)}$.
 Suppose that~$w$ is unramified in the finite separable field extension $k(y)|k(x)$.
 Then there exists an affine open neighborhood $\Spec A$ of~$x$ and $f,g \in A[T]$ with $f = T^n + f_{n-1} T^{n-1} + \ldots + f_0$ monic and $f'$ a unit in
 $$
 B = \left( A[T]/(f) \right)_g
 $$
  such that $\Spec B$ is isomorphic over~$A$ to an open neighborhood of~$y$
  and $v_1(f(T)) \le 1$ and $w(g) = 1$ (viewing~$g$ as an element of~$B$ and~$w$ as a valuation of~$B$).
\end{proposition}

\begin{proof}
 We may assume that $X = \Spec A$ and $Y = \Spec B$ are affine.
 Denote by $\mathfrak{p} \subseteq A$ and $\mathfrak{q} \subseteq B$ the prime ideals corresponding to~$x$ and~$y$.

 There exists an \'{e}tale ring homomorphism $A_0 \to B_0$ with~$A_0$ of finite type over~$\mathbb{Z}$ and a ring homomorphism $A_0 \to A$ such that $B = A \otimes_{A_0} B_0$.
 Denote the image of $y$ in $\Spec B_0$ by $y_0$ and the restriction of~$w$ to $k(y_0)$ by~$w_0$.
 Then it suffices to prove the lemma for $\Spec B_0 \to \Spec A_0$ and~$(y_0,w_0)$ instead of~$\varphi$ and~$(y,w)$.
 Hence, we may assume that~$A$ is noetherian.

 By Zariski's main theorem there is a finite ring homomorphism $A \to B'$, an $A$-algebra map $\beta: B' \to B$, and an element $b' \in B'$ with $\beta(b') \notin \mathfrak{q}$ such that $B'_{b'} \to B_{\beta(b')}$ is an isomorphism.
 Thus we may assume that $A \to B$ is finite and \'{e}tale at~$\mathfrak{q}$.

 By~\cref{unramified_etale} the valuation ring $\cO_w \subseteq k(y)$ associated with~$w$ is a local ring of an \'etale $\cO_v$-algebra, where~$\cO_v \subseteq k(x)$ is the valuation ring associated with~$v$.
 Hence, there are polynomials $\bar{f},\bar{g} \in \mathcal{O}_v[T]$ with~$\bar{f}$ monic and and
 \begin{equation} \label{f'unit}
 \bar{f}' \in \left( \mathcal{O}_v[T]/(\bar{f}) \right)_{\bar{g}}^{\times}
 \end{equation}
  such that~$\mathcal{O}_w$ is isomorphic over~$\mathcal{O}_v$ to a local ring of $\left( \mathcal{O}_v[T]/(\bar{f}) \right)_{\bar{g}}$.
 Then $v_1(\bar{f}(T)) \le 1$, $v_1(\bar{g}(T)) \le 1$, $w(\bar{g}) = 1$, and the image~$\beta \in \mathcal{O}_w$ of~$T$ generates the field extension $k(\mathfrak{q})|k(\mathfrak{p})$.

 Write
 \begin{equation} \label{decompositionofp}
 B \otimes_A k(\mathfrak{p}) = \prod_{i=1}^n B_i
 \end{equation}
 with local, Artinian rings~$B_i$ such that~$\mathfrak{q}$ corresponds to the maximal ideal of~$B_1$, i.e. $B_1 = B_{\mathfrak{q}}/\mathfrak{p}B_{\mathfrak{q}} = k(\mathfrak{q})$.
 Denote by $\mathfrak{q}_2,\ldots,\mathfrak{q}_n$ the prime ideals of~$B$ corresponding to the maximal ideals of $B_2,\ldots,B_n$, respectively.
 Consider the element
 $$
 \bar{b}  = (\beta,0,\ldots,0) \in \prod_{i=1}^n B_i = B \otimes_A k(\mathfrak{p}).
 $$
 There is~$\lambda \in A$ whose residue class $\bar{\lambda} \in k(\mathfrak{p})$ is non-zero such that $\bar{\lambda}\bar{b}$ lies in the image of~$B$.
 After replacing~$A$ by~$A_{\lambda}$, we may assume that~$\lambda \in A^{\times}$.
 We can thus lift~$\bar{b}$ to an element $b \in B$.

 Let~$I$ be the kernel of the $A$-algebra homomorphism $A[T] \to B$ mapping~$T$ to~$b$.
 Set $B' = A[T]/I$ and denote by~$\mathfrak{q}'$ the preimage of~$\mathfrak{q}$ in~$B'$.
 Then in the same way as in \cite[Tag 00UE]{stacks-project} we obtain~$B'_{\mathfrak{q}'} \cong B_{\mathfrak{q}}$.
 Therefore, we may replace~$B$ by~$B'$ and henceforth assume that
 $$
 B = A[T]/I.
 $$

 The image~$\bar{I}$ of~$I$ in~$k(\mathfrak{p})[T]$ is a principal ideal generated by a monic polynomial~$\bar{h}$.
 According to the decomposition~(\ref{decompositionofp}) we obtain a decomposition of~$\bar{h}$ into monic irreducible factors:
 $$
 \bar{h} = \bar{h}_1 \cdot \bar{h}_2^{e_2} \cdot \ldots \cdot \bar{h}_n^{e_n}.
 $$
 In particular, $\bar{h}_1 = \bar{f}$, which is a separable polynomial.

 Possibly replacing~$A$ by~$A_{\lambda}$ for $\lambda \in A$ as before we can lift~$\bar{h}$ to a monic polynomial $f \in I$.
 Similarly, by (\ref{f'unit}), we can lift some power of $\bar{g} \in k(\mathfrak{p})[T]$ to a polynomial $g \in A[T]$ of the form $g = a_1 f + a_2 f'$ for some $a_1, a_2 \in A[T]$.
 We obtain a surjection
 $$
 \varphi:A[T]/(f) \to B = A[T]/I
 $$
 mapping $g$ to an element~$b$ of~$B \backslash \mathfrak{q}$ with $w(b) = 1$.

 Since $A \to B$ is \'{e}tale at~$\mathfrak{q}$, there is~$b' \in B \backslash \mathfrak{q}$ such that $A \to B_{bb'}$ is \'{e}tale.
 We can find $a' \in A$ such that $v(a') = w(b')$ as $w|v$ is unramified.
 Upon replacing~$A$ by~$A_{a'}$ we may assume that $a' \in A^{\times}$.
 Then $w(bb'/a) = 1$.
 Choose a preimage~$g'$ under~$\varphi$ of $bb'/a'$.
 Then~$\varphi$ induces a  surjection
 $$
 \varphi_{g'} :\left( A[T]/(f) \right)_{g'} \longrightarrow B_{\varphi(g')} = B_{bb'/a'},
 $$
 Since both rings are \'etale over~$A$, $\varphi_{g'}$ is moreover \'etale, hence a localization.
 Modifying~$g'$ further in the same way as above we achieve that~$\varphi_{g'}$ is an isomorphism.
\end{proof}

\begin{corollary} \label{adicstandardetale}
 Let $\varphi:\cY \to \cX$ be an \'etale morphism of adic spaces and $y \in \cY$ a point where~$\varphi$ is strongly \'etale.
 Then there exist an affinoid open neighborhood $\Spa(A,A^+)$ of $x := \varphi(y)$, an affinoid open neighborhood~$\cV$ of~$y$, and $f,g \in A[T]$ with $f = T^n + f_{n-1} T^{n-1} + \ldots + f_0$ monic and $f'$ a unit in
 $$
 B = \left( A[T]/(f) \right)_g
 $$
  such that $|f_i(x)| \le 1$, $|g(y)| = 1$ and $\cV$ is $\cX$-isomorphic to $\Spa(B,B^+)$ where~$B^+$ is the integral closure of an open subring of~$B$ which is algebraically of finite type over~$A^+$.
\end{corollary}

\begin{proof}
 We may assume that $\cX = \Spa(R,R^+)$ and $\cY = \Spa(S,S^+)$ are affinoid.
 By \cite{Hu96}, Corollary~1.7.3~iii) \'etale morphisms are locally of algebraically finite type.
 More precisely, for every \'etale morphism $\cZ \to \Spa(R,R^+)$ of affinoid adic spaces there is an \'etale ring map $R \to C$ of finite type and a ring of integral elements $C^+ \subseteq C$ which is the integral closure of a subring of~$C$ of finite type over~$C^+$ such that $\cZ \cong \Spa(S,S^+)$ over $(R,R^+)$.
 Hence, we may assume that $(R,R^+) \to (S,S^+)$ is of algebraically finite type and $R \to S$ is \'etale (in the algebraic sense).
 Denote by $x$ the image point of~$y$ in~$\cX$.
 By \cref{standardetale} there exist an affine open neighborhood $\Spec A$ of~$\supp x \in \Spec R$ and $f,g \in A[T]$ with $f = T^n + f_{n-1} T^{n-1} + \ldots f_0$ monic and $f'$ a unit in
 $$
 B = \left( A[T]/(f) \right)_g
 $$
 such that $\Spec B$ is isomorphic over~$A$ to an open neighborhood of $\supp y$, $|f_i(x)| \le 1$ and $|g(y)| = 1$.

 Set $\cU = \Spa(R,R^+) \times_{\Spec R} \Spec A$.
 This is an open subspace of $\cX = \Spa(R,R^+)$.
 By construction of the fiber product (see \cite{Hu94}, Proposition~3.8),~$\cU$ is glued together from affinoid adic spaces of the form $\Spa(A,A_i^+)$ for $i \in \N$ and where~$A_i^+$ is the integral closure in~$A$ of a finite type $R^+$-subalgebra of~$A$.
 Choose $i \in \N$ such that $x \in \Spa(A,A_i^+)$ and set $A^+ := A_i^+$.
 Similarly, we find an open affinoid neighborhood of~$y$ in $\cV = \Spa(A,A^+) \times_{\Spec A} \Spec B$ of the form $\Spa(B,B^+)$ such that $B^+$ is the integral closure in~$B$ of a finite type $A^+$-subalgebra of~$B$.
 This finishes the proof.
\end{proof}

\begin{corollary} \label{unramified}
 Let $\varphi:\cY \to \cX$ be an \'{e}tale morphism of adic spaces.
 The subset of $\cY$ where $\varphi$ is strongly \'etale is open.
\end{corollary}

\begin{proof}
 Let~$y\in \cY$ be a point where~$\varphi$ is strongly \'etale and set $x = \varphi(y)$.
 By \cref{adicstandardetale} we may assume that $\cX = \Spa(A,A^+)$ and $\cY = \Spa(B,B^+)$ as in the statement of the corollary.
 Then $\varphi$ is strongly \'etale at any point $y' \in \cY$ with $|f_i(\varphi(y'))| \le 1$ and $|g(y')| = 1$.
 Indeed, set $x' = \varphi(y')$ and denote by~$\bar{f}$ and~$\bar{g}$ the residue classes of~$f$ and~$g$ in~$k(x')[T]$.
 We obtain an \'{e}tale ring extension $k(x')^+ \to \left(k(x')^+[T]/(\bar{f})\right)_{\bar{g}}$.
 Since $|g(y')| = 1$, $k(y')^+$ is a localization of $\left(k(x')^+[T]/(\bar{f})\right)_{\bar{g}}$.
 The subset $\{y' \in \cY \mid |f_i(y')| \le 1~\forall i,|g(y')| =1\}$ of $\cY$ is open and thus we are done.
\end{proof}

\begin{corollary} \label{tamelocusopen}
 Let $\varphi:\cY \to \cX$ be an \'{e}tale morphism of adic spaces.
 The subset of $\cY$ where $\varphi$ is tame, is open.
\end{corollary}

\begin{proof}
 We may assume that $\cX = \Spa(A,A^+)$ and $\cY = \Spa(B,B^+)$ are affinoid.
 Let $y \in \cY$ be a point where~$\varphi$ is tame and set $x := \varphi(y)$.
 By a generalization of Abhyankar's lemma (see \cite{GR03}, Corollary~6.2.14) the extension of the strict henselizations $k(y)^{\sh}|k(x)^{\sh}$ takes the following form.
 There are non-zero elements $\bar{a}_1,\ldots,\bar{a}_n \in k(x)^{\sh}$ and integers $m_1,\ldots,m_n$ prime to the residue characteristic of $k(x)^+$ such that
 \[
  k(y)^{\sh} = k(x)^{\sh}[\bar{a}_1^{1/m_1},\ldots,\bar{a}_n^{1/m_n}].
 \]
 We can replace the $\bar{a}_i$'s by any other elements of $k(x)^{\sh}$ as long as their valuation stays the same.
 Therefore, we may assume that $\bar{a}_i \in k(x)$.
 Let~$m$ be the least common multiple of the $m_i$'s.
 Then any lift to $k(x)[\mu_m,\sqrt[m]{\bar{a}_1},\ldots,\sqrt[m]{\bar{a}_n}]$ of the valuation corresponding to~$x$ is unramified in
 $$
 k(x)[\mu_m,\sqrt[m]{\bar{a}_1},\ldots,\sqrt[m]{\bar{a}_n}] \otimes_{k(x)} k(y)~\big|~k(x)[\mu_m\sqrt[m]{\bar{a}_1},\ldots,\sqrt[m]{\bar{a}_n}].
 $$
 We may choose the~$\bar{a}_i$ as images of some~$a_i \in A$.
 Replacing~$\Spa (A,A^+)$ by a rational open neighborhood of~$x$ we may further assume that~$a_i \in A^{\times}$ and that~$m$ is invertible on~$\Spec A^+$.
 The ring homomorphism
 \[
 A \to A' := A[T_0,T_1,\ldots,T_n] \big/ (\Phi_m(T_0),T_1^m - a_1,\ldots,T_n^m - a_n),
 \]
 where $\Phi_m$ denotes the $m$-th cyclotomic polynomial, is finite \'etale.
 Set $\cX' :=\Spa (A',A'^+)$ where~$A'^+$ is the integral closure of~$A^+$ in~$A'$.
 Then $\cX' \to \cX$ is tame.
 Moreover,
 \[
 \cY' := \cY \times_{\cX} \cX' \to \cX'
 \]
 is strongly \'etale at any lift of $x$ to~$\cX'$.
 Fix such a lift $x' \in \cX'$.
 We find a point $y' \in \cY'$ lying over $x'$ as well as $y$ (\cite{Hu96}, Corollary~1.2.3 iii) d)).
 Denote by~$\varphi'$ the morphism $\cY' \to \cX'$ and by~$\psi$ the morphism $\cX' \to \cX$.
 By \cref{unramified} there is an open neighborhood $\cV' \subseteq \cY'$ of $y'$ such that $\cV' \to \cX'$ is strongly \'etale.
 Then $\cV' \to \cX$ is tame.
 Since \'etale morphisms are open (\cite{Hu96}, Proposition~1.7.8), the image~$\cV$ of~$\cV'$ in $\cY$ is an open neighborhood of~$y$ and moreover, $\cV \to \cX$ is tame.
\end{proof}

\section{Limits of adic spaces} \label{limits}

In \cite{Hu96}, \S~2.4 Huber defines the notion of a projective limit of adic spaces:
Let~$\mathcal{A}$ be the category of quasi-compact, quasi-separated pseudo-adic spaces with adic morphisms.
We consider a functor~$p$ from a cofiltered category~$I$ to~$\mathcal{A}$ and write~$\cX_i$ for $p(i)$.
Let $c: I \to \mathcal{A}$ be the constant functor to some object~$\cX$ of~$\mathcal{A}$ and
\[
 \varphi : c \to p,	\quad i \mapsto (\varphi_i : \cX \to \cX_i)
\]
a morphism of functors.
We say that~$\cX$ is a projective limit of the~$\cX_i$ and write
\[
 \varphi : \cX \sim \lim_i \cX_i
\]
if the following conditions are satisfied:
\begin{enumerate}
 \item Denote by $\lim_i |\cX_i|$ the projective limit in the category of topological spaces.
       Then the natural mapping
       \[
        \psi: |\cX| \to \lim_i |\cX_i|
       \]
       induced by~$\varphi$ is a homeomorphism.
 \item For every $x \in |\cX|$, there is an affinoid open neighborhood~$\cU$ of~$x$ such that the subring
       \[
        \bigcup_{(i,\cV)} \im(\underline{\varphi}_i^* : \cO_{\underline{\cX}_i}(\cV) \to \cO_{\underline{\cX}}(\cU))
       \]
       of $\cO_{\underline{\cX}}(\cU)$ is dense where the union is over all pairs $(i,\cV)$ with $i \in I$ and~$\cV$ an open subset of~$\underline{\cX}_i$ with $\underline{\varphi}_i(\cU) \subseteq \cV$.
\end{enumerate}

In this situation we have the following proposition (\cite{Hu96}, Proposition~2.4.4):

\begin{proposition} \label{limitetaletopos}
 Let
 \[
  \tilde{\varphi} : \tilde{\cX}_{\et} \times I \to (\tilde{\cX}_{i,\et})_{i \in I}
 \]
 be the morphism of topoi fibered over~$I$ which is induced by the $\tilde{\varphi}_i : \tilde{\cX}_{\et} \to \tilde{\cX}_{i,\et}$.
 Assume that $\varphi : \cX \sim \lim_i \cX_i$.
 Then $(\tilde{\cX}_{\et},\tilde{\varphi})$ is a projective limit of the fibered topos $(\tilde{\cX}_{i,\et})_{i \in I}$.
\end{proposition}

In order to prove this proposition Huber proceeds as follows:
For each $i \in I$ denote by $\cX_{i,\et,f.p.}$ the restricted \'etale site,
 i.e. the site consisting of those objects in $\cX_{i,\et}$ whose structure morphisms are quasi-compact and quasi-separated (\cite{Hu96}, (2.3.12)).
The topos associated with the projective limit site $\zerounderset{\rightarrow}{\cX}$ of the fibered site $(\cX_{i,\et,f.p.})_{i \in I}$
 is isomorphic to the projective limit of the fibered topos $(\tilde{\cX}_{i,\et})_{i \in I}$.
Moreover, $\tilde{\cX}_{\et}$ is isomorphic to the topos associated with the site $\cX_{\et,g}$ which is defined as follows (\cite{Hu96}, Remark~2.3.4~ii)):
The objects are the \'etale morphisms to~$\cX$ and the morphisms $\cY \to \cZ$ are the equivalence classes of $\cX$-morphisms $\cY' \to \cZ$
 where~$\cY'$ is an open subspace of~$\cY$ with~$|\cY'| = |\cY|$ and two morphisms are equivalent if they coincide on an open subspace~$\cV$ of~$\cY$ with $|\cV| = |\cY|$.
There is a natural morphism of sites
\[
 \lambda : \cX_{\et,g} \to \zerounderset{\rightarrow}{\cX}
\]
for which Huber proves that the conditions in the following proposition (\cite{Hu96}, Corollary~A.5) are satisfied:

\begin{proposition} \label{toposisomorphic}
 Let $f: C \to C'$ be a morphism of sites.
 The induced morphism of topoi $\tilde{f}: \tilde{C} \to \tilde{C}'$ is an equivalence if~$f$ satisfies the following conditions.
 \begin{enumerate}[\rm (a)]
  \item In~$C'$ there exist finite projective limits and $f^{-1}$ commutes with these.
  \item Every~$X \in \ob(C)$ has a covering $(X_i \to X)_{i \in I}$ in~$C$ such that every~$X_i \in \ob(C)$ lies in the image of the functor~$f^{-1}$.
  \item A family $(X_i \to X)_{i \in I}$ of morphisms in~$C'$ is a covering in~$C'$ if $(f^{-1}(X_i) \to f^{-1}(X))_{i \in I}$ is a covering in~$C$.
  \item For every $X \in \ob(C)$, $Y \in \ob(C')$ and $(\varphi : X \to f^{-1}(Y)) \in \mor(C)$, there exist a covering $(\psi_i : X_i \to X)$ of~$X$ in~$C$,
         and, for every $i \in I$ a~$Y_i \in \ob(C')$, a $(\tau_i : Y_i \to Y) \in \mor(C')$ and a $(\varphi_i : X_i \to f^{-1}(Y_i)) \in \mor(C)$
         such that, for every $i \in I$ the diagram in~$C$
         \[
          \begin{tikzcd}
           X_i	\ar[r,"\varphi_i"]	\ar[d,"\psi_i"']	& f^{-1}(Y_i)	\ar[d,"f^{-1}(\tau_i)"]	\\
           X	\ar[r,"\varphi"']				& f^{-1}(Y)
          \end{tikzcd}
         \]
         commutes and $\varphi_i : X_i \to f^{-1}(Y_i)$ is an epimorphism and a covering of $f^{-1}(Y_i)$ in~$C$.
 \end{enumerate}
\end{proposition}

We are now going to prove an analogue of \cref{limitetaletopos} for the tame and the strongly \'etale topos:

\begin{proposition}
 In the situation of \cref{limitetaletopos} the topos $(\tilde{\cX}_{\set},\tilde{\varphi})$ is a projective limit of the fibered topos $(\tilde{\cX}_{i,\set})_{i \in I}$
  and $(\tilde{\cX}_t,\tilde{\varphi})$ is a projective limit of the fibered topos $(\tilde{\cX}_{i,t})_{i \in I}$.
\end{proposition}

\begin{proof}
 We check that the strongly \'etale and tame analogues~$\lambda_{\set}$ and~$\lambda_t$ of $\lambda$ satisfy the conditions of \cref{toposisomorphic}:

 (a) is true because $\cX_{\set}$ and~$\cX_t$ have fiber products and a terminal object.

 (b) Let $\cZ \to \cX$ be strongly \'etale.
     In particular, it is \'etale.
     In the proof of \cref{limitetaletopos} Huber constructs an open covering $\cZ = \bigcup_{j \in J} \cZ_j$
      such that~$\cZ_j$ is $\cX$-isomorphic to an open subspace of $\cY_i \times_{\cX_i} \cX$ for some $i \in I$ (depending on~$j$) and $\psi_i: \cY_i \to \cX_i$ in $\cX_{i,\et,f.p.}$
      with $|\cZ_j| = |\cY_i \times_{\cX_i} \cX|$.
     We have to find $k \to i$ in~$I$ such that
     \[
      \psi_k : \cY_k := \cY_i \times_{\cX_i} \cX_k \to \cX_k
     \]
     is strongly \'etale.
     By \cref{unramified} for every~$k \to i$ the set of points in $|\cY_k|$ where $\psi_k$ is \emph{not} strongly \'etale is closed, hence compact in the constructible topology
      (note that~$|\cY_k|$ is locally pro-constructible by the definition of a pseudo-adic space and quasi-compact as $|\cX_k|$ is quasi-compact and $\cY_k \to \cX_k$ is qcqs).
     Therefore, its image~$D_k$ in~$|\cX_k|$ is compact in the constructible topology of~$|\cX_k|$.
     We write~$D^c_k$ for the set~$D_k$ equipped with the constructible topology.
     For $a:k \to k'$ denote by
     $$
     u_a: \underline{\cX}_k \to \underline{\cX}_{k'}
     $$
     the transition map and by
     $$
     u_k : \underline{\cX} \to \underline{\cX}_k
     $$
     the natural projection.
     Then~$u_a$ and~$u_k$ are continuous for the constructible topology by~\cite{Hu93}, Proposition~3.8~(iv).
     Since the property of being strongly \'etale is stable under base change,
     $$
     u_a(D_k) \subseteq D_{k'}.
     $$
     Furthermore, the assumption that $\cZ \to \cX$ is strongly \'etale implies that
     $$
     \lim_{k \to i} D^c_k = \bigcap_{k \to i} u_k^{-1}(D^c_k) = \varnothing.
     $$
     Since the projective limit of nonempty compact spaces is nonempty, there is $k \to i$ such that $D^c_k = \varnothing$.
     In other words $\cY_k \to \cX_k$ is strongly \'etale.
     The proof for the tame topology is the same except for using \cref{tamelocusopen} instead of \cref{unramified}.

 (c) is obvious by the corresponding statement for the \'etale site and the proof for~(d) is the same as for the \'etale site.
\end{proof}

\begin{corollary} \label{cohomologyoflimit}
 In the situation of \cref{limitetaletopos} assume that~$i_0 \in I$ is a final object.
 Let~$\mathcal{F}_0$ be a sheaf of abelian groups on $\cX_{i_0,\set}$.
 For $i \in I$ denote by~$\mathcal{F}_i$ its pullback to~$\cX_{i,\set}$ and by~$\mathcal{F}$ its pullback to~$\cX_{\set}$.
 Then the natural map
 $$
 \colim_{i \in I} H^p(\cX_{i,\set},\mathcal{F}_i) \longrightarrow H^p(\cX_{\set},\mathcal{F})
 $$
 is an isomorphism for all $p \geq 0$.
 Moreover, the analogous statement holds for the tame site.
\end{corollary}

\begin{corollary}
 Let~$\cS$ be an adic space and $\tau \in \{\et,t,\set\}$.
 In the situation of \cref{limitetaletopos} assume that~$\cX_i$ are adic spaces over~$\cS$ with compatible quasi-compact quasi-separated structure morphisms $g_i:\cX_i \to \cS$.
 We write $g:\cX \to \cS$ for the resulting morphism.
 For every $i \in I$ let~$\mathcal{F}_i$ be an abelian sheaf on~$(\cX_i)_{\tau}$ and for all $\alpha:i\to j$ let $\varphi_{\alpha} :\alpha^*\mathcal{F}_j \to \mathcal{F}_i$ be compatible transition morphisms.
 Denote by~$\mathcal{F}$ the sheaf $\colim_I \varphi_i^*\mathcal{F}_i$.
 Then for all $p \ge 0$
 \[
  R^pg_*\mathcal{F} \cong \colim_I R^pg_{i,*}\mathcal{F}_i.
 \]
\end{corollary}

\section{Localizations}

\subsection{Local adic spaces}

\begin{definition}
 \item A Huber pair $(A,A^+)$ is \emph{local} if $A$ and~$A^+$ are local, $A^+$ is the preimage of a valuation ring~$k_A^+$ of the residue field~$k_A$ of~$A$, and the maximal ideal~$\m^+$ of~$A^+$ is open and bounded.
 \item A homomorphism of local Huber pairs $(A,A^+) \to (B,B^+)$ is \emph{local} if $A \to B$ and $A^+ \to B^+$ are local ring homomorphisms.
\end{definition}

For a local Huber pair $(A,A^+)$ the maximal ideal~$\m_A$ of~$A$ is contained in~$A^+$ and is indeed a prime ideal of~$A^+$.
Moreover, $A^+/\m_A$ is the valuation ring~$k_A^+$ mentioned in the definition and $A^+_{\m_A} = A$.

\begin{remark}
 The concept of a local Huber pair is closely related to the theory of $I$-valuative rings introduced in \cite{FuKa18}, Chapter~0, \S~8.7.
 A ring~$B$ is $I$-valuative for an ideal $I \subseteq B$ if the $I$-adic topology on~$B$ can be defined by a finitely generated ideal of definition and in addition every finitely generated ideal containing a power of~$I$ is invertible.
 Without loss of generality one can always replace~$I$ by a finitely generated ideal of definition.
 If~$B$ is local and $I$-valuative for a finitely generated ideal~$I$, this automatically implies that~$I$ is principal.
 
 Recall from \cite{Hu96}, Definition~1.1.4, that a valuation ring~$\cO$ with quotient field~$K$ is microbial if it has a prime ideal of height~$1$.
 Equivalently, $K$ equipped with the valuation topology has a topologically nilpotent unit $u$.
 In this case~$\cO$ is $u$-valuative.
 Let us say that a local Huber pair $(A,A^+)$ is \emph{microbial} if the valuation ring $k^+ = A^+/\m_A$ is microbial.
 
 For a microbial Huber pair $(A,A^+)$ and $a \in A^+$ such that its residue class in $k$ is topologically nilpotent, the ring~$A^+$ is $(a)$-valuative.
 Conversely, suppose that we start with an $I$-valuative local ring~$B$ with~$I$ finitely generated, hence principal, $I = (b)$.
 Then $\p = \bigcap_n I^n$ is a prime ideal of~$B$ and $(B_{\p},B)$ is a microbial local Huber pair.
 These constructions are inverse to each other up to replacing~$I$ by some other finitely generated ideal that defines the same topology (compare \cite{FuKa18}, Chapter~0, Theorem~8.7.8).
\end{remark}

\begin{lemma} \label{ideals_local_Huber_pair}
 Let $(A,A^+)$ be a local Huber pair.
 For every ideal~$\mathfrak{a}$ of~$A^+$ we either have $\mathfrak{a} \subseteq \m_A$ or $\m_A \subseteq \mathfrak{a}$.
 In case $\m_A \subsetneq \mathfrak{a}$ and~$\mathfrak{a}$ is finitely generated, it is even principal.
\end{lemma}

\begin{proof}
 Suppose $\mathfrak{a} \nsubseteq \m_A$.
 So there is $a \in \mathfrak{a}$ with $a \notin \m_A$.
 In particular, $a$ is a unit in~$A$.
 Let~$m \in \m_A$.
 Then $m':= a^{-1}m \in \m_A$ as~$\m_A$ is an ideal in~$A$.
 This implies $m = am' \in \mathfrak{a}$.
 
 For the second assertion let $x \in \Spa(A,A^+)$ be the point corresponding to the valuation ring~$k_A^+$ of~$k_A$.
 From a finite set of generators of~$\mathfrak{a}$ pick one, let us call it $a$, with maximal valuation $|a(x)|$.
 Since $\m_A \subsetneq \mathfrak{a}$, $a$ is not contained in $\m_A$, i.e., $|a(x)| > 0$ and~$a$ is a unit in~$A$.
 For another element~$b$ of~$\mathfrak{a}$, we have
 \[
  |a^{-1}b(x)| \le |a^{-1}a(x)| = 1.
 \]
 Hence, $a':=a^{-1}b$ is an element of~$A^+$ and $b = a'a \in aA^+$.
\end{proof}

\begin{lemma} \label{ring_of_definition}
 For a local Huber pair $(A,A^+)$, $A^+$ is a ring of definition.
\end{lemma}

\begin{proof}
 Let~$A_0$ be a ring of definition and $I \subseteq A_0$ an ideal of definition.
 Then $I \subseteq \m^+$ as~$\m^+$ is open.
 By the boundedness of~$\m^+$ and as $IA^+ \subseteq \m^+$, we can find $n \in \N$ with $(IA^+)^n \subseteq I$.
 In total we have
 \[
  (IA^+)^n \subseteq I \subseteq IA^+,
 \]
 i.e., the topology of~$A^+$ is $IA^+$-adic.
\end{proof}

\begin{lemma} \label{completion_local}
 Let $(A,A^+)$ be a local Huber pair.
 Then also its completion $(\hat{A},\hat{A}^+)$ is local.
\end{lemma}

\begin{proof}
 Suppose first that the maximal ideal~$\m_A$ of~$A$ is not open.
 Let~$I \subseteq A^+$ be an ideal of definition ($A^+$ is a ring of definition by \cref{ring_of_definition}).
 Then \cref{ideals_local_Huber_pair} implies $\m_A \subsetneq I^n$ for all $n \in \N$.
 The completion of $(A,A^+)$ thus factors through $(k_A,k_A^+)$, where the topology coincides with the valuation topology.
 The completion of a valuation ring is again a valuation ring, whence the assertion in this case.
 
 If~$\m_A$ is open, $I \subseteq \m_A$.
 In this case~$A$ carries the $IA$-adic topology because $IA \subseteq \m_A \subseteq A^+$ and thus
 \[
  (IA)^2 \subseteq IA^+ \subseteq IA.
 \]
 The completion~$\hat{A}$ of~$A$ is again local and~$\hat{A}^+$ contains the maximal ideal~$\m_{\hat{A}}$ of~$\hat{A}$ since this is true at every finite level $A^+/(IA)^n \subseteq A/IA$.
 Moreover,
 \[
  \hat{A}^+/\m_{\hat{A}} = A^+/\m_A = k_A^+
 \]
 is a valuation ring.
 Finally, $\m_{\hat{A}}$ is open and bounded by construction.
 \end{proof}

\begin{lemma} \label{local_strongly_etale}
 Let $(A,A^+)$ be a local Huber pair and $(A,A^+) \to (B,B^+)$ \'etale such that $B = B^+ \otimes_{A^+} A$.
 If $\Spa(B,B^+) \to \Spa(A,A^+)$ is strongly \'etale, then $A^+ \to B^+$ is \'etale.
\end{lemma}

\begin{proof}
 By assumption $A \to B$ is \'etale and by \cref{comparenotionsofetale} also $A^+/\m \to B^+/\m B^+$ is \'etale.
 In particular, both morphisms are flat and of finite presentation and thus \cite{Tem11}, Lemma~2.3.1 implies that $A^+ \to B^+$ is flat and of finite presentation
  (the flatness is a consequence of the flattening result by Raynaud and Gruson \cite{RG71}, Theorem~5.2.2).
 Let us show that $A^+ \to B^+$ is unramified, i.e. that $\Omega^1_{B^+/A^+} = 0$.
 Since $A^+/\m \to B^+/\m B^+$ is unramified, $\Omega^1_{B^+/A^+} \otimes_{A^+} A^+/\m = 0$.
 It remains to show that $\m \Omega^1_{B^+/A^+} = 0$.
 But the isomorphism $\m \cong \m \otimes_{A^+} A$ induces an isomorphism
 \[
  \m \Omega^1_{B^+/A^+} \cong \m(\Omega_{B^+/A^+} \otimes_{A^+} A)
 \]
 and $\Omega^1_{B^+/A^+} \otimes_{A^+} A = 0$ as $A \to B$ is unramified.
\end{proof}

\begin{remark}
 In case $(A,A^+)$ is microbial, the assertion of \cref{local_strongly_etale} also follows from \cite{FuKa18}, Chapter~0, Proposition~8.7.12.
\end{remark}

\begin{lemma} \label{localaffinoid}
 An adic space~$\cX$ is the spectrum of a local Huber pair if and only if~$\cX$ has a unique closed point~$x$ and any other point specializes to~$x$.
\end{lemma}

\begin{proof}
 Suppose that every point of~$\cX$ specializes to~$x$.
 Then every affinoid open neighborhood of~$x$ must contain all points of~$\cX$.
 Hence $\cX = \Spa(A,A^+)$ for some Huber pair $(A,A^+)$, which we can moreover assume to be complete.
 Let~$\m \subseteq A$ denote the support of~$x$.
 Suppose there is a maximal ideal $\m' \subseteq A$ different from~$\m$.
 By \cite{Hu94}, Lemma~1.4 there is a point $y \in \Spa(A,A^+)$ whose support is~$\m'$.
 But~$y$ does not specialize to~$x$, hence~$A$ is local with maximal ideal~$\m$.

 Let~$a$ be an element of~$A$ which is not contained in~$A^+$.
 We want to show that~$a$ is a unit in~$A$ and $1/a \in A^+$.
 Then we are done by \cite{KneZha}, Theorem~2.5.
 Let~$A_a^+$ denote the integral closure of $A^+[1/a]$ in~$A_a$.
 Then
 \[
 R(\frac{1}{a}) = \Spa(A_a,A_a^+)
 \]
 is a rational subset of~$\cX$.
 Since $a \notin A^+$, there is $y \in \cX$ with $|a(y)| > 1$.
 But~$y$ specializes to~$x$ and thus $|a(x)| > 1$.
 This implies $x \in \Spa(A_a,A_a^+)$, so $\Spa(A_a,A_a^+) = \cX$.
 Moreover, $a \notin \m = \supp x$, i.e., $a \in A^{\times}$.
 For any point $z \in X$ we have $|(1/a)(z)| \le 1$.
 But
 \[
  A^+ = \{b \in A | |b(z)| \le 1~\forall z \in \cX\},
 \]
 whence $1/a \in A^+$.
\end{proof}

In view of the lemma we say that a pseudo-adic space~$\cX$ is \emph{local} if~$\underline{\cX}$ is the adic spectrum of a local Huber pair and the closed point of~$\underline{\cX}$ is contained in~$|\cX|$.

\subsection{Henselian adic spaces}

\begin{definition}
 A Huber pair $(A,A^+)$ is \emph{henselian} if it is local and~$A^+$ is henselian.
\end{definition}

In this subsection, we collect some properties of henselian Huber pairs.

\begin{proposition} \label{local_henselian}
 A local Huber pair $(A,A^+)$ is henselian if and only if both~$A$ and $k_A^+ = A^+/\m_A$ are henselian.
\end{proposition}

\begin{proof}
 Without loss of generality we may assume that the topology of $A$ is discrete.
 Suppose~$A^+$ is henselian.
 Being a quotient of~$A^+$, it is clear that $k_A^+$ is henselian.
 In order to show that~$A$ is henselian, consider an \'etale homomorphism $\varphi: A \to B$ together with a maximal ideal $\m_B$ of~$B$ over~$\m_A$ with trivial residue field extension.
 We need to show that~$\varphi$ has a section.
 Upon localizing~$B$ we may assume that~$\m_B$ is the only prime ideal of~$B$ lying over~$\m_A$.
 
 Let~$B^+$ be the preimage of~$k_A^+$ under the projection
 \[
  B \to B/\m_B \cong A/\m_A = k_A.
 \]
 Then $(B,B^+)$ is a Huber pair and the valuation ring $k_A^+ \hookrightarrow B/\m_B$ determines a closed point $y \in \Spa(B,B^+)$ lying over the closed point $x$ of $\Spa(A,A^+)$.
 Consider the multiplicative subset $S = A^+ \setminus \m_A$ of~$A^+$.
 Since $(A,A^+)$ is local, we know that $A = S^{-1}A^+$ and from the definition of~$B^+$ we conclude that $B = S^{-1}B^+$.
 
 Let us show that~$B^+$ is a finitely generated $A^+$-algebra.
 Let $b_1,\ldots,b_n$ be generators of~$B$ as an $A$-algebra.
 We claim that $B^+ = A^+[b_1,\ldots,b_n]$.
 Since $B = S^{-1}B^+$, we may choose the $b_i$ to be contained in~$B^+$.
 Every element of~$B^+$ can be written as a sum of an element of~$A^+$ and an element of~$\m_B$.
 But $\m_B = \m_A B$ as $A \to B$ is \'etale and~$\m_B$ is the only prime ideal over~$\m_A$.
 Therefore, we are reduced to showing that every element of the form $mb$ with $m \in \m_A$ and $b \in B$ is contained in $A^+[b_1,\ldots,b_n]$.
 By assumption $b = P(b_1,\ldots,b_n)$ for some polynomial $P \in A[T_1,\ldots,T_n]$.
 Then $mb = mP(b_1,\ldots,b_n)$ and the coefficients of $mP$ are contained in $\m_A \subseteq A^+$, which is what we needed.

 Next we check that for every point~$y'$ of $\Spa(B,B^+)$, the characteristic soubgroup $c\Gamma_{x'} \subseteq \Gamma_{x'}$ of the image point $x' \in \Spa(A,A^+)$ generates $c\Gamma_{y'}$ as a convex subgroup of $\Gamma_{y'}$.
 Let $b \in B$ with $|b(y')| > 1$.
 We write $b = b^+/s$ with $b^+ \in B^+$ and $s \in S$.
 Then
 \[
  1 < |b(y')| = |(1/s)(y')| \cdot |b^+(y')| \le |(1/s)(y')|= |(1/s)(x')| \in c\Gamma_{x'},
 \]
 which proves the claim.
 
 Let us now show that $\Spa(B,B^+) \to \Spa(A,A^+)$ is strongly \'etale.
 We already know it is \'etale and by \cref{unramified} it is enough to show it is strongly \'etale at the closed points.
 Let $y' \in \Spa(B,B^+)$ be a closed point mapping to $x' \in \Spa(A,A^+)$.
 Then $c\Gamma_{y'} = \Gamma_{y'}$.
 By the last paragraph, $c\Gamma_{x'} = \Gamma_{x'}$, i.e., there are no horizontal specializations.
 Since $(A,A^+)$ is local, this means that either~$x'$ is a trivial valuation or the support of~$x'$ equals~$\m_A$.
 In the first case the strongly \'etale condition at~$y'$ is automatically satisfied.
 In the second case it follows as the residue field extension $k(\m_B)|k(\m_A)$ is trivial.
 
 We can thus apply \cref{local_strongly_etale} to conclude that $\varphi^+: A^+ \to B^+$ is \'etale (note that $B = S^{-1}B^+ = B^+ \otimes_{A^+} A$).
 Let~$\m^+_B$ be the kernel of the surjection
 \[
  B^+ \to B^+/\m_B \cong k_A^+ \to A^+/\m_A^+.
 \]
 It is a maximal ideal of~$B^+$ lying over~$\m_A^+$ whose residue field extension is trivial.
 Since~$A^+$ is henselian, $\varphi^+$ has a section $\sigma^+ : B^+ \to A^+$.
 But then $\sigma = \sigma^+ \otimes_{B^+} B$ is a section to~$\varphi$.
 
 Now assume that~$A$ and~$k_A^+$ are henselian.
 Let $\varphi^+: A^+ \to B^+$ be an \'etale homomorphism together with a maximal ideal~$\m_B^+$ of~$B^+$ over~$\m_A^+$ with trivial residue extension.
 The base change of~$\varphi^+$ to $k_A^+$ has a section~$\sigma_k$ as~$k_A^+$ is henselian.
 It maps the generic point of $\Spec k_A^+$ to some point $x \in \Spec B^+$ with trivial residue field extension.
 Let
 \[
  \varphi: A \to B = B^+ \otimes_{A^+} A
 \]
 be the base change of~$\varphi^+$.
 Then $\Spec B \subseteq \Spec B^+$ and $x \in \Spec B$ corresponds to a maximal ideal~$\m_B$ of~$B$ lying over~$\m_A$ with trivial residue extension.
 Since~$A$ is henselian, $\varphi$ admits a section $\sigma: B \to A$ such that $\sigma^{-1}(\m_A) = \m_B$.
 Let us check that~$\sigma$ maps~$B^+$ to~$A^+$.
 Consider the following cocartesian diagram with surjective rows.
 \[
  \begin{tikzcd}
   B	\ar[r]							& B \otimes_{A^+} k_{A^+} = B/\m_A B		\ar[r]						& B/\m_B \cong k_A	\\
   B^+	\ar[r]	\ar[u,hookrightarrow]	& B^+ \otimes_{A^+} k_{A^+} = B^+/\m_A B	\ar[r,"\sigma_k"]	\ar[u,hookrightarrow]	& k_A^+	\ar[u,hookrightarrow].
  \end{tikzcd}
 \]
 We learn from this diagram that the kernel of the lower horizontal map is $\m_B \cap B^+$.
 Moreover every element of~$B^+$ can be written in the form $m + a^+$ with $m \in \m_B$ and $a^+ \in A^+$.
 In particular, this shows $\sigma(B^+) \subseteq A^+$.
 The resulting homomorphism $\sigma^+ : B^+ \to A^+$ is a section to~$\varphi^+$.
\end{proof}

\begin{remark}
 \begin{enumerate}
  \item The first half of the proof of \cref{local_henselian} can be simplified considerably by resorting to the results of \cref{section_Laurent}.
		We start with an \'etale homomorphism $A \to B$ and $\m_B \in \Spec B$ over~$\m_A$ with trivial residue field extension.
		Defining~$B^+$ to be the integral closure of~$A^+$ in~$B$ we obtain an \'etale morphism $\Spa(B,B^+) \to \Spa(A,A^+)$ and a point $y \in \Spa(B,B^+)$ with support~$\m_B$ lying over the closed point $x$ of $\Spa(A,A^+)$ with trivial residue field extension.
		By \cref{unramified} and \cref{basisofneighborhoods}, there is an affinoid open neighborhood $\Spa(C,C^+) \subseteq \Spa(B,B^+)$ such that $(A,A^+) \to (C,C^+)$ is Cartesian (i.e., $C \cong C^+ \otimes_{A^+} A$) and strongly \'etale.
		Then we proceed as in the proof of the proposition concluding that $A^+ \to C^+$ is \'etale and constructing a section.
  \item In case the valuation ring $k^+ = A^+/\m_A$ has finite rank~$n$, the results of \cref{local_henselian} can also be deduced from \cite{FuKa18}, Chapter~0, Proposition~8.7.13.
		Assume that~$A$ and~$k^+$ are henselian.
		For $i=0,\ldots,n$ let~$\p_i$ be the prime ideal of~$A^+$ which is the preimage of the prime ideal of~$k^+$ of height~$i$.
		We start by choosing an element $a \in \p_1$ that is not contained in $\p_0 = \m_A$.
		Then~$A^+$ is $a$-valuative (see \cite{FuKa18}, Chapter~0, Theorem~8.7.8) and Proposition~8.7.13 in loc.\,cit. implies that $A^+_{\p_1}$ is henselian.
		If we already know that $A^+_{\p_{i-1}}$ is henselian, we can deduce by the same argument and using an element $a \in \p_i \setminus \p_{i-1}$ that $A^+_{\p_i}$ is henselian.
		So by induction we obtain that $A^+$ is henselian.
		The converse direction is similar.
 \end{enumerate}

\end{remark}

\begin{lemma} \label{max_id_integral}
 Let $A$ be a henselian local ring and $A \to B$ an integral ring homomorphism of local rings.
 Then the maximal ideal~$\m_B$ of~$B$ is integral over~$\m_A$, i.e., for every element $b \in \m_B$, there is a monic polynomial $P \in A[T]$ whose non-leading coefficients are in~$\m_A$ such that $P(b) = 0$.
\end{lemma}

\begin{proof}
 Let $b \in \m_B$ and
 \[
  P(T) = T^n + a_{n-1}T^{n-1} + \ldots + a_0
 \]
 a monic polynomial in $A[T]$ of minimal degree such that $P(b) = 0$.
 Assume that one of the coefficients $a_i$ is not contained in~$\m_A$ and let~$j$ be minimal with $a_j \notin \m_A$.
 Because $P(b) = 0$ and $b \in \m_B$, we have $a_0 \in \m_B \cap A = \m_A$, whence $j > 0$.
 Over the residue field~$k_A$ of~$A$ we thus obtain a decomposition of the reduction of~$P$ into two coprime factors:
 \[
  \bar{P}(T) = T^j(T^{n-j} + \bar{a}_{n-1}T^{n-j-1} + \ldots + \bar{a_j}).
 \]
 By Hensel's lemma, it lifts to a decomposition of~$P$ into monic polynomials:
 \[
  P(T) = P_1(T)P_2(T).
 \]
 The constant term of~$P_2$ is congruent to~$a_j$ modulo~$\m_A$, hence a unit in~$A$.
 Since $b \in \m_B$, this implies that $P_2(b)$ is a unit in~$B$.
 Combining this information with $P(b) = 0$, we conclude that $P_1(b) = 0$, in contradiction to the minimality of the degree of~$P$.
\end{proof}

\begin{lemma} \label{factorization_finite}
 Let $R \to S$ be a finite ring homomorphism and assume that $\Spec R$ and $\Spec S$ are connected.
 Then $R \to S$ factors as
 \[
  R \to T \to S
 \]
 with $\Spec T$ connected, $R \to T$ finite and faithfully flat, and $T \to S$ surjective. 
\end{lemma}

\begin{proof}
 Let $s_1,\ldots,s_n$ be generators of~$S$ as an $R$-algebra.
 For each $i = 1,\ldots,n$ we choose a monic polynomial $P_i \in R[X]$ with image $\bar{P}_i \in S[X]$ such that $\bar{P}_i(s_i) = 0$.
 Then we obtain a factorization
 \[
  R \to T':=R[X_1,\ldots,X_n]/(P_1(X_1),\ldots,P_n(X_n)) \to S.
 \]
 The left hand map is finite and flat and the right hand map is surjective.
 Since $\Spec S$ is connected, the image of $\Spec S \to \Spec T'$ lies entirely in one connected component $X$ of $\Spec T'$.
 
 Let us recall the scheme structure of the connected component~$X$ (compare \cite{Lazard67}, \S~4).
 Let $\p$ be a prime ideal of~$T'$ such that the corresponding point of $\Spec T'$ is contained in~$X$.
 The set~$M$ of idempotents of~$T'$ that are not contained in~$\p$ is partially ordered by divisibility, $e \le e'$ if and only if $e | e'$.
 Then
 \[
  X = \Spec \left( \colim_{e \in M} T'_e \right) = \Spec \left( \colim_{e \in M} T'/(1-e)T' \right),
 \]
 where we have used the canonical identification $T'_e = T'/(1-e)T'$.
 On the one hand we have $T:= \colim_{e \in M} T'_e = M^{-1}T'$, so $T' \to T$ is a localization.
 On the other hand $\colim_{e \in M} T'/(1-e)T' = T'/I$, where~$I$ is the ideal generated by all idempotents $1-e$ with $e \in M$.
 This shows that $T' \to T$ is surjective and exhibits~$X$ as a closed subscheme of $\Spec T'$.
 
 Let us now show that the factorization $R \to T \to S$ has the required properties.
 Clearly, the ring homomorphism $T \to S$ is again surjective.
 Moreover, the composition $R \to T' \to T$ is finite and flat since, as we have noted above, $T' \to T$ is a surjective localization.
 By the connectedness of $\Spec R$, we conclude that $R \to T$ is even faithfully flat.
\end{proof}

\begin{lemma} \label{finite_henselian}
 Let $(A,A^+)$ be a henselian Huber pair and $(A,A^+) \to (B,B^+)$ a homomorphism with $A \to B$ integral and $\Spec B$ connected.
 Then also $(B,B^+)$ is henselian.
\end{lemma}

\begin{proof}
 Since~$A$ is henselian, the same holds for~$B$. 
 In order to complete the proof, it suffices to show that~$B^+$ is the preimage in~$B$ of a henselian valuation ring of $k_B = B/\m_B$.
 As $\m_A \subseteq A^+$, \cref{max_id_integral} implies that~$\m_B \subseteq B^+$.
 
 First we want to show that $k_B^+ := B^+/\m_B$ is integrally closed.
 Suppose $\bar{b} \in k_B$ is a zero of a monic polynomial $\bar{P}(T) \in k_B^+[T]$.
 We can lift~$\bar{b}$ to $b \in B$ and~$\bar{P}$ to a monic polynomial $P \in B^+[T]$.
 Then $P(b) \in \m_B \subseteq B^+$, so~$b$ is a zero of the monic polynomial $P(T) - P(b)$ in $B^+[T]$.
 But~$B^+$ is integrally closed in~$B$.
 Therefore $b \in B^+$ and $\bar{b} \in k_B^+$.
 
 The integral closure~$\cO$ of the henselian valuation ring~$k_A^+ = A^+/\m_A$ in $k_B$ is again a henselian valuation ring.
 By the last paragraph, it is contained in $k_B^+$.
 Therefore, $k_B^+$ is a localization of~$\cO$.
 In particular, it is a henselian valuation ring, as well (see \cite{EP2005}, Corollary~4.1.4).
\end{proof}

Remember that a homomorphism of Huber pairs $(A,A^+) \to (B,B^+)$ with $(A,A^+)$ complete is \emph{finite} if $A \to B$ is finite and $A^+ \to B^+$ is integral.
Then automatically $(B,B^+)$ is complete as well and the above lemma implies that if $(A,A^+)$ is henselian, then also $(B,B^+)$.

\begin{proposition} \label{characterization_henselian}
 For a local pseudo-adic space~$\cX$ the following are equivalent:
 \begin{enumerate}[(i)]
  \item $\underline{\cX}$ is the spectrum of a complete, henselian Huber pair.
  \item For every finite morphism $\cY \to \cX$ with~$\cY$ connected, $\cY$ is local.
 \end{enumerate}
\end{proposition}

\begin{proof}
 The implication from~(i) to~(ii) directly follows from \cref{finite_henselian}.
 Assume that~(ii) holds and write $\underline{\cX} = \Spa(A,A^+)$ for a complete Huber pair $(A,A^+)$.
 We have to show that~$A^+$ is henselian.
 Let $A^+ \to B^+$ be a finite ring homomorphism and assume that $\Spec B^+$ is connected.
 \cref{factorization_finite} gives us a factorization
 \[
  A^+ \to C_0^+ \to B^+,
 \]
 where $\Spec C_0^+$ is connected, $A^+ \to C_0^+$ is finite and faithfully flat, and $C_0^+ \to B^+$ is surjective.
 If we manage to show that~$C_0^+$ is local, we know that~$B^+$ is local, as well.
 
 Set $C = C_0^+ \otimes_{A^+} A$.
 Being the base change of the injective homomorphism $A^+ \to A$ along the flat homomorphism $A^+ \to C_0^+$, $C_0^+ \to C$ is injective.
 Let~$C^+$ be the integral closure of~$C_0^+$ in~$C$.
 Then $(A,A^+) \to (C,C^+)$ is a finite homomorphism of complete Huber pairs and $\Spa(C,C^+)$ is connected.
 By hypothesis, $(C,C^+)$ is local.
 
 Consider the homomorphism $C_0^+ \to C^+$.
 It is integral and injective.
 Hence, $\Spec C^+ \to \Spec C_0^+$ is surjective.
 We further know that~$C^+$ is local.
 But then also~$C_0^+$ has to be local because the preimage of every closed point of $\Spec C_0^+$ in $\Spec C^+$ is nonempty and consists of closed points of $\Spec C^+$.
\end{proof}

In order to obtain a geometric description of henselian Huber pairs, we introduce the \emph{Nisnevich site} of a pseudo-adic space~$\cX$.
The underlying category is the category of strongly \'etale morphisms $\cY \to \cX$.
Coverings are strongly \'etale coverings $(\cU_i \to \cU)_{i \in I}$ such that for every $u \in |\cU|$ there is $i \in I$ and $u_i \in |\cU_i|$ such that the residue field extension of $k(u_i)^+|k(u)^+$ is trivial ($k(u_i)|k(u)$ does not need to be trivial).
We denote the Nisnevich site of~$\cX$ by $\cX_{\Nis}$.

\begin{lemma} \label{characterization_Nis}
 For a pseudo-adic space~$\cX$, the following conditions are equivalent:
 \begin{enumerate} [\rm (i)]
  \item There is an $x \in |\cX|$ such that for every strongly \'etale morphism of pseudo-adic spaces $f: \cY \to \cX$ and every $y \in |\cY|$ with $f(y) = x$ and trivial residue field extension of $k(y)^+|k(x)^+$, there is an open neighborhood~$\cU$ of~$y$ such that~$f$ induces an isomorphism $\cU \to \cX$.
  \item $\cX$ is local and every Nisnevich covering of~$\cX$ splits.
  \item $\underline{\cX}$ is the spectrum of a henselian Huber pair and~$|\cX|$ contains the closed point of~$\underline{\cX}$.
 \end{enumerate}
\end{lemma}

\begin{proof}
 If~(i) is true,~$x$ is the unique closed point of~$\underline{\cX}$ as otherwise we get a contradiction by taking for~$f$ an open immersion which is not an isomorphism.
 Hence,~$\cX$ is local by \cref{localaffinoid}.
 Moreover, it is clear by condition~(i) that every covering of~$\cX$ splits.
 This shows that~(i) implies~(ii).

 Assuming~(ii), $\underline{\cX} = \Spa(A,A^+)$ for a local Huber pair $(A,A^+)$.
 By \cref{completion_local}, we may assume that $(A,A^+)$ is complete.
 Let us show that~$A^+$ is henselian.
 Let $A^+ \to B^+$ be \'etale such that $\Spec B^+$ is connected and contains a maximal ideal~$\m_B^+$ mapping to $\m_A^+ \in \Spec A^+$ with trivial residue field extension.
 Set $B = B^+ \otimes_{A^+} A$.
 Then~$B^+$ is integrally closed in~$B$ as this property is stable under smooth base change.
 Furthermore,
 \[
  (A,A^+) \to (B,B^+)
 \]
 is a strongly \'etale morphism of Huber pairs by \cref{comparenotionsofetale}.
 Since $A^+ \to B^+$ is flat, there is a prime ideal~$\m'_B$ of~$B^+$ specializing to~$\m_B^+$ and mapping to $\m_A \in \Spec A^+$.
 By definition $B = (A^+\setminus \m_A)^{-1}B^+$ and $\m_B := (A^+ \setminus \m_A)^{-1}\m'_B$ is an ideal of~$B$.
 As $\m_B \cap A = \m_A$ and $A \to B$ is \'etale, $\m_B$ is a maximal ideal of~$B$.
 Choose a valuation of~$B$ with support~$\m_B$ and center~$\m_B^+$.
 It is automatically continuous by the characterization given in \cite{Hu93}, Theorem~3.1.
 The corresponding point $y \in Spa(B,B^+)$ maps to the closed point~$x$ of $\Spa(A,A^+)$ and the residue field extension $k(y)^+|k(x)^+$ coincides with the residue field extension $k(\m_B^+)|k(\m_A^+)$, i.e., it is trivial.
 
 Let $\cU \subseteq \Spa(A,A^+)$ be the complement of the closed point.
 Then
 \[
  \Spa(B,B^+) \amalg \cU \to \Spa(A,A^+)
 \]
 is a Nisnevich covering.
 By assumption it splits.
 The image of the splitting lies in $\Spa(B,B^+)$ as~$\cU$ does not contain the closed point of $\Spa(A,A^+)$.
 Since we assumed $\Spec B^+$ (and hence $\Spa(B,B^+)$) to be connected, $(B,B^+) = (A,A^+)$.
 In particular, $B^+ = A^+$, so~$A^+$ is henselian.
 
 Assume that (iii) holds.
 We write $\underline{\cX} = \Spa(A,A^+)$ for a complete, henselian Huber pair $(A,A^+)$ and denote the closed point of~$\underline{X}$ by~$x$.
 Let $f: \cY \to \cX$ be a strongly \'etale morphism and $y \in |\cY|$ with $f(y) = x$ such the residue field extension of $k(y)^+|k(x)^+$ is trivial.
 Replacing $\cY$ by an open neighborhood of~$y$ we may assume that $\underline{\cY}$ is affinoid and connected.
 By \cite{Hu96}, Corollary~1.7.3~iii), there is a Huber pair $(B,B^+)$ of algebraically finite type over $(A,A^+)$ such that $A \to B$ is \'etale and $\underline{\cY} \cong \Spa(B,B^+)$.
 By Zariski's main theorem, there is a finite $A$-algebra~$C$ and an $A$-homomorphism $C \to B$ such that $\Spec B \to \Spec C$ is an open immersion.
 Let~$C^+$ be the integral closure of~$A^+$ in~$C$.
 We obtain a diagram
 \[
  \begin{tikzcd}
   \Spa(B,B^+)  \ar[rr,open]    \ar[dr] &               & \Spa(C,C^+)   \ar[dl] \\
                                        & \Spa(A,A^+)
  \end{tikzcd}
 \]
 By \cref{characterization_henselian}, $\Spa(C,C^+)$ is local with closed point~$z$.
 Since $\Spa(C,C^+) \to \Spa(A,A^+)$ is finite, $z$ is the only point mapping to the closed point $x \in \cX$.
 Therefore, $y = z$ and $(B,B^+) = (C,C^+)$. 
 
 Since $\Spa(B,B^+) \to \Spa(A,A^+)$ is strongly \'etale, $k_B^+ = B^+/\m_B$ is unramified over $k_A^+$.
 By \cref{comparenotionsofetale}, $k_B^+$ is thus \'etale over~$k_A^+$.
 Moreover, the residue field extension of $k_B^+|k_A^+$ is trivial.
 But $k_A^+$ is henselian, whence $k_B^+ = k_A^+$.
 It follows that $k_B = k_A$ and since $A$ is henselian and $A \to B$ is finite \'etale, we obtain $B = A$.
 Finally $B^+ = A^+$ holds because $B^+$ is the integral closure of~$A^+$ in~$B$.
\end{proof}

\begin{definition}
We call an adic space \emph{henselian} if it satisfies the equivalent conditions of \cref{characterization_Nis}.
A \emph{Nisnevich point} is a henselian pseudo-adic space~$\xi$ such that~$\underline{\xi}$ is the spectrum of an affinoid field and $|\xi| = \{s\}$ where~$s$ is the closed point of~$\underline{\xi}$.
\end{definition}

\subsection{Strongly and tamely henselian adic spaces}

We now move on to the strongly \'etale and tame topologies.

\begin{definition}
 Let $(A,A^+)$ be a Huber pair.
 \begin{enumerate}[\rm (i)]
  \item $(A,A^+)$ is \emph{strongly henselian} if it is local and~$A^+$ is strictly henselian.
  \item $(A,A^+)$ is \emph{tamely henselian} if it is strongly henselian and the value group of the associated valuation~$v$ is a $\Z[\frac{1}{p}]$-module, where~$p$ denotes the residue characteristic of~$A^+$.
 \end{enumerate}
\end{definition}

We have an analog of \cref{characterization_Nis}:

\begin{lemma} \label{characterizationlocal}
 For a pseudo-adic space~$\cX$, the following conditions are equivalent:
 \begin{enumerate}[\rm (i)]
  \item There is $x \in |\cX|$ such that for every strongly \'etale (tame) morphism of pseudo-adic spaces $f: \cY \to \cX$ and every $y \in |\cY|$ with $f(y) = x$
         there is an open neighborhood~$\cU$ of~$y$ such that~$f$ induces an isomorphism $\cU \to \cX$.
  \item $\cX$ is local and every strongly \'etale (tame) covering of~$\cX$ splits.
  \item $\underline{\cX}$ is the spectrum of a strongly (tamely) henselian Huber pair and~$|\cX|$ contains the closed point of $\underline{\cX}$.
 \end{enumerate}
\end{lemma}

\begin{proof}
 The implication from (i) to (ii) is proved in the same way as in \cref{characterization_Nis}.

 Assuming~(ii), we can write $\underline{\cX} = \Spa(A,A^+)$ for a complete, henselian Huber pair $(A,A^+)$ by \cref{characterization_Nis}.
 Let us show that~$A^+$ is strictly henselian.
 Let $A^+ \to B^+$ be finite \'etale and set $B = B^+ \otimes_{A^+} A$.
 Then~$B^+$ is integrally closed in~$B$ as this property is stable under smooth base change.
 Furthermore,
 \[
  \varphi: (A,A^+) \to (B,B^+)
 \]
 is a finite strongly \'etale morphism of Huber pairs by \cref{comparenotionsofetale}.
 By assumption, $\varphi$ splits.
 This implies~(iii) in the strongly \'etale case.

 In the tame case it remains to show that the value group~$\Gamma$ of the valuation~$|\cdot|$ corresponding to the closed point of~$\cX$ is divisible by all integers prime to the residue characteristic of~$A^+$.
 Take $\gamma \in \Gamma$ and an integer~$m$ prime to the residue characteristic of~$A^+$.
 We have to find $\gamma' \in \Gamma$ with $m\gamma' = \gamma$.
 We may assume that $\gamma \le 1$.
 Otherwise we replace $\gamma$ by its inverse.
 Take $a \in A$ with $|a| = \gamma$.
 Then $a \in A^{\times} \cap A^+$.
 Set
 \[
  B^+ = A^+[T]/(T^m-a)  \qquad  \text{and}  \qquad  B = B^+ \otimes_{A^+} A = A[T]/(T^m-a).
 \]
 We obtain a finite tame homomorphism $\varphi:(A,A^+) \to (B,B^+)$.
 By assumption, $\varphi$ splits.
 Let $\sigma: (B,B^+) \to (A,A^+)$ be a splitting.
 Then $\sigma(T)$ is an element of~$A$ with valuation equal to~$\gamma'$.

 In order to show that~(iii) implies~(i) assume that $\underline{\cX}$ equals the spectrum of a strongly (tamely) henselian Huber pair $(A,A^+)$ and that the closed point $x$ of $\underline{\cX}$ is contained in $|\cX|$.
 Let $f: \cY \to \cX$ be a strongly \'etale (tame) morphism and $y \in |\cY|$ with $f(y) = x$.
 Replacing $\cY$ by an open neighborhood of~$y$ we may assume that $\underline{\cY}$ is affinoid and connected.
 By the same arguments as in \cref{characterization_Nis}, we obtain that $\underline{\cY} = \Spa(B,B^+)$ for $(B,B^+)$ finite over $(A,A^+)$.
 
 In the strongly \'etale case, $k_B^+ = B^+/\m_B$ is unramified over $k_A^+$.
 By \cref{comparenotionsofetale}, $k_B^+$ is thus \'etale over~$k_A^+$.
 But $k_A^+$ is strictly henselian, whence $k_B^+ = k_A^+$.
 In the tame case $k_B|k_A$ is a tame extension of strictly henselian valued fields.
 Denote by $\Gamma_A$ and~$\Gamma_B$ the value groups of the valuations corresponding to~$k_A^+$ and~$k_B^+$, respectively.
 By \cite{GR03}, Corollary~6.2.14, $k_B|k_A$ is galois of degree prime to the residue characteristic~$p$ of~$k_A^+$ and
 \[
  \Gamma_B/\Gamma_A \cong \Hom_{\ZZ}(\Gal(k_B|k_A),\mu(k_A)).
 \]
 But $\Gamma_A$ is divisible by every integer prime to~$p$.
 Therefore, $k_B^+ = k_A^+$ also in this case.
 
 It follows that $k_B = k_A$ and since $A$ is henselian and $A \to B$ is finite \'etale, we obtain $B = A$.
 Finally $B^+ = A^+$ holds because $B^+$ is the integral closure of~$A^+$ in~$B$.
\end{proof}

\begin{definition}
 A prepseudo-adic space~$\cX$ is called \emph{strongly (tamely) local} or \emph{strongly (tamely) henselian} if~$\cX$ satisfies the equivalent conditions of \cref{characterizationlocal}.
 A \emph{strongly \'etale (tame) point} (in the category of prepseudo-adic spaces) is a strongly (tamely) local pseudo-adic space~$\xi$ such that~$\underline{\xi}$ is the spectrum of an affinoid field and $|\xi| = \{s\}$ where~$s$ is the closed point of~$\underline{\xi}$.
\end{definition}

In~\cite{Hu96}, Proposition~2.3.10 Huber proves the following:

\begin{proposition} \label{equivalencetopoihenselian}
 Let~$\cX$ be an adic space and~$x$ a point of~$\cX$.
 Let~$K$ be the henselization of $k(x)$ with respect to the valuation ring $k(x)^+$.
 Then the \'etale topos $(\cX,\{x\})^{\sim}_{\et}$ of the pseudo-adic space $(\cX,\{x\})$ is naturally equivalent to the \'etale topos $(\Spec K)^{\sim}_{\et}$.
\end{proposition}

Restricting to the Nisnevich, strongly \'etale, and tame site, respectively, we obtain:

\begin{corollary}
 In the situation of \cref{equivalencetopoihenselian} let~$K^+$ be an extension of $k(x)^+$ to~$K$.
 Let $K_{nr}$ and $K_t$ be the maximal extensions of~$K$ where~$K^+$ is unramified and tamely ramified, respectively.
 Set $G_{nr} = \Gal(K_{nr}|K)$ and $G_t = \Gal(K_t|K)$.
 Then the strongly \'etale topos $(\cX,\{x\})^{\sim}_{\set}$ of $(\cX,\{x\})$ is naturally equivalent to the topos $(\Spec K^+)_{\et}^{\sim}$, which in turn is equivalent to the topos of $G_{nr}$-sets,
  and the tame topos $(\cX,\{x\})^{\sim}_t$ is naturally equivalent to the $G_t$-sets.
 The Nisnevich topos of $(\cX,\{x\})$ is trivial.
\end{corollary}

\begin{corollary}
 For every strongly \'etale point~$\cS$ the global section functor
 \[
  \Gamma(\cS,-) : \tilde{\cS}_{\set} \to \mathit{sets}
 \]
 is an equivalence of categories.
 Analogously for tame and Nisnevich points.
\end{corollary}

\begin{definition}
 For a strongly \'etale point $u: \xi \to \cX$ of a prepseudo-adic space~$\cX$ and a sheaf~$\mathcal{F}$ on $\tilde{\cX}_{\et}$ we define the \emph{stalk} of~$\mathcal{F}$ at~$\xi$:
 \[
  \mathcal{F}_{\xi} := \Gamma(\xi,u^*\mathcal{F}).
 \]
 and for tame and Nisnevich points and sheaves accordingly.
\end{definition}
For a strongly \'etale or tame point $u: \xi \to \cX$ of a prepseudo-adic space~$\cX$ we consider the category~$C_{\xi}$ of pairs $(\cV,v)$
 where~$\cV$ is an object of the strongly \'etale or tame site, respectively, and $v : \xi \to \cV$ is a morphism over~$\cX$.
For a Nisnevich point $u: \xi \to \cX$ the category~$C_{\xi}$ consists of pairs $(\cV,v)$ with $\cV \to \cX$ strongly \'etale and $v: \xi \to \cV$ a morphism over~$\cX$ satisfying the following condition:
Let $x \in \cX$ and $y \in \cV$ be the respective images of~$\xi$.
Then the residue field extension of $k(y)^+|k(x)^+$ is trivial.
Similarly, for a point $\xi \in \cX$, we define the category~$C_{\xi}$ of open neighborhoods of~$\xi$.

The same argument as for the \'etale site (see \cite{Hu96}, Lemma~2.5.4) shows:

\begin{lemma} \label{stalkofpresheaf}
 In all cases the category~$C_{\xi}$ is cofiltered.
 For every presheaf~$\mathcal{P}$ on $\cX$, $\cX_{\Nis}$, $\cX_{\set}$ or $\cX_t$, respectively, there is a functorial isomorphism
 \[
  (a\mathcal{P})_{\xi} \cong \colim_{(\cV,v) \in C_{\xi}} \mathcal{P}(\cV),
 \]
 where~$a\mathcal{P}$ denotes the sheaf associated with~$\mathcal{P}$.
\end{lemma}

Over every point $x \in |\cX|$ we can choose a geometric point
\[
 \bar{x} := (\Spa(\bar{k}(x),\bar{k}(x)^+),\{s\})
\]
such that $\bar{k}(x)$ is a separable closure of~$k(x)$ (see~\cite{Hu96}, (2.5.2)).
Restricting to the henselization, the maximal unramified, and the maximal tamely ramified extension, respectively, yields a Nisnevich, a strongly \'etale, and a tame point
\begin{IEEEeqnarray*}{C}
 x_{\Nis} = (\Spa(k_h(x),k_h(x)^+),\{s_h\})	\qquad x_{\set} = (\Spa(k_{nr}(x),k_{nr}(x)^+),\{s_{\set}\}), \\
 x_t = (\Spa(k_t(x),k_t(x)^+),\{s_t\}),
\end{IEEEeqnarray*}
where $k_{nr}(x)$ and $k_t(x)$ are the maximal unramified and maximal tamely ramified subextensions of $\bar{k}(x)|k(x)$ and $k_h(x)$ is the henselization of~$k(x)$.
From \cref{stalkofpresheaf} we conclude that there are enough points:

\begin{corollary}
 The families of functors
 \begin{IEEEeqnarray*}{C}
  (\tilde{\cX}_{\Nis} \to \mathit{sets}, \mathcal{F} \mapsto \mathcal{F}_{x_{\Nis}})_{x \in |\cX|},	\quad
  (\tilde{\cX}_{\set} \to \mathit{sets}, \mathcal{F} \mapsto \mathcal{F}_{x_{\set}})_{x \in |\cX|}	\\
  \text{and}	\quad	(\tilde{\cX}_t \to \mathit{sets}, \mathcal{F} \mapsto \mathcal{F}_{x_t})_{x \in |\cX|}
 \end{IEEEeqnarray*}
 are conservative.
\end{corollary}

\begin{proof}
 Let~$\mathcal{F}$ be a sheaf on~$\cX_{\set}$ and assume that $\mathcal{F}_{x_{\set}} = 0$ for all $x \in |\cX|$.
 Take a strongly \'etale morphism $f:\cU \to \cX$ and an element $a \in \mathcal{F}(\cU)$.
 By \cref{stalkofpresheaf} we find for each $u \in |\cU|$ a strongly \'etale neighborhood $\cU(u) \to \cX$ of~$f(u)_{\set}$ factoring through $(\cU,u)$ such that $a|_{\cU(u)} = 0$.
 The $\cU(u) \to \cU$ comprise a covering of~$\cU$, whence $a = 0$.
 For the other topologies the proof is the same.
\end{proof}

The last part of this section is dedicated to localizations of adic spaces in the various topologies we are studying.
The constructions and the proofs of the resulting properties are analogous to the case of the strict localization treated in \cite{Hu96}, \S~2.5.
We thus allow ourselves to omit the proofs.

Let~$\xi \to \cX$ either be a set theoretic point, a Nisnevich point, a strongly \'etale, or a tame point of~$\cX$.
In each of the four cases we consider the respective category~$C_{\xi}$ defined above and define the \emph{localization}~$X_{\xi}$ as follows.
We set
\begin{align*}
 \cO_{\cX,\xi}		&:= \colim_{(\cV,v) \in C_{\xi}} \cO_{\underline{\cV}}(\underline{\cV}), \\
 \cO^+_{\cX,\xi} 	&:= \colim_{(\cV,v) \in C_{\xi}} \cO^+_{\underline{\cV}}(\underline{\cV}).
\end{align*}
and equip these rings with the following topology:
Let $(\cV,v)$ be an object of~$C_{\xi}$ with~$\cV$ affinoid.
Choose an ideal of definition~$I$ of a ring of definition of $\cO_{\underline{\cV}}(\underline{\cV})$ and take
\[
 \{I^n \cdot \cO^+_{\cX,\xi} \mid n \in \N \}
\]
to be a fundamental system of neighborhoods of zero.
As in \cite{Hu96}, (2.5.9) this topology is independent of the choice of $(\cV,v)$ and~$I$
 and $(\cO_{\cX,\xi},\cO_{\cX,\xi}^+)$ is a sheafy Huber pair.
Put
\[
 \underline{\cX}_{\xi} := \Spa(\cO_{\cX,\xi},\cO_{\cX,\xi}^+)
\]
and
\[
 |\cX_{\xi}| := \bigcap_{(\cV,v) \in C_{\xi}} \underline{\varphi}^{-1}_{(\cV,v)}(|\cV|),
\]
where~$\underline{\varphi}_{(\cV,v)}$ is the natural morphism $\underline{\cX}_{\xi} \to \underline{\cV}$.
We obtain a local (respectively henselian, respectively strongly henselian, respectively tamely henselian) prepseudo-adic space
\[
 \cX_{\xi} := (\underline{\cX}_{\xi},|\cX_{\xi}|).
\]
We call~$\cX_{\xi}$ the \emph{localization} (respectively \emph{henselization}, respectively \emph{strong henselization}, respectively \emph{tame henselization}) of~$\cX$ at~$\xi$.
Let $D_{\xi}$ be the full (cofinal) subcategory of $C_{\xi}$ consisting of those pairs $(\cV,v)$ in $C_{\xi}$ with affinoid~$\underline{\cV}$ and quasi-compact~$|\cV|$.
Then $\cX_{\xi}$ is a projective limit of the spaces~$\cV$ for $(\cV,v) \in D_{\xi}$ in the sense of \cite{Hu96}, (2.4.2).
In particular, the results of \cref{limits} apply.

We want to give a more explicit, ring theoretic description of the localization described above.
Let $(A,A^+)$ be a Huber pair and $I$ an ideal of definition of a ring of definition of~$A$.
For a point $x \in \Spa(A,A^+)$  we define the localization $(A_x,A_x^+)$ by defining~$A_x$ to be the localization of~$A$ at the support of~$x$ and~$A_x^+$ to be the preimage in~$A_x$ of the valuation ring $k(x)^+ \subseteq k(x)$.
We equip $(A_x,A_x^+)$ with the topology such that $I^nA_x^+$ for $n \in \N$ form a basis of neighborhoods of zero.
Then $(A_x,A_x^+)$ is a local Huber pair.

Given a Nisnevich point
\[
 \xi = (\Spa(k,k^+),\{s\}) \to \cX
\]
with image $x \in \cX$, we define the henselization $(A_{\xi}^h,A_{\xi}^{h+})$ at~$\xi$:
Let~$A_x^h$ be the henselization of~$A_x$.
Denoting the maximal ideals of~$A_x$ and~$A_x^h$ by~$\m_x$ and~$\m_x^h$, respectively, we have a natural isomorphism of residue fields
\[
 k_{A_x} = A_x/\m_x \cong A_x^h/\m_x^h = k_{A_x^h}.
\]
Via this identification, the valuation ring $k^+_{A_x} = A_x^+/\m_x$ corresponds to a valuation ring $k^+_{A_x^h}$ with quotient field $k_{A_x^h}$.
The morphism $\xi \to \cX$ induces a morphism of Huber pairs
\[
 (k_{A_x^h},k^+_{A_x^h}) \to (k,k^+).
\]
As~$k^+$ is henselian, we can consider the henselization of~$k^+_{A_x^h}$ as a subring $k^+_{A_{\xi}^h} \subseteq k^+$.
By \cite{Nag53}, Theorem~8, $k^+_{A_{\xi}^h}$ is again a valuation ring and we denote its quotient field by $k_{A_{\xi}^h}$.
Then $k_{A_{\xi}^h}|k_{A_x^h}$ is a separable extension and as~$A_x^h$ is henselian, there is a unique ind-\'etale, integral extension $A_x^h \to A_{\xi}^h$ whose residue field extension is $k_{A_{\xi}^h}|k_{A_x^h}$.
Writing~$A_{\xi}^{h+}$ for the preimage of $k^+_{A_{\xi}^h}$ in~$A_{\xi}^h$, we obtain the henselization $(A_{\xi}^h,A_{\xi}^{h+})$ of $(A,A^+)$ at~$\xi$ (with topology defined in the same way as for $(A_x,A_x^+)$).

Note that by construction, $k^+_{A_{\xi}^h}$ and~$A_{\xi}^h$ are henselian.
Therefore, $A_{\xi}^{h+}$ is henselian by \cref{local_henselian}.
Moreover, $A^+_x \to A_{\xi}^{h+}$ is ind-\'etale.
We conclude that~$A_{\xi}^{h+}$ is the henselization of~$A^+_x$ at its maximal ideal and $A_{\xi}^h = A_{\xi}^{h+} \otimes_{A_x^+} A_x$.
We could have taken this as a definition of the henselization of $(A,A^+)$.
However, it would have been difficult to see directly that $(A_{\xi}^h,A_{\xi}^{h+})$ is local without resorting to the theory of Pr\"ufer extensions that we only explain in \cref{PrueferHuber}.

Next, if~$\xi$ is even a strongly \'etale point, we can define the strong henselization of $(A,A^+)$ at~$\xi$.
Let $A_{\xi}^{\sh+}$ be the strict henselization of $A_{\xi}^{h+}$ and set
\[
 A_{\xi}^{\sh} = A_{\xi}^{\sh +} \otimes_{A^+} A = A_{\xi}^{\sh +} \otimes_{A_{\xi}^{h+}} A_{\xi}^h.
\]
Then $A_{\xi}^h \to A_{\xi}^{\sh}$ is an integral, ind-\'etale extension of local rings and~$A_{\xi}^{\sh +}$ is integrally closed in $A_{\xi}^{\sh}$.
Using \cref{finite_henselian}, we conclude that $(A_{\xi}^{\sh},A_{\xi}^{\sh +})$ is strongly henselian.
Equipped with the topology defined by the ideal~$I$ as before, $(A_{\xi}^{\sh},A_{\xi}^{\sh +})$ is the strong henselization of $(A,A^+)$ at~$\xi$.

Finally, assume that~$\xi$ is also a tame point.
We define the tame henselization of $(A,A^+)$ at~$\xi = (\Spa(k,k^+),\{s\})$ as follows.
Let $k_{\xi}^t$ be the maximal tamely ramified subextension of $k|k_{\xi}^{\sh}$ and denote by~$k_{\xi}^{t+}$ the corresponding valuation ring of $k_{\xi}^t$.
Let $A_{\xi}^{\sh} \to A_{\xi}^t$ be the unique ind-\'etale extension with residue field extension $k_{\xi}^{\sh} \to k_{\xi}^t$.
Denote by $A_{\xi}^{t+}$ the preimage of $k_{\xi}^{t+}$ in $A_{\xi}^t$.
Then $(A_{\xi}^t,A_{\xi}^{t+})$ is the tame henselization of~$(A,A^+)$ at~$\xi$ (with topology as before).

\begin{proposition}
 Let~$\cX$ be a prepseudo-adic space, $\xi \to \cX$ a set theoretic point (respectively Nisnevich, strongly \'etale, or tame point) of~$\cX$ with image $x \in |\cX|$.
 \begin{enumerate}[\rm (i)]
  \item Assume~$x$ is analytic.
        Consider the natural morphisms
        \begin{IEEEeqnarray*}{ll}
         p: \Spa(k(x),k(x)^+) \to \underline{\cX},					\quad &p_{\Nis} : \Spa(k_h(x),k_h(x)^+) \to \underline{\cX}, \\
         p_{\set} : \Spa(k_{nr}(x),k_{nr}(x)^+) \to \underline{\cX},	\quad &p_t : \Spa(k_t(x),k_t(x)^+) \to \underline{\cX}.
        \end{IEEEeqnarray*}
        Then
        \begin{IEEEeqnarray*}{lcl}
         \cX_{\xi} \cong (\Spa(k(x),k(x)^+),p^{-1}(|\cX|))	\quad &\text{or}	\quad & \cX_{\xi} \cong (\Spa(k_h(x),k_h(x)^+),p_h^{-1}(|\cX|)), \\
         \cX_{\xi} \cong (\Spa(k_{nr}(x),k_{nr}(x)^+),p_{\set}^{-1}(|\cX|))	\quad &\text{or}	\quad & \cX_{\xi} \cong (\Spa(k_t(x),k_t(x)^+),p_t^{-1}(|\cX|)),
        \end{IEEEeqnarray*}
        according to whether~$\xi$ is a set theoretic, a Nisnevich, a strongly \'etale, or a tame point of~$\cX$.
  \item Assume that~$x$ is non-analytic.
        Take an affinoid open neighborhood $\cU = \Spa(A,A^+)$ of~$x$.
        Let $(B,B^+)$ be the localization (respectively henselization, strong henselization, or tame henselization) of $(A,A^+)$ with respect to~$\xi$.
        Let~$p$ be the natural morphism $\Spa(B,B^+) \to \underline{\cX}$.
        Then
        \[
         \cX_{\xi} \cong (\Spa(B,B^+),p^{-1}(|\cX|)).
        \]
 \end{enumerate}
\end{proposition}

\begin{proof}
The argument is the same as in the proof of the corresponding statement for the \'etale site (\cite{Hu96}, Proposition~2.5.13).
\end{proof}

\section{Topological invariance}

Let $\tau \in \{\set,t,\et\}$ be one of the topologies.
In this section we prove some assertions concerning the topological invariance of the $\tau$-cohomology.
They are in analogy with the respective results concerning the \'etale topology.

\begin{proposition} \label{topological_invariance}
 Let $\cX \to \cY$ be a morphism of adic spaces which induces an isomorphism on the underlying reduced adic spaces.
 Then
 \[
  \cU \mapsto \cU \times_{\cY} \cX
 \]
 defines an isomorphism of sites $\cX_{\tau} \to \cY_{\tau}$.
 In particular, the topoi $\Sh(\cX_{\tau})$ and~$\Sh(\cY_{\tau})$ are equivalent.
\end{proposition}

\begin{proof}
 Without loss of generality we may assume that~$\cX$ and~$\cY$ are affinoid.
 Moreover, it suffices to prove that $\cU \mapsto \cU \times_{\cY} \cX$ defines an equivalence of the subcategories of affinoid spaces in $\cX_{\tau}$ and~$\cY_{\tau}$, respectively.
 The general statement follows by gluing.
 Write $\cX = \Spa(A,A^+)$ and $\cY = \Spa(B,B^+)$.
 By \cite{Hu96}, Corollary~1.7.3, the affinoid adic spaces that are \'etale over~$\cX$ are precisely the open subspaces of adic spaces of the form $\Spa(R,R^+)$ with~$R$ \'etale over~$A$ and~$R^+$ the integral closure of~$A^+$ in~$R$ and analogously for~$\cY$.
 By \cite{EGAIV.4}, 18.1.2, the assignment $S \mapsto S \otimes_B A$ defines an equivalence of the categories of \'etale $B$-algebras and \'etale $A$-algebras.
 Moreover, for~$S$ \'etale over~$B$ and~$S^+$ the integral closure of~$B^+$ in~$S$, the categories of open subspaces of $\Spa(S,S^+)$ and $\Spa((S,S^+) \otimes_{(B,B^+)} (A,A^+))$ are equivalent as the underlying topological spaces of $\Spa(S,S^+)$ and $\Spa((S,S^+) \otimes_{(B,B^+)} (A,A^+))$ are naturally homeomorphic.
 We conclude that~$\cX_{\et}$ and~$\cY_{\et}$ are equivalent.
 In order to see that this is also true for the tame and strongly \'etale sites it suffices to note that the properties of being tame or strongly \'etale only depend on the underlying reduced subspaces.
\end{proof}

The following two results are analogs of the excision theorems in \'etale cohomology.
They concern the $\tau$-cohomology of an adic space with support on a Zariski closed subspace.
Cohomology groups with support are defined, more generally, for a closed subspace~$\cZ$ of a pseudo-adic space~$\cX$.
Writing~$\cU$ for the complement of~$\cZ$, the cohomology groups with support $H^i_{\cZ}(\cX_{\tau},\mathcal{F})$ of a sheaf~$\mathcal{F}$ on $\cX_{\tau}$ are the cohomology groups of the derived functor of
\[
 \mathcal{F} \mapsto \ker(\mathcal{F}(\cX) \to \mathcal{F}(\cU)).
\]

\begin{lemma} \label{excision}
 Consider the following commutative diagram of adic spaces
 \[
  \begin{tikzcd}
   \cZ'_{\textit{red}}	\ar[r,closed]	\ar[d,"\sim"]	& \cX'	\ar[d,"\pi"]	\\
   \cZ_{\textit{red}}	\ar[r,closed]					& \cX	
  \end{tikzcd}
 \]
 where $\cZ' \to \cX'$ and $\cZ \to \cX$ are closed immersions, $\pi$ is a morphism in $\cX_{\tau}$, and $\cZ'_{\textit{red}} \to \cZ_{\textit{red}}$ is an isomorphism.
 Then for any sheaf $\mathcal{F}$ on~$\cX_{\tau}$ and any $i \ge 0$ we have
 \[
  H^i_{\cZ}(\cX_{\tau},\mathcal{F}) \overset{\sim}{\longrightarrow} H^i_{\cZ'}(\cX'_{\tau},\mathcal{F}|_{\cX'}).
 \]
\end{lemma}

\begin{proof}
 The proof is the same as for the \'etale topology on schemes (see \cite{Fu15}, Proposition~5.6.12).
\end{proof}

\begin{proposition} \label{excision_point}
 Let~$\cX$ be an adic space and $x \in \cX$ a Zariski-closed point (i.e., $x=\Spa(k(x),k(x))$ and $\Spa(k(x),k(x)) \to \cX$ is a closed immersion).
 Then for any sheaf~$\mathcal{F}$ on $\cX_{\tau}$ and any $i \ge 0$ we have
 \[
  H^i_x(\cX_{\tau},\mathcal{F}) = H^i_x((\cX_x^h)_{\tau},\mathcal{F}),
 \]
 where~$\cX_x^h$ denotes the henselization of~$\cX$ at~$x$.
\end{proposition}

\begin{proof}
 As $\tau$-cohomology commutes with limits by \cref{cohomologyoflimit}, we have
 \[
  H^i_x((\cX_x^h)_{\tau},\mathcal{F}) = \colim_{(\cY,y) \to (\cX,x)} H^i_y(\cY_{\tau},\mathcal{F}),
 \]
 where the colimit runs over all pointed \'etale morphisms $(\cY,y) \to (\cX,x)$ such that $k(y) = k(x)$.
 We can as well restrict to pointed morphisms $(\cY,y) \to (\cX,x)$ in $\cX_{\tau}$ as every \'etale morphism $(\cY,y) \to (\cX,x)$ as above is strongly \'etale, hence tame, at~$y$ and the strongly \'etale locus is open (\cref{unramified}).
 For a pointed morphism $(\cY,y) \to (\cX,x)$ in $\cX_{\tau}$ with $k(y) = k(x)$ we know by \cref{excision} that
 \[
  H^i_y(\cY_{\tau},\mathcal{F}) = H^i_x(\cX_{\tau},\mathcal{F}).\qedhere
 \]
\end{proof}

\section{Comparison with \'etale cohomology} \label{comparisonwithetalecohomology}

\begin{lemma} \label{henselian}
 Let $(A,A^+)$ be a henselian Huber pair.
 Denote by~$k$ the residue field of~$A$ and by~$\kappa$ the residue field of~$A^+$.
 Choose a separable closure~$\bar{k}$ of~$k$ and denote by~$\bar{v}$ the continuation of the valuation of~$k$ corresponding to the closed point of $\Spa(A,A^+)$.
 This defines a geometric point $\xi \to \Spa(A,A^+)$ which we can also view as tame and strongly \'etale point.
 Write~$k^t$ for the maximal subextension of~$\bar{k}|k$ where~$\bar{v}$ is tamely ramified.
 Then for any abelian sheaf~$\mathcal{F}$ on $\Spa(A,A^+)_{\et}$ and any $i \ge 0$
 \[
  H^i(\Spa(A,A^+)_{\et},\mathcal{F}) = H^i(G_k,\mathcal{F}_{\xi}),
 \]
 for any sheaf~$\mathcal{F}$ on $\Spa(A,A^+)_{\set}$ and any $i \ge 0$
 \[
  H^i(\Spa(A,A^+)_{\set},\mathcal{F}) = H^i(G_{\kappa},\mathcal{F}_{\xi}),
 \]
 and for any sheaf~$\mathcal{F}$ on $\Spa(A,A^+)_t$ and any $i \ge 0$
 \[
  H^i(\Spa(A,A^+)_t,\mathcal{F}) = H^i(\Gal(k^t|k),\mathcal{F}_{\xi}).
 \]
\end{lemma}

\begin{proof}
 This follows using the Hochschild-Serre spectral sequence for $G_k$, $G_{\kappa}$ (which can be identified with the Galois group of the maximal unramified subextension of $\bar{k}|k$) and $\Gal(k^t|k)$, respectively.
\end{proof}

For a prepseudo-adic space~$\cX$ we write $\ch^+(\cX)$ for the set of characteristics of the residue fields of $\cO_{\cX,x}^+$ for $x \in |\cX|$.

\begin{proposition} \label{comparetameetale}
 Let~$\cX$ be a prepseudo-adic space and~$\mathcal{F}$ a torsion sheaf on $\cX_{\et}$ with torsion prime to $\ch^+(\cX)$.
 Then the morphism of sites $\varphi:\cX_{\et} \to \cX_t$ induces isomorphisms
 \[
  H^i(\cX_t,\varphi_*\mathcal{F}) \overset{\sim}{\longrightarrow} H^i(\cX_{\et},\mathcal{F})
 \]
 for all $i \geq 0$.
\end{proposition}

\begin{proof}
 We have to show that for any tamely henselian $(A,A^+)$ and any torsion sheaf~$\mathcal{G}$ on $(A,A^+)_{\et}$ with torsion prime to the residue characteristic~$p$ of~$A^+$, the cohomology groups
 \[
  H^i(\Spa(A,A^+)_{\et},\mathcal{G})
 \]
 vanish for all $i \ge 1$.
 By \cref{henselian} these cohomology groups equal
 \[
  H^i(G_k,\mathcal{G}_{\xi}),
 \]
 where~$k$ and~$\xi$ are defined as in \cref{henselian}.
 But $G_k$ is a pro-$p$-group (see \cite{EP2005}, Theorem~5.3.3) and $\mathcal{G}_{\xi}$ is a torsion $G_k$-module with torsion prime to~$p$.
 Therefore, the above cohomology groups vanish.
\end{proof}

\begin{lemma} \label{compareetalecohomology}
 Let $X \to S$ be a morphism of schemes and~$\mathcal{F}$ a torsion sheaf on~$X_{\et}$.
 Then the morphism of sites
 \[
  \psi:\Spa(X,S)_{\et} \to X_{\et}
 \]
 induces isomorphisms
 \[
  H^i(X_{\et},\psi^*\mathcal{F}) \overset{\sim}{\longrightarrow} H^i(\Spa(X,S)_{\et},\mathcal{F}).
 \]
 for all $i \ge 0$.
\end{lemma}

\begin{proof}
 If~$X$ and~$S$ are affine, the result is a special case of \cite{Hu96}, Theorem~3.3.3.
 Let us now assume that~$S$ is affine and~$X$ is arbitrary.
 By virtue of the Leray spectral sequence associated with~$\psi$, it suffices to show
 \[
  \psi_*\psi^*\mathcal{F} \overset{\sim}{\to} \mathcal{F} \qquad \text{and} \qquad R^i\psi_*(\psi^* \mathcal{F}) = 0~\text{for $i > 0$}.
 \]
 These assertions are local on~$X$.
 Hence, we are reduced to the affine case.

 The next step is to only require~$S$ to be separated.
 We choose an open covering $\mathcal{U}$ of~$S$ by affine schemes~$S_i$.
 It induces an open covering~$\mathcal{V}$ of $\Spa(X,S)$ by the open subspaces
 \[
  \Spa(X \times_S S_i,S_i) \subseteq \Spa(X,S).
 \]
 We obtain a morphism of \v{C}ech-to-derived spectral sequences
 \[
  \begin{tikzcd}
   \check{H}^i(\mathcal{U},\mathcal{H}^j(\mathcal{F}))		\ar[r,Rightarrow]	\ar[d]	& H^{i+j}(X,\mathcal{F})		\ar[d]	\\
   \check{H}^i(\mathcal{V},\mathcal{H}^j(\psi^*\mathcal{F}))	\ar[r,Rightarrow]		& H^{i+j}(\Spa(X,S),\psi^*\mathcal{F}).
  \end{tikzcd}
 \]
 The separatedness assumption assures finite intersections of the~$S_i$ to be affine.
 Therefore, we can use the previous case to conclude that all vertical morphisms on the left are isomorphisms.
 Hence, the right vertical morphism is an isomorphism.
 The general case follows from the case where~$S$ is separated by the same argument using a covering of~$S$ by separated open subschemes.
\end{proof}

Combining \cref{compareetalecohomology} with \cref{comparetameetale} we obtain:

\begin{corollary}
 Let $X \to S$ be a morphism of schemes and~$\mathcal{F}$ a torsion sheaf on~$X_{\et}$ with torsion prime to the residue characteristics of~$S$.
 Then the morphisms of sites
 \[
  \Spa(X,S)_t \overset{\varphi}{\longleftarrow} \Spa(X,S)_{\et} \overset{\psi}{\longrightarrow} X_{\et}
 \]
 induce isomorphisms
 \[
  H^i(\Spa(X,S)_t,\varphi_*\psi^*\mathcal{F}) \cong H^i(X_{\et},\mathcal{F})
 \]
 for all $i \ge 0$.
\end{corollary}

We prove the following comparison of tame and strongly \'etale cohomology.

\begin{proposition} \label{comparison_strongly_etale_tame}
 Let~$\cX$ be an adic space with $\ch^+(\cX) = p > 0$.
 Then for any $p$-torsion sheaf~$\mathcal{F}$ on~$\cX_t$ the natural morphism of sites
 \[
  \varphi: \cX_t \to \cX_{\set}
 \]
 induces isomorphisms
 \[
  H^i(\cX_{\set},\varphi_*\mathcal{F}) \overset{\sim}{\longrightarrow} H^i(\cX_t,\mathcal{F})
 \]
 for every integer~$i$.
\end{proposition}

\begin{proof}
 We have to show that the stalks of the higher direct images $R^i\varphi_*\mathcal{F}$ vanish for $i > 0$.
 Let~$\bar{x}$ be a strongly \'etale point of~$\cX$.
 The strong henselization $\cX_{\bar{x}}^{\set}$ is of the form $\Spa(A,A^+)$ with $(A,A^+)$ local and~$A^+$ strictly henselian.
 For the stalk of $R^i\varphi_*\mathcal{F}$ at~$\bar{x}$ we get by \cref{cohomologyoflimit} and \cref{henselian}
 \[
  R^i\varphi_*\mathcal{F}_{\bar{x}} = H^i(\Spa(A,A^+)_t,\mathcal{F}) = H^i(\Gal(k^t|k),\mathcal{F}_{\bar{x}}),
 \]
 where $k^t|k$ is the maximal tamely ramified extension of the residue field of~$A$ with respect to the valuation corresponding to~$\bar{x}$.
 But by assumption~$\mathcal{F}$ is a $p$-torsion sheaf and as $A^+$ is strictly henselian, $\Gal(k^t|k)$ has trivial $p$-Sylow subgroups.
 Therefore, the above cohomology group vanishes by \cite{NSW}, Proposition~1.6.2.
\end{proof}

\cref{comparison_strongly_etale_tame} tells us that for $p$-torsion sheaves tame and strongly \'etale cohomology coincide.
Moreover, by \cref{comparetameetale}, for torsion sheaves with torsion invertible on~$\cX$, tame cohomology coincides with \'etale cohomology.
In that sense the tame topology is a bridge between \'etale and strongly \'etale topology.

\begin{lemma} \label{centerisomorphism}
 Let~$X$ be a scheme and $\tau \in \{t,\set,\et\}$ one of the topologies.
 Then the center map $c : \Spa(X,X) \to X$ induces for every sheaf~$\mathcal{F}$ on $\Spa(X,X)_{\tau}$ isomorphisms
 \[
  H^i(X_{\et},c_*\mathcal{F}) \overset{\sim}{\longrightarrow} H^i(\Spa(X,X)_{\tau},\mathcal{F})
 \]
 for all $i \ge 0$.
\end{lemma}

\begin{proof}
 It is easy to check that~$c$ induces a morphism of cites $\Spa(X,X)_{\tau} \to X_{\et}$ by mapping an \'etale morphism $Y \to X$ to the strongly \'etale (and thus \'etale and tame) morphism $\Spa(Y,Y) \to \Spa(X,X)$.
 We need to check that the higher direct images of~$\mathcal{F}$ vanish.
 In order to do so we may assume that~$X$ is strictly henselian.
 But then $\Spa(X,X)$ is strictly local (so in particular tamely and strongly local) and thus its cohomology groups vanish in degree greater than zero.
\end{proof}

Combining \cref{centerisomorphism} with \cref{properisomorphismonSpa} ($\Spa(X,X) \overset{\sim}{\to} \Spa(X,S)$ for $X \to S$ proper) we obtain the following.

\begin{proposition}
 Let $X \to S$ be a proper morphism of schemes and let~$\tau \in \{t,\set,\et\}$ be one of the topologies.
 Then the center map $c : \Spa(X,S) = \Spa(X,X) \to X$ induces for every sheaf~$\mathcal{F}$ on $\Spa(X,S)_{\tau}$ isomorphisms
 \[
  H^i(X_{\et},c_*\mathcal{F}) \overset{\sim}{\longrightarrow} H^i(\Spa(X,S)_{\tau},\mathcal{F})
 \]
 for all $i \ge 0$.
\end{proposition}

\section{Comparison with the tame fundamental group} \label{tamefundamentalgroup}

Let~$X$ be a regular scheme of finite type over some base scheme~$S$.
Suppose there is a regular compactification~$\bar{X}$ of~$X$ over~$S$ such that the complement of~$X$ in~$\bar{X}$ is the support of a strict normal crossing divisor~$D$.
Then, following \cite{SGA1}, Exp.~VIII, \S~2, we can study finite \'etale covers of~$X$ which are tamely ramified along~$D$.
This results in the definition of the tame fundamental group $\pi_1^t(X/S,\bar{x})$ for some geometric point~$\bar{x}$ of~$X$.

Under less favorable regularity assumptions, there are several approaches to define the tame fundamental group.
We only state the two of these which we use in this section.
Fix an integral, pure-dimensional, separated, and excellent base scheme~$S$.
In \cite{Wie08} Wiesend introduces the notion of curve-tameness.
It has been slightly extended by Kerz and Schmidt in \cite{KeSch10} to the following definition:
A curve over~$S$ is a scheme of finite type $C$ over~$S$ which is integral and such that
\[
 \dim_S C := \trdeg(k(C)|k(T)) + \dim_{\text{Krull}} T = 1,
\]
where~$T$ denotes the closure of the image of~$C$ in~$S$.
Any curve~$C$ has a canonical compactification~$\bar{C}$ over~$S$ which is regular at the points in $\bar{C}-C$.
Hence, we can define tameness over~$C$ as in \cite{SGA1}:
A finite \'etale cover $C' \to C$ by a connected, hence integral, curve~$C'$ is tame at a point $c \in \bar{C} - C$ if the corresponding valuation of the function field of~$C$ is tamely ramified in the extension of function fields $k(C')|k(C)$.
For a general finite \'etale cover $C' \to C$ we require tameness for each connected component of~$C'$.
Given a scheme~$X$ of finite type over~$S$, a finite \'etale cover $Y \to X$ is \emph{curve-tame} if the base-change to any curve $C \to X$ is tamely ramified outside $C \times_X Y$.

Let us recall next the notion of valuation-tameness considered in \cite{KeSch10}.
A finite \'etale cover $Y \to X$ of connected, normal schemes of finite type over~$S$ is \emph{valuation-tame} if every valuation of the function field $k(X)$ with center on~$S$ is tamely ramified in the finite, separable field extension $k(Y)|k(X)$.

This section is concerned with comparing the fundamental group of the tame site with the curve-tame and the valuation tame fundamental group.
In order to do so we need to relate tame covers with torsors in the tame topos.

\begin{lemma} \label{descent}
 Let $\pi:\cY \to \cX$ be a surjective \'etale morphism of discretely ringed adic spaces.
 Then~$\pi$ satisfies descent for finite morphisms.
\end{lemma}

\begin{proof}
 The same arguments as for schemes reduce us to the case where $\cX = \Spa(A,A^+)$ and $\cY = \Spa(B,B^+)$ are affinoid.
 Then $\Spec B \to \Spec A$ is a surjective \'etale morphism of schemes.
 Moreover, finite morphisms to~$\cX$ and~$\cY$ correspond to finite $A$-algebras and $B$-algebras, respectively.
 Hence, we can apply descent theory for schemes (\cite{SGA1}, Exp. VIII, Th\'eor\`eme~2.1) to obtain the result.
\end{proof}

\begin{corollary} \label{torsorscoverings}
 Let $\tau \in \{\et,t,\set\}$ be one of the topologies on a discretely ringed adic space~$\cX$.
 Let~$\mathcal{F}$ be a torsor in $\Sh(\cX_{\tau})$ for some finite group~$G$.
 Then~$\mathcal{F}$ is represented by a finite Galois morphism $\cY \to \cX$ in~$\cX_{\tau}$ with Galois group~$G$.
\end{corollary}

\begin{proof}
 Let $\cX' \to \cX$ be a covering of~$\cX$ such that $\mathcal{F}|_{\cX'}$ is trivial, hence represented by $\pi':\coprod_G \cX' \to \cX'$.
 By \cref{descent} the morphism~$\pi'$ descends to a finite Galois morphism $\pi:\cY \to \cX$ in~$\cX_{\tau}$ representing~$\mathcal{F}$.
\end{proof}

For a geometric point $\bar{x}$ of a connected, locally noetherian adic space~$\cX$ we want to define the fundamental group of the corresponding pointed site $(\cX_{\tau},\bar{x})$ (for $\tau \in \{\et,t,\set\}$).
To be more precise, we want a pro-finite group $\pi_1(\cX_{\tau},\bar{x})$ that classifies finite torsors, i.e. for every finite group~$G$ the set of isomorphism classes of $G$-torsors in $\Sh(\cX_{\tau})$ should be given by
\[
 \Hom(\pi_1(\cX_{\tau},\bar{x}),G).
\]
In \cite{AM}, \S 9 Artin and Mazur describe the construction of the fundamental pro-group of a locally connected site via the Verdier functor.
By \cite{AM}, Corollary~10.7, it classifies all torsors (not just finite).
Taking the pro-finite completion we obtain a pro-finite group classifying finite torsors.
In order to apply these results in our situation, we need to check that~$\cX_{\tau}$ is locally connected.
But this is true because the connected components of an affinoid noetherian adic space~$\cX$ are in one-to-one correspondence with the idempotents of the noetherian ring~$\cO_{\cX}(\cX)$.
By descent (\cref{torsorscoverings}), the resulting fundamental group $\pi_1(\cX_{\tau},\bar{x})$ not only classifies finite $G$-torsors in $\Sh(\cX_{\tau})$ but also finite Galois $\tau$-covers.

\begin{proposition} \label{etalefundamentalgroup}
 Let $X \to S$ be a morphism of connected, noetherian schemes and~$\bar{x}$ a geometric point of~$X$.
 We can view~$\bar{x}$ as a geometric point of $\Spa(X,S)$ by taking the trivial valuation on the residue field of~$\bar{x}$.
 Then there is a natural isomorphism
 \[
  \pi_1(X_{\et},\bar{x}) \cong \pi_1(\Spa(X,S)_{\et},\bar{x}).
 \]
\end{proposition}

\begin{proof}
 By what we have just discussed, the \'etale fundamental group of $\Spa(X,S)$ classifies finite \'etale covers of $\Spa(X,S)$.
 Similarly, $\pi_1(X_{\et},\bar{x})$ classifies finite \'etale covers of~$X$.
 Every finite \'etale cover $Y \to X$ induces a finite \'etale cover $\Spa(Y,S) \to \Spa(X,S)$.
 For two finite \'etale covers $Y \to X$ and $Y' \to X$ the natural homomorphism
 \[
  \Hom_X(Y,Y')	\longrightarrow \Hom_{\Spa(X,S)}(\Spa(Y,S),\Spa(Y',S))
 \]
 is bijective, an inverse being given by assigning to a morphism $\Spa(Y,S) \to \Spa(Y',S)$ the corresponding morphism of supports $Y \to Y'$.
 It remains to show that every finite \'etale cover of $\Spa(X,S)$ comes from a finite \'etale cover of~$X$.

 Let $\varphi: \cZ \to \Spa(X,S)$ be a finite \'etale cover of adic spaces.
 We need to show that it comes from a finite \'etale cover of~$X$ as above.
 Let $\Spa(B,B^+)$ and $\Spa(A,A^+)$ be affinoid open subspaces of~$\cZ$ and $\Spa(X,S)$, respectively, such that $\varphi(\Spa(B,B^+)) \subseteq \Spa(A,A^+)$.
 By \cite{Hu96}, Corollary~1.7.3, we obtain a factorization
 \[
  \begin{tikzcd}
   \Spa(B,B^+)	\ar[rr,open]	\ar[dr]	&			& \Spa(B,A^+)	\ar[dl]	\\
					& \Spa(A,A^+).
  \end{tikzcd}
 \]
 and $A \to B$ is \'etale.
 Since we are working with discretely ringed adic spaces, this construction glues and we obtain a diagram
 \[
  \begin{tikzcd}
   \cZ	\ar[rr,open]	\ar[dr,"\varphi"']	&   		& \Spa(Y,S)	\ar[dl]	\\
                                            & \Spa(X,S)
  \end{tikzcd}
 \]
 with $Y \to X$ \'etale and~$\cZ$ dense in $\Spa(Y,S)$.

 By assumption there is an \'etale covering $\cW \to \Spa(X,S)$ trivializing $\varphi$.
 Without loss of generality we may assume that~$\cW$ is a disjoint union of adic spaces of the form $\Spa(X_i,S_i)$.
 In particular, $\coprod_i X_i \to X$ is an \'etale covering of~$X$.
 Moreover,
 \[
  \cZ_i := \cZ \times_{\Spa(X,S)} \Spa(X_i,S_i) \cong \Spa(X_i,S_i) \otimes G
 \]
 for some finite group~$G$.
 Base changing the above diagram to $\Spa(X_i,S_i)$ we obtain
 \[
  \begin{tikzcd}
   \Spa(X_i,S_i) \otimes G	\ar[rr,open]	\ar[dr]	&		& \Spa(Y \times_X X_i,S_i)	\ar[dl]	\\
							& \Spa(X_i,S_i)
  \end{tikzcd}
 \]
 and $\Spa(X_i,S_i) \otimes G$ is open and dense in $\Spa(Y \times_X X_i,S_i)$.
 But $\Spa(X_i,S_i) \otimes G \to \Spa(X_i,S_i)$ satisfies the valuative criterion for properness and hence,
 \[
  \Spa(X_i,S_i) \otimes G = \Spa(X_i \otimes G,S_i) = \Spa(Y \times_X X_i,S_i).
 \]
 We conclude that $X_i \otimes G = Y \times_X X_i$.
 This shows that $Y \to X$ is a finite \'etale cover such that $\cZ = \Spa(Y,S)$.
\end{proof}

\begin{proposition}
 Let~$X$ be a connected, regular scheme of finite type over~$S$ and~$\bar{x}$ a geometric point of~$X$.
 Then the valuation-tame fundamental group $\pi_1^{vt}(X/S,\bar{x})$ is canonically isomorphic to the fundamental group $\pi_1(\Spa(X,S)_t,\bar{x})$ of the tame site $\Spa(X,S)_t$.
\end{proposition}

\begin{proof}
 By \cref{etalefundamentalgroup} we have to show that a finite \'etale cover $Y \to X$ is valuation-tame over~$S$ if and only if $\Spa(Y,S) \to \Spa(X,S)$ is tame.
 If the latter is true, it is clear that the former also holds.
 Suppose that $Y \to X$ is valuation-tame and pick a point $z = (x,R,\phi) \in \Spa(X,S)$.
 Since~$X$ is regular at~$x$, we find a discrete valuation~$v$ (not necessarily of rank one) supported on the generic point $\eta = \Spec k(X)$ and a morphism $\psi: \Spec \cO_v \to X$ mapping the closed point of $\Spec \cO_v$ to~$x$ such that $k(v) = k(x)$.
 (It can be obtained by taking a a regular system of parameters $a_1,\ldots,a_n$ of $\cO_{X,x}$ and composing the valuations~$v_i$ corresponding to the divisor $V(a_1,\ldots,a_i)$ of $V(a_1,\ldots,a_{i-1})$.)
 The concatenation of~$v$ with the valuation corresponding to~$R$ gives a valuation ring~$R'$ of~$k(X)$ and $\phi$ and~$\psi$ determine a morphism $\alpha:\Spec R' \to S$.
 By assumption any point of $\Spa(Y,S)$ lying over $(\eta,R',\alpha)$ is tame over $\Spa(X,S)$.
 This implies that the same is true for any point lying over~$z$.
\end{proof}

Here is a stronger version but with some assumptions on resolutions of singularities:

\begin{proposition}
 Let~$S$ be an integral, excellent and pure-dimensional base scheme and~$X$ a connected scheme of finite type over~$S$ with a geometric point~$\bar{x}$.
 Assume that every finite separable extension of every residue field of~$X$ admits a regular proper model.
 Then the curve-tame fundamental group $\pi_1^{ct}(X/S,\bar{x})$ is canonically isomorphic to $\pi_1(\Spa(X,S)_t,\bar{x})$.
\end{proposition}

\begin{proof}
 By \cref{etalefundamentalgroup} we have to show that a finite \'etale cover $Y \to X$ is curve-tame over~$S$ if and only if $\Spa(Y,S) \to \Spa(X,S)$ is tame.
 Suppose $\Spa(Y,S) \to \Spa(X,S)$ is tame and let $C \rightarrow X$ be a curve mapping to $X$ with compactification~$\bar{C}$.
 Without loss of generality we may assume that $C \to X$ is a closed immersion.
 Let~$\eta_C$ be the generic point of~$C$ viewed as a point of~$X$.
 A point $c \in \bar{C}-C$ corresponds to a valuation ring $\cO_c \subseteq k(\eta_C)$ and comes naturally with a morphism $\phi_c:\Spec \cO_c \to S$.
 This defines a point $(\eta_C,\cO_c,\phi_c)$ of $\Spa(X,S)$.
 By assumption all points of~$\Spa(Y,S)$ lying over $(\eta_C,\cO_c,\phi_c)$ are tame over $\Spa(X,S)$.
 This translates to $C \times_X Y \to C$ being tamely ramified over~$c$.
 We conclude that $Y \to X$ is curve-tame.

 Suppose now that $Y \to X$ is curve-tame.
 Take a point $(x,R,\phi) \in \Spa(X,S)$.
 Let $Z$ be the closed subset $\overline{\{x\}}$ of~$X$ with the reduced scheme structure.
 In order to show that $\Spa(Y,S) \to \Spa(X,S)$ is tame we may replace $Y \to X$ by its base change to~$Z$.
 Note that $Z \times_X Y \to Z$ is still curve-tame.
 Hence, we may assume that~$X$ is integral with generic point~$x$.
 Furthermore, by the same argument, we may replace~$X$ by a nonempty open subscheme.
 We may thus assume that~$X$ is regular.
 But now under our assumption on resolution of singularities $Y \to X$ is curve tame if and only if it is valuation-tame (see \cite{KeSch10}, Theorem~4.4).
 In particular, every point of $\Spa(Y,S)$ lying over $(x,R,\phi)$ is tame over $\Spa(X,S)$.
\end{proof}

\section{Coherent cohomology for discretely ringed adic spaces} \label{cohomologydiscrete}

All cohomology groups in this section are sheaf cohomology groups on the underlying topological space of the scheme or adic space in question (not on the tame or \'etale site etc.).

Let~$S$ be an excellent noetherian scheme.
We say that \emph{resolution of singularities holds over~$S$} if for any reduced scheme~$X$ of finite type over~$S$ there is a locally projective birational morphism $X' \to X$ such that~$X'$ is regular and $X' \to X$ is an isomorphism over the regular locus of~$X$.

A morphism of schemes $X \to S$ is said to be a \emph{pro-open immersion} if it is a limit of open immersions with affine transition morphisms.
In this case we also say that~$X$ is pro-open in~$S$.
Examples are open subschemes of~$S$ and the localization of~$S$ at some point $s \in S$.
A scheme~$X$ is \emph{essentially of finite type} over~$S$ if there is a scheme~$T$ of finite type over~$S$ and a pro-open immersion $X \to T$ over~$S$.
A \emph{compactification} of a scheme~$X$ essentially of finite type over~$S$ is a proper $S$-scheme~$T$ together with a pro-open immersion $X \to T$ over~$S$.
By \cite{Con07}, if~$S$ is quasi-compact and quasi-separated and $X \to S$ is separated and essentially of finite type, a compactification exists.

Let $\pi:X \to S$ be a morphism of schemes.
In this section we compare the cohomology of the sheaf~$\cO_{\cX}^+$ on the discretely ringed adic space $\cX = \Spa(X,S)$ with the cohomology of the structure sheaf~$\cO_{\bar{X}}$ of a regular compactification~$\bar{X}$ of~$X$ over~$S$ (provided it exists).
More precisely, we assume that $\pi: X \to S$ is separated and essentially of finite type, $X$ is regular, and resolution of singularities holds over~$S$.
Then a regular compactification $\bar{\pi}:\bar{X} \to S$ of~$X$ over~$S$ exists and $\Spa(X,S) = \Spa(X,\bar{X})$.
In \cref{section_adic_spaces} we defined the center map, which is a morphism of ringed spaces
\[
 (c,c^+): (\cX=\Spa(X,S),\cO_{\cX}) \to (S,\cO_S)
\]
sending $(x,R,\phi) \in \Spa(X,S)$ to the image of the closed point of $\Spec R$ under~$\phi$.
In this section we show under the assumption of resolution of singularities that the center map induces an isomorphism
\[
 Rc_*\cO_{\cX}^+ \longrightarrow R\bar{\pi}_*\cO_{\bar{X}}.
\]
In particular, this implies
\[
 H^i(\cX,\cO_{\cX}^+) \cong H^i(\bar{X},\cO_{\bar{X}}).
\]
Since $\Spa(X,S)$ is naturally isomorphic to $\Spa(X,\bar{X})$ we may replace~$S$ by~$\bar{X}$ and work with a regular scheme~$S$ and a pro-open $X \subseteq S$.
First, we want to show that $c^+ : \cO_S \to c_*\cO_{\cX}$ is an isomorphism.

\begin{lemma} \label{restrictionisomorphism}
 Let~$X, Y \subseteq S$ be dense pro-opens in an integral normal scheme such that $X \subseteq Y$.
 Set~$\cS = \Spa(S,S)$.
 The restriction
 \[
  \rho:\cO_{\cS}^+(\Spa(Y,S)) \to \cO_{\cS}^+(\Spa(X,S))
 \]
 is an isomorphism.
\end{lemma}

\begin{proof}
 It suffices to prove the lemma for~$Y=S$ and~$S$ affine.
 If $X = \Spec A$ is affine,
 \[
  \Spa(X,S) = \Spa(A,A^+),
 \]
 where~$A^+$ is the integral closure of the image of~$\cO_S(S)$ in~$A$.
 By our assumptions on~$S$ and~$X$, we obtain
 \[
  A^+ = \cO_S(S)
 \]
 and thus
 \[
  \Spa(S,S) = \Spa(A^+,A^+).
 \]
 The homomorphism~$\rho$ becomes the identity on~$A^+$.

 In the general case cover~$X$ by affine open subschemes~$X_i$.
 We obtain an affinoid covering
 \[
  \coprod_i \Spa(X_i,S) \to \Spa(X,S)
 \]
 and thus a diagram of exact sequences
 \[
  \begin{tikzcd}
   0	\ar[r]	& \cO_{\cS}^+(\Spa(S,S))	\ar[r]	\ar[d,"\rho"']	& \prod_i \cO_{\cS}^+(\Spa(S,S))	\ar[r]	\ar[d,"\sim"]	& \prod_{ij} \cO_{\cS}^+(\Spa(S,S))		\ar[d]	\\
   0	\ar[r]	& \cO_{\cS}^+(\Spa(X,S))	\ar[r]			& \prod_i \cO_{\cS}^+(\Spa(X_i,S))	\ar[r]		& \prod_{ij} \cO_{\cS}^+(\Spa(X_i \cap X_j,S)).
  \end{tikzcd}
 \]
 Note that the assumptions of the lemma also hold for~$X_i$ or $X_i \cap X_j$ instead of~$X$.
 Since the middle arrow is injective, $\rho$ is injective.
 Applying the same reasoning to $\Spa(X_i \cap X_j,S)$ instead of $\Spa(X,S)$, we see that the right arrow is injective.
 This implies that~$\rho$ is surjective.
\end{proof}

\begin{proposition} \label{structuresheaf}
 Let~$X$ be pro-open in an integral normal scheme $S$.
 With the above notation the homomorphism
 \[
  c^+:\cO_S \to c_*\cO_{\cX}^+
 \]
 is an isomorphism.
\end{proposition}

\begin{proof}
 We can check this on open affines of~$S$, i.e. we may assume that~$S$ is affine and have to show that
 \[
  c^+(S) : \cO_S(S) \to \cO_{\cX}^+(\cX)
 \]
 is an isomorphism.
 Denote by~$c': \cS = \Spa(S,S) \to S$ the center map.
 By functoriality we obtain a commutative diagram
 \[
  \begin{tikzcd}
					& \cO_S(S)	\ar[dl,"(c')^+(S)"']	\ar[dr,"c^+(S)"]	&			\\
   \cO_{\cS}^+(\cS)	\ar[rr,"\rho"]	&								& \cO_{\cX}^+(\cX).
  \end{tikzcd}
 \]
 Since~$\rho$ is an isomorphism by \cref{restrictionisomorphism}, it suffices to show that $(c')^+(S)$ is an isomorphism.
 But $(c')^+(S)$ is just the identity on~$\cO_S(S)$.
\end{proof}

For the rest of this section we assume that~$S$ is regular and connected and that~$X$ is dense pro-open in~$S$.
Denote by~$\mathcal{B}$ the full subcategory of the category of open subspaces of $\Spa(X,S)$ of the form $\Spa(Y,T)$ coming from a commutative diagram of regular schemes
 \begin{equation} \label{diagramneighborhoodbasis}
  \begin{tikzcd}
   Y	\ar[r,open]	\ar[d,open]	& X	\ar[d,open]	\\
   T	\ar[r]					& S,			\\
  \end{tikzcd}
 \end{equation}
 such that $Y \to X$ is an open immersion, $Y \to T$ is a pro-open immersion, and $T \to S$ is locally quasi-projective.
 Since $X \to S$ is a pro-open immersion as well, $T \to S$ is birational.

 Our assumption on resolution of singularities implies that the objects of~$\mathcal{B}$ form a basis of neighborhoods of the topological space $\Spa(X,S)$.
 Indeed, if we start with an affinoid open $\Spa(A,A^+)$, we can first choose a projective compactification of $\Spec A$ over~$A^+$ and then resolve its singularities to obtain a regular, locally projective compactification $Z$.
 Then
 \[
  \Spa(A,A^+) = \Spa(A,Z)
 \]
 is an object of~$\mathcal{B}$.
 In particular, all affinoid open subspaces are contained in~$\mathcal{B}$.
 By elimination of indeterminacies and resolution of singularities, we see that every morphism in $\mathcal{B}$ comes from a diagram
 \[
  \begin{tikzcd}
   Y'	\ar[r,open]	\ar[d,open]	& Y	\ar[r,open]	\ar[d,open]	& X	\ar[d,open]	\\
   T'	\ar[r]					& T	\ar[r]					& S,
  \end{tikzcd}
 \]
 such that $Y' \to Y$ and $Y \to X$ are open immersions, $Y \to T$ and $Y' \to T'$ are pro-open immersions, and $T' \to T$ and $T \to S$ are locally quasi-projective and birational.

\begin{lemma} \label{intersection_stable}
 The intersection of two objects in~$\mathcal{B}$ is again an object of~$\mathcal{B}$.
\end{lemma}

\begin{proof}
 Suppose we are given two objects $\Spa(Y_1,T_1)$ and $\Spa(Y_2,T_2)$ in~$\mathcal{B}$.
 The intersection of $\Spa(Y_1,T_1)$ with $\Spa(Y_2,T_2)$ is the same as the intersection of $\Spa(Y_1 \cap Y_2,T_1)$ with $\Spa(Y_1,\cap Y_2,T_2)$.
 Hence, we may assume that $Y_1 = Y_2 =: Y$.
 Choose locally projective compactifications~$\bar{T}_i$ of~$T_i$ over~$S$.
 By elimination of indeterminacies and resolution of singularities, we find a locally projective birational morphism $T' \to S$ from a regular scheme~$T'$ dominating~$\bar{T}_1$ and~$\bar{T}_2$ which is an isomorphism over~$Y$.
 We denote the preimages of~$T_1$ and~$T_2$ in~$T'$ by $T'_1$ and~$T'_2$.
 As $T'_i \to T_i$ is proper, we have
 \[
  \Spa(Y,T_i) = \Spa(Y,T'_i).
 \]
 by \cref{properisomorphismonSpa}.
 But then
 \[
  \Spa(Y,T_1) \cap \Spa(Y,T_2) = \Spa(Y,T'_1 \cap T'_2),
 \]
 which is in~$\mathcal{B}$.
\end{proof}

We equip~$\mathcal{B}$ with the structure of a site by defining coverings in~$\mathcal{B}$ to be surjective families.

\begin{lemma} \label{basisofneighborhoods}
 The topoi associated with~$\mathcal{B}$ and $\Spa(X,S)$ are equivalent.
\end{lemma}

\begin{proof}
 We have a natural morphism of sites $\varphi: \Spa(X,S)_{\top} \to \mathcal{B}$, where $\Spa(X,S)_{\top}$ denotes the site associated with the topological space $\Spa(X,S)$.
 The pullback $\varphi^{-1}$ is fully faithful and the topology on $\mathcal{B}$ is induced by the topology of $\Spa(X,S)$.
 In order to show that the corresponding morphism of topoi is an equivalence, it suffices to verify that the objects of~$\mathcal{B}$ form a basis of the topology of $\Spa(X,S)$ (see \cite{SGA4}, Expos\'e~III, Th\'eor\`eme~4.1).
This is the case as we have seen above.
\end{proof}

Before we prove the next proposition we want to explain what we mean by a flasque sheaf.
For a sheaf~$\mathcal{F}$ on a site~$\mathcal{B}$ the following are equivalent (see \cite{SGA4}, Expos\'e~V, Proposition~4.3):
\begin{enumerate}[(i)]
 \item $H^i(U,\mathcal{F}) = 0$ for all $i > 0$ and all $U \in \mathcal{B}$,
 \item $\check{H}^i(U,\mathcal{F}) = 0$ for all $i > 0$ and all $U \in \mathcal{B}$.
\end{enumerate}
If~$\mathcal{F}$ satisfies these equivalent conditions, it is called \emph{flasque}.

\begin{proposition} \label{compareZariskicohomology}
 Let~$X$ be dense and pro-open in an excellent, regular, connected scheme~$S$ and assume that resolution of singularities holds over~$S$.
 Then
 \[
  Rc_*\cO_{\cX}^+ \cong \cO_S,
 \]
 where 
 \[
  c : \cX = \Spa(X,S) \to S
 \]
 is the center map.
 In particular,
 \[
  H^i(S,\cO_S) \cong H^i(\cX,\cO_{\cX}^+)
 \]
 for all $i \geq 0$.
\end{proposition}

\begin{proof}
 By \cref{structuresheaf},
 \[
  c_* \cO_{\cX}^+ \cong \cO_S.
 \]
 In order to prove that $R^j c_* \cO_{\cX}^+ = 0$ for $j \ge 1$, it is enough to show that
 \[
  H^j(\Spa(X \times_S S',S'),\cO_{\cX}^+)
 \]
 vanishes for every open affine $S' \subseteq S$.
 Since~$S'$ and~$X \times_S S'$ satisfy the assumptions of the proposition if~$S$ and~$X$ do, we are reduced to proving that
 \[
  H^j(\cX,\cO_{\cX}^+) = 0
 \]
 in case~$S$ is affine.

 Consider the site~$\mathcal{B}$ defined before \cref{basisofneighborhoods}.
 By \cref{basisofneighborhoods} we can compute the cohomology group $H^j(\cX,\cO_{\cX}^+)$ in~$\mathcal{B}$.
 We claim that the restriction of~$\cO_{\cX}^+$ to~$\mathcal{B}$ is flasque.
 Take an open covering
 \[
  \Spa(Y,T) = \bigcup_{i\in I} \Spa(Y_i,T_i)
 \]
 in~$\mathcal{B}$ coming from commutative diagrams (\ref{diagramneighborhoodbasis}) as before
  and assume in addition that~$I$ is finite and that all~$T_i$ are affine.
 Every covering of~$\Spa(Y,T)$ in~$\mathcal{B}$ is dominated by one of this type.
 We want to examine the \v{C}ech complex
 \begin{equation}	\label{chechcomplex}
   0 \to \cO^+_{\cX}\big(\Spa(Y,T)\big) \to \prod_i \cO^+_{\cX}\big(\Spa(Y_i,T_i)\big) \to \prod_{ij} \cO^+_{\cX}\big(\Spa(Y_i,T_i) \cap \Spa(Y_j,T_j)\big) \to \ldots
 \end{equation}
 By \cref{restrictionisomorphism} this complex does not change if we replace~$Y$ and~$Y_i$ by $\bigcap_{i \in I} Y_i$.
 We may thus assume that $Y = Y_i$ for all~$i \in I$.
 By the same argument as before, we may find a locally projective birational morphism $\pi:T' \to T$ with~$T'$ regular and open subschemes~$T_i'$ of~$T'$ such that the morphisms $T'_i \to T$ factor through locally projective birational morphisms $\pi_i:T_i' \to T_i$.
 Since the adic spaces $\Spa(Y,T_i')$ cover $\Spa(Y,T')$, it follows that the schemes~$T_i'$ cover~$T'$.
 The following diagram summarizes the situation:
 \[
  \begin{tikzcd}
										& Y		\ar[dl,open]	\ar[d,open]	\ar[dr,open]		\\
   T'_i		\ar[r,"\pi_i"]	\ar[d,open]	& T_i	\ar[r]										& T	\\
   T'	\ar[urr,"\pi"],
  \end{tikzcd}
 \]
 where the morphisms~$\pi_i$ and $\pi$ are locally projective birational and all schemes in the diagram are regular.

 By \cref{structuresheaf}, the above \v{C}ech complex (\ref{chechcomplex}) equals
 \[
  0 \to \cO_{T'}(T') \to \prod_i \cO_{T'}(T'_i) \to \prod_{i,j} \cO_{T'}(T'_i \cap T'_j) \to \ldots
 \]
 This is the \v{C}ech complex for the covering $T' = \bigcup_i T'_i$ and the structure sheaf~$\cO_{T'}$.
 By \cite{ChRu15}, Theorem~1.1, we know that for each~$i$ the higher direct images $R^j\pi_{i*}\cO_{T'_i}$ vanish.
 Since~$T_i$ is affine, this implies
 \[
  H^q(T'_i,\cO_{T'}) = 0	\quad \forall i ,\ \forall q \ge 1.
 \]
 Our \v{C}ech complex thus computes the cohomology groups $H^q(T',\cO_{T'})$.
 Applying \cite{ChRu15}, Theorem~1.1 to $\pi: T' \to T$, we obtain that the corresponding higher direct images are trivial.
 Together with the fact that~$T$ is affine, this yields
 \[
  H^q(T',\cO_{T'}) \cong H^q(T,\cO_T) = 0  ,\ \forall q \ge 1.
 \]
 As a consequence the cohomology of the complex (\ref{chechcomplex}) is trivial.
 We conclude that~$\cO_{\cX}^+$ is flasque on~$\mathcal{B}$ and thus
 \[
  H^q(\cX,\cO_{\cX}^+) = 0  ,\ \forall q \ge 1. 
 \]
 The second assertion follows using the Leray spectral sequence
 \[
  H^i(S,R^j c_* \cO_{\cX}^+) \Rightarrow H^{i+j}(\cX,\cO_{\cX}^+). \qedhere
 \]
\end{proof}

\begin{corollary}
 Let~$S$ be an excellent scheme and assume that resolution of singularities holds over~$S$.
 Let $\pi: X \to S$ be separated and essentially of finite type and $\bar{\pi} : \bar{X} \to S$ a regular compactification of~$X$ over~$S$.
 Set $\cX = \Spa(X,S)$.
 Then there is a natural isomorphism
 \[
  Rc_*\cO_{\cX}^+ \to R\bar{\pi}_*\cO_{\bar{X}},
 \]
 where $c: \cX \to S$ denotes the center map.
 In particular,
 \[
  H^i(\cX,\cO_{\cX}^+) \cong H^i(\bar{X},\cO_{\bar{X}})
 \]
 for all $i \ge 0$.
\end{corollary}

\begin{proof}
 Consider the commutative diagram
 \[
  \begin{tikzcd}
   \Spa(X,\bar{X})	\ar[r,"c_{\bar{X}}"]	\ar[d,"="]	& \bar{X}	\ar[d,"\bar{\pi}"]	\\
   \Spa(X,S)		\ar[r,"c"]							& S.
  \end{tikzcd}
 \]
 We have an isomorphism in the derived category of~$S$
 \[
  (R\bar{\pi}_*) \circ (Rc_{\bar{X}*})\cO_{\cX}^+ \cong Rc_*\cO_{\cX}^+.
 \]
 By \cref{compareZariskicohomology}, $Rc_{\bar{X}^*}\cO_{\cX}^+ \cong \cO_{\bar{X}}$.
 This yields the first statement.
 The second one follows by comparing the two Leray spectral sequences
 \begin{IEEEeqnarray*}{lCl}
  H^i(S,R^j\bar{\pi}_*\cO_{\bar{X}}) &\Rightarrow& H^{i+j}(\bar{X},\cO_{\bar{X}}), \\
  H^i(S,R^jc_*\cO_{\cX}^+)				 &\Rightarrow& H^{i+j}(\cX,\cO_{\cX}^+).
 \end{IEEEeqnarray*} \qedhere
\end{proof}

\section{Pr\"ufer Huber pairs} \label{PrueferHuber}

For an affinoid adic space $\cX = \Spa(A,A^+)$ the cohomology of the structure sheaf~$\cO_{\cX}$ vanishes (see \cite{Hu94}, Theorem~2.2).
For the sheaf~$\cO_{\cX}^+$, however, we can not expect in general that $H^i(\cX,\cO_{\cX}^+) = 0$.
Of course, if $(A,A^+)$ is local, the cohomology of~$\cO_{\cX}^+$ vanishes.
But the class of local adic spaces turns out to be too small to compute cohomology groups as an \'etale covering of a local adic space does not necessarily admit a refinement by local adic spaces.
In the following we investigate a broader class of Huber pairs containing the local Huber pairs: the Pr\"ufer Huber pairs.

For a subring~$A^+$ of a ring~$A$ and a prime ideal~$\p^+$ of~$A^+$ we use the notation
\[
 A_{\p^+} := (A^+ \setminus \p^+)^{-1}A
\]

\begin{definition}
 A Huber pair $(A,A^+)$ is said to be Pr\"ufer if $A^+ \subseteq A$ is a Pr\"ufer extension, i.e. if $(A_{\m^+},A^+_{\m^+})$ is local for every maximal ideal~$\m^+$ of~$A^+$.
\end{definition}

Note that if $(A_{\m^+},A^+_{\m^+})$ is local, we a have a unique valuation of~$A_{\m^+}$ corresponding to the closed point of $\Spa(A_{\m^+},A^+_{\m^+})$.
Denote by~$v$ the valuation of~$A$ obtained by composing with $A \to A_{\m^+}$.
Then $v \in \Spa(A,A^+)$ and $(A_{\m^+},A^+_{\m^+}) = (A_v,A_v^+)$.

Examples of Pr\"ufer Huber pairs are $(\Q,\Z)$ and $(\bC[T,T^{-1}],\bC[T])$.
More generally, if~$D^+$ is a Dedekind domain and~$D$ is a localization of~$D$, then $(D,D^+)$ is a Pr\"ufer Huber pair.
A less trivial example is given by taking $A = \Q[T]$ and $A^+$ its subring
\begin{IEEEeqnarray*}{rCl}
 A^+	& =	& \{P(T) \in \Q[T] \mid P(0) \in \Z\}	\\
		& = & p^{-1}(\Z),
\end{IEEEeqnarray*}
where $p : \Q[T] \to \Q$ is the projection to the residue field of $(T)$.

\begin{remark}
 The definition of a Pr\"ufer extension in \cite{KneZha}, \S~5, differs from the one given here.
 First of all, if the topology of~$A$ is non-discrete there is the additional condition that~$\m_A^+$ be open and bounded.
 But apart from this topological issue the two definitions coincide.
 Let us first explain their notation and then show that the two definitions are equivalent in the case of discrete topology.
 It follows that in the general case our Pr\"ufer extensions are mapped to their Pr\"ufer extensions by the forgetful functor from the category of topological rings to the category of rings.
 
 A valuation~$v$ of a ring~$A$ is called \emph{Manis} if $A \to \Gamma_v \cup \{0\}$ is surjective.
 A \emph{Manis pair} in~$A$ is a pair $(A^+,\p^+)$ comprised of a subring~$A^+$ of~$A$ and a prime ideal $\p^+ \subseteq A^+$ such that there is a (unique) Manis valuation~$v$ of~$A$ such that
 \begin{IEEEeqnarray*}{rCl}
  A^+	& =	& \{a \in A \mid v(a) \le 1\},	\\
  \p^+	& = & \{a \in A \mid v(a) < 1\}.
 \end{IEEEeqnarray*}
 The pair $(A^+,\p^+)$ is called local if~$A^+$ is local with maximal ideal~$\p^+$.
 
 If $(A,A^+)$ is a local Huber pair, the corresponding valuation is clearly Manis, so $(A^+,\m_A^+)$ is a local Manis pair.
 Conversely, suppose that $(A^+,\m^+)$ is a local Manis pair in~$A$ associated with a Manis valuation~$v$.
 By \cite{KneZha}, Proposition~2.2, Proposition~2.5, the pair $(A,\supp v)$ is also local, hence $(A,A^+)$ is a local Huber pair.
 
 For a subring~$A^+$ of~$A$, a prime ideal $\p^+$ of~$A^+$ and an $A^+$-submodule $M$ of~$A$ we define the \emph{saturation} of~$M$ at~$\p^+$:
 \[
  M_{[\p^+]} = \{x \in A \mid \exists s \in A^+ \setminus \p^+~\text{with}~sx \in M\}.
 \]
 According to the definition in \cite{KneZha},~\S~5, $A^+ \subseteq A$ is a Pr\"ufer extension if $(A^+_{[\m^+]},\m^+_{[\m^+]})$ is a Manis pair in~$A$ for any maximal ideal~$\m^+$ of~$A^+$.
 By \cite{KneZha}, Proposition~2.10 this is equivalent to $(A_{\m^+},\m^+_{\m^+})$ being a Manis pair in $A_{\m^+}$, which in turn is equivalent to $(A_{\m^+},A^+_{\m^+})$ being local as we have seen above.
 Thus we recover our definition of Pr\"ufer extensions. 
\end{remark}

Recall from \cite{KneZha}, \S~3, that a ring homomorphism $A^+ \to A$ is called \emph{weakly surjective} if for any prime ideal~$\mathfrak{p}^+$ of~$A^+$ with $\mathfrak{p}^+A\ne A$ the homomorphism $A^+_{\mathfrak{p}^+} \to A_{\mathfrak{p}^+}$ is surjective.
Examples of weakly surjective ring homomorphisms are surjective ring homomorphisms and localizations.
By~\cite{KneZha}, Theorem~5.2, $(1) \Leftrightarrow (2)$, a ring extension $A^+ \to A$ is Pr\"ufer if and only if~$A^+$ is weakly surjective in any $A$-overring of~$A^+$.

It will turn out in \cref{Zariskiacyclic} that if $(A,A^+)$ is a complete Pr\"ufer Huber pair and~$A$ is either a strongly noetherian Tate ring or noetherian with the discrete topology, then the cohomology of~$\cO_{\cX}^+$ vanishes on $\cX = \Spa(A,A^+)$.

\begin{lemma} \label{completionPruefer}
 Let $(A,A^+)$ be a Pr\"ufer Huber pair.
 Then its completion $(\hat{A},\hat{A}^+)$ is Pr\"ufer.
\end{lemma}

\begin{proof}
 We factor $(A,A^+) \to (\hat{A},\hat{A}^+)$ as
 \[
  (A,A^+) \to (\bar{A},\bar{A}^+) \to (\hat{A},\hat{A}^+)
 \]
 such that $A \to \bar{A}$ is surjective and $\bar{A} \to \hat{A}$ is injective.
 Then $(\bar{A},\bar{A}^+)$ is Pr\"ufer by \cite{Rho91}, Proposition~3.1.1 (or \cite{KneZha}, Proposition~5.8) and $(\hat{A},\hat{A}^+)$ is the completion of $(\bar{A},\bar{A}^+)$.
 We may therefore assume that the morphism $\iota: A \to \hat{A}$ is injective.

 By \cite{KneZha}, Theorem~5.2, $(1) \Leftrightarrow (2)$, a ring extension $B \hookrightarrow R$ is Pr\"ufer if and only if every $R$-overring of~$B$ is integrally closed in~$R$.
 We have mutually inverse bijections
 \[
  \begin{tikzcd}[column sep=large]
   \{\text{open subrings of~$A$}\}	\ar[r, shift left=0.1cm,"B \mapsto \hat{B}"]	& \{\text{open subrings of~$\hat{A}$}\}.	\ar[l, shift left=0.1cm,"C \cap A \mapsfrom C"]
  \end{tikzcd}
 \]
 The subsequent lemma shows that this correspondence restricts to a bijection of the open, integrally closed subrings of~$A$ with the open, integrally closed subrings of~$\hat{A}$.
 Since~$A^+$ is open and integrally closed in~$A$, we obtain a bijection of the integrally closed $A$-overrings of~$A^+$ with the integrally closed $\hat{A}$-overrings of~$\hat{A}^+$.
 In particular, an $\hat{A}$-overring~$C$ of~$\hat{A}^+$ is integrally closed in~$\hat{A}$ if and only if $C \cap A$ is integrally closed in~$A$.
 This finishes the proof as all $A$-overrings of~$A^+$ are integrally closed in~$A$ by assumption.
\end{proof}

\begin{lemma}
 For any linearly topologized ring~$A$ with completion $\sigma : A \to \hat{A}$ the mutually inverse bijections
 \[
  \begin{tikzcd}[column sep=large]
   \{\text{open subrings of~$A$}\}	\ar[r, shift left=0.1cm,"B \mapsto \hat{B}"]	& \{\text{open subrings of~$\hat{A}$}\}.	\ar[l, shift left=0.1cm,"\sigma^{-1}(C) \mapsfrom C"]
  \end{tikzcd}
 \]
 establish a correspondence of the open, integrally closed subrings.
\end{lemma}

\begin{proof}
 The argument is taken from the proof of Lemma~2.4.3 in \cite{HuHabil}.
 The only nontrivial assertion we have to check is that the completion~$\hat{B}$ of any open, integrally closed subring~$B$ of~$A$ is integrally closed.
 Denote by~$C$ the integral closure of~$\hat{B}$ in~$\hat{A}$.
 This is an open subring of~$\hat{A}$.
 Take an element $c \in C$.
 In order to show that $c \in \hat{B}$, it suffices to check that for any open neighborhood~$U$ of~$c$ in~$C$ we have
 \[
  U \cap \sigma(B) \ne \varnothing.
 \]
 Since~$\sigma(A)$ is dense in~$\hat{A}$, we can find $a \in A$ with $\sigma(a) \in U$.
 Being contained in~$C$ the element~$\sigma(a)$ satisfies an integral equation
 \[
  \sigma(a)^n + \hat{b}_{n-1} \sigma(a)^{n-1} + \ldots + \hat{b}_0 = 0
 \]
 with $\hat{b}_i \in \hat{B}$.
 As~$\hat{B}$ is open, we can approximate the~$\hat{b}_i$ by elements of the form $\sigma(b_i)$ with $b_i \in B$ such that
 \[
  \sigma(a)^n + \sigma(b_{n-1}) \sigma(a)^{n-1} + \ldots \sigma(b_0) \in \hat{B}.
 \]
 Together with $B = \sigma^{-1}(\hat{B})$ this implies the existence of an element $b \in B$ such that
 \[
  a^n + b_{n-1} a^{n-1} + \ldots + (b_0-b) = 0
 \]
 We conclude that $a \in B$ and thus $\sigma(a) \in U \cap \sigma(B)$.
\end{proof}

\subsection{A flatness criterion}

For this subsection we fix a local Huber pair $(A,A^+)$.
We denote by~$\m$ the maximal ideal of~$A$.
It is contained in~$A^+$ and $A^+/\m$ is a valuation ring.
Hence, every proper ideal of~$A$ is contained in~$A^+$.
We write $|\cdot|$ for the valuation of~$A$ corresponding to $A^+/\m$.

We want to investigate whether an $A^+$-module~$M^+$ is flat if its base change to~$A$ is flat.
To this end we examine for an ideal $a^+ \subseteq A^+$ the vanishing of $\Tor^{A^+}_1(M^+,A^+/\mathfrak{a}^+)$.

\begin{lemma} \label{idealscomingfromA}
 Let~$\mathfrak{a}$ be a proper ideal of~$A$.
 Let~$M^+$ be an $A^+$-module such that $M := M^+ \otimes_{A^+} A$ is a flat $A$-module.
 Then
 \[
  \Tor^{A^+}_1(M^+,A^+/\mathfrak{a}) = 0.
 \]
\end{lemma}

\begin{proof}
 Consider the commutative diagram
 \begin{equation} \label{tensorwithM+}
  \begin{tikzcd}
   \mathfrak{a} \otimes_{A^+} M^+	\ar[r]			\ar[d]	& M^+	\ar[d]	\\
   \mathfrak{a} \otimes_A M		\ar[r,hookrightarrow]		& M.
  \end{tikzcd}
 \end{equation}
 The lower horizontal map is injective as~$M$ is a flat~$A$-module.
 As $(A,A^+)$ is local, $A^+ \to A$ is a localization ($A = A^+_{\m_A}$).
 Hence, it is flat and thus the homomorphism
 \[
  \mathfrak{a} \otimes_{A^+} A \to A
 \]
 is injective.
 Its image is $A \cdot \mathfrak{a} = \mathfrak{a}$.
 We obtain an isomorphism $\mathfrak{a} \otimes_{A^+} A \to \mathfrak{a}$ whose inverse~$\varphi$ is given by $a \mapsto a \otimes 1$.
 Tensoring~$\varphi$ with~$M^+$ yields the left vertical map in diagram~(\ref{tensorwithM+}), which is thus an isomorphism.
 We conclude that the upper horizontal map is injective.
 Hence,
 \[
  \Tor^{A^+}_1(M^+,A^+/\mathfrak{a}) = \ker(\mathfrak{a} \otimes_{A^+} M^+ \to M^+) = 0.
 \]
\end{proof}

\begin{lemma} \label{maximalidealpower}
 Let~$\mathfrak{a}^+$ be an ideal of~$A^+$.
 Let~$M^+$ be an $A^+$-module such that $M := M^+ \otimes_{A^+} A$ is a flat $A$-module and $M^+/\m M^+$ is torsion free over~$A^+/\m$.
 Then
 \[
  \Tor^{A^+}_1(M^+,A^+/(\m^n + \mathfrak{a}^+)) = 0.
 \]
 for all $n \geq 1$.
\end{lemma}

\begin{proof}
 Consider the commutative diagram
 \begin{equation}
  \begin{tikzcd} \label{diagramflatness}
   \m^n \otimes_{A^+} M^+				\ar[r,"\sim"]		\ar[d,hookrightarrow]		& \m^n M^+	\ar[d,hookrightarrow]		\\
   (\m^n + \mathfrak{a}^+) \otimes_{A^+} M^+		\ar[r]			\ar[d,twoheadrightarrow]	& M^+		\ar[d,twoheadrightarrow]	\\
  (\m^n + \mathfrak{a}^+)/\m^n \otimes_{A^+} M^+	\ar[r,hookrightarrow]					& M^+/\m^n M^+.
  \end{tikzcd}
 \end{equation}
 The upper horizontal map is an isomorphism by \cref{idealscomingfromA}.
 This implies that the upper left vertical map is injective.
 Let us show that the lower horizontal map is injective.
 Since
 \[
  (\m^n + \mathfrak{a}^+)/\m^n \otimes_{A^+} M^+ \to (\m^n + \mathfrak{a}^+)/\m^n \otimes_{A^+/\m^n} M^+/\m^nM^+
 \]
 is an isomorphism, this comes down to showing that $M^+/\m^n M^+$ is a flat $A^+/\m^n$-module.
 If~$n=1$, this is true as~$A^+/\m$ is a valuation ring and $M^+/\m M^+$ is torsion free, hence flat.
 The case $n > 1$ follows from the case $n = 1$ by \cite[Tag 051C]{stacks-project}.
 Note that the assumption
 \[
  \Tor_1^{A^+}(M^+,A^+/\m) = 0
 \]
 in \cite[Tag 051C]{stacks-project} is satisfied by \cref{idealscomingfromA}.
 We conclude that the lower horizontal map in diagram~(\ref{diagramflatness}) is injective.
 A diagram chase now shows the injectivity of the middle horizontal map, which concludes the proof.
\end{proof}

The following lemma is a variant of the Artin-Rees lemma for local Huber pairs.

\begin{lemma} \label{ArtinRees}
 Assume that~$A$ is noetherian.
 Let~$\mathfrak{a}$ be an ideal of~$A$ and $N^+ \subseteq M^+$ finite $A^+$-modules.
 Set $M := M^+ \otimes_{A^+} A$ and $N := N^+ \otimes_{A^+} A$ and assume that $M^+ \to M$ is injective.
 Then there is $K \in \N$ such that for all $n > K$
 \[
  \mathfrak{a}^n M^+ \cap N^+ = \mathfrak{a}^{n-K} (\mathfrak{a}^K M^+ \cap N^+) = \mathfrak{a}^n M \cap N = \mathfrak{a}^{n-K} (\mathfrak{a}^K M \cap N).
 \]
\end{lemma}

\begin{proof}
 As $A^+ \to A$ is flat, the natural map $N \to M$ is injective and we view $N$, $M^+$ and~$N^+$ as submodules of~$M$.
 For positive integers $n > K$ consider the diagram
 \[
  \begin{tikzcd}
   \mathfrak{a}^{n-K} (\mathfrak{a}^K M^+ \cap N^+)	\ar[r,hookrightarrow]	\ar[d,hookrightarrow]	& \mathfrak{a}^n M^+ \cap N^+	\ar[d,hookrightarrow]	\\
   \mathfrak{a}^{n-K} (\mathfrak{a}^K M \cap N)		\ar[r,hookrightarrow]				& \mathfrak{a}^n M \cap N.
  \end{tikzcd}
 \]
 For $K$ big enough the lower horizontal inclusion is the identity by the Artin-Rees lemma.
 Moreover, since $A^+ \to A$ is a localization and $\mathfrak{a}$ is an ideal not only of~$A^+$ but of~$A$, the left vertical map is the identity.
 This implies that the upper horizontal map and the right vertical map are the identity.
\end{proof}

\begin{proposition} \label{flatness}
 Let $(B,B^+)$ be a Pr\"ufer Huber pair such that~$B$ is noetherian.
 Let~$M^+$ be a torsion free $B^+$-module such that $M := M^+ \otimes_{B^+} B$ is flat over~$B$.
 Then~$M^+$ is flat.
\end{proposition}

\begin{proof}
 It suffices to show that~$M^+_{\m^+}$ is a flat $B^+_{\m^+}$-module for every maximal ideal~$\m^+$ of~$B^+$.
 By hypothesis, the pair $(B_{\m^+},B^+_{\m^+})$ is a local Huber pair.
 In particular, $B_{\m^+} = B_{\m}$ for some prime ideal~$\m$ of~$B$.
 As the assumptions are stable under localization, we may assume that $(B,B^+)$ is local right away.

 We observe that $M^+/\m M^+$ is torsion free over $B^+/\m$.
 Indeed, take $b^+ \in B^+ \setminus \m$ and $m^+ \in M^+$ satisfying $b^+m^+ \in \m$.
 As~$B$ is local with maximal ideal~$\m$, $b^+$ is a unit in~$B$.
 Since~$\m$ is a $B$-module, we obtain
 \[
  m^+ = (b^+)^{-1}(b^+m^+) \in \m.
 \]
 Let~$\mathfrak{b}^+ \subseteq B^+$ be a finitely generated ideal.
 We have to show that
 \[
  \mathfrak{b}^+ \otimes_{B^+} M^+ \to M^+
 \]
 is injective.
 For $n \geq 1$ consider the following diagram of short exact sequences:
 \[
  \begin{tikzcd}
   0	\ar[r]	& \mathfrak{b}^+ \cap \m^n	\ar[r]	\ar[d,hookrightarrow]	& \mathfrak{b}^+ \oplus \m^n	\ar[r]	\ar[d,hookrightarrow]	& \mathfrak{b}^+ + \m^n	\ar[r]	\ar[d,hookrightarrow]	& 0	\\
   0	\ar[r]	& B^+				\ar[r]				& B^+ \oplus B^+		\ar[r]				& B^+			\ar[r]				& 0.
  \end{tikzcd}
 \]
 Tensoring with~$M^+$ we obtain
 \[
  \begin{tikzcd}[column sep=tiny]
		& (\mathfrak{b}^+ \cap \m^n) \otimes_{B^+} M^+	\ar[r]	\ar[d]	& \mathfrak{b}^+ \otimes_{B^+} M^+ \oplus \m^n \otimes_{B^+} M^+	\ar[r]	\ar[d]	& (\mathfrak{b}^+ + \m) \otimes_{B^+} M^+	\ar[r]	\ar[d]	& 0	\\
   0	\ar[r]	& M^+						\ar[r]		& M^+ \oplus M^+						\ar[r]		& M^+						\ar[r]		& 0.
  \end{tikzcd}
 \]
 Since $\m^n \otimes_{B^+} M^+ \to M^+$ and $(\mathfrak{b}^+ + \m^n) \otimes_{B^+} M^+ \to M^+$ are injective by \cref{maximalidealpower}, the snake lemma implies that
 \[
  \ker\big((\mathfrak{b}^+ \cap \m^n) \otimes_{B^+} M^+ \to M^+\big) \to \ker\big(\mathfrak{b}^+ \otimes_{B^+} M^+ \to M^+\big)
 \]
 is surjective.
 We now apply \cref{ArtinRees} to the finite $B^+$-modules $\mathfrak{b}^+ \subseteq B^+$.
 Setting $\mathfrak{b} = \mathfrak{b}^+ \otimes_{B^+} B$ there is $N \in \N$ such that for all $n > N$
 \[
  \m^n \cap \mathfrak{b}^+ = \m^{n-N} (\m^N \cap \mathfrak{b}^+) = \m^n \cap \mathfrak{b} = \m^{n-N} (\m^N \cap \mathfrak{b}).
 \]
 The ideal $\m^n \cap \mathfrak{b}^+$ of~$B^+$ is thus also an ideal of~$B$
  and by \cref{idealscomingfromA} we obtain
 \[
  \ker\big((\mathfrak{b}^+ \cap \m^n) \otimes_{B^+} M^+ \to M^+\big) = 0,
 \]
 which implies that
 \[
  \ker\big(\mathfrak{b}^+ \otimes_{B^+} M^+ \to M^+\big) = 0.
 \]
\end{proof}

\begin{remark}
 \begin{enumerate}
  \item In case $(A,A^+)$ is microbial the results of \cref{flatness} have been shown in \cite{FuKa18}, Chapter~0, Proposition~8.7.12.
  \item	The flatness criterion \cref{flatness} in case~$M^+$ is a $B^+$-algebra resembles the one given in \cite{Tem11}, Lemma~2.3.1 (iii).
		However, in our application $M^+$ is not of finite type, in general.
		This impedes the application of Raynaud-Gruson flattening (\cite{RG71}) in contrast to the situation in \cite{Tem11}.
 \end{enumerate}
\end{remark}

\subsection{Cartesian coverings of Huber pairs} \label{section_cartesian}

\begin{definition}
 A homomorphism
 \[
  (A,A^+) \to (B,B^+)
 \]
 of Huber pairs is called \emph{Cartesian} if the natural homomorphism
 \[
  B^+ \otimes_{A^+} A \to B
 \]
 is bijective.
 We say that an \'etale covering $(\cU_i \to \Spa(A,A^+))_{i \in I}$ is Cartesian if for every $i \in I$ there is an \'etale Cartesian homomorphism $(A,A^+) \to (B,B^+)$ of Huber pairs such that $\cU_i = \Spa(B,B^+)$.
\end{definition}

It is straightforward to see that the composition of two Cartesian homomorphisms is again Cartesian.

\begin{proposition} \label{stronglyetale+etale}
 Let $(A,A^+)$ be a Pr\"ufer Huber pair.
 Let $(A,A^+) \to (B,B^+)$ be a Cartesian, strongly \'etale homomorphism.
 Then $A^+ \to B^+$ is \'etale.
\end{proposition}

\begin{proof}
 Let $\m^+$ be a maximal ideal of~$A^+$.
 In order to show that $A^+ \to B^+$ is \'etale at~$\m^+$ we can base change to $A^+_{\m^+}$.
 As $(A,A^+)$ is Pr\"ufer, there is a unique point $x \in \cX := \Spa(A,A^+)$ such that $\cO_{\cX,x} = A_{\m^+}$ and $\cO_{\cX,x}^+ = A^+_{\m^+}$.
 Therefore, base changing $\cY \to \cX$ to $\cX_x$ induces the base change of $A^+ \to B^+$ to~$A^+_{\m+}$.
 We may thus assume that $(A,A^+)$ is local such that~$\m^+$ is the maximal ideal of~$A^+$.
 The assertion then follows from \cref{local_henselian}.
\end{proof}

\begin{lemma} \label{integralcartesian}
 Let $(A,A^+)$ be a complete Pr\"ufer Huber pair.
 Then, every integral homomorphism $(A,A^+) \to (B,B^+)$ is Cartesian and $(B,B^+)$ is Pr\"ufer.
\end{lemma}

\begin{proof}
 By definition $A \to B$ is integral and~$B^+$ is the integral closure of~$A^+$ in~$B$.
 Hence, $B$ is generated by~$B^+$ and the image of~$A$ (\cite{KneZha}, Theorem~5.9).
 By \cite{KneZha}, Proposition~3.10,  $B^+ \to B$ and $B^+ \to B^+ \otimes_{A^+} A$ are weakly surjective.
 Moreover, both are injective (the injectivity of $B^+ \to B^+ \otimes_{A^+} A$ follows from the injectivity of $B^+ \to B$).
 Therefore, by \cite{KneZha}, Corollary~3.16  the surjective homomorphism $B^+ \otimes_{A^+} A \to B$ is injective.
\end{proof}

\begin{lemma} \label{quasifinitelocalization}
 Let $(A,A^+)$ be a Pr\"ufer Huber pair with~$A$ noetherian and
 \[
  (A,A^+) \to (B,B^+)
 \]
 a Cartesian homomorphism such that $\Spec B$ is quasi-finite and essentially of finite type over $\Spec A$.
 Then $(B,B^+)$ is Pr\"ufer, too.
\end{lemma}

\begin{proof}
 By Zariski's main theorem $A \to B$ factors as $A \to B_0 \to B$ with $B_0/A$ finite and $B/B_0$ a localization.
 Denote by~$B_0^+$ the integral closure of~$A^+$ in~$B_0$.
 Since~$B^+$ is integrally closed in~$B$, we obtain a diagram
 \[
  \begin{tikzcd}
   B				& B_0	\ar[l,hookrightarrow,"\text{loc.}"]				& A	\ar[l,"\varphi","\text{finite}"']				\\
   B^+	\ar[u,hookrightarrow]	& B_0^+ \ar[l,hookrightarrow]			\ar[u,hookrightarrow]	& A^+.	\ar[l,"\varphi^+","\text{int.}"']	\ar[u,hookrightarrow]	\\
  \end{tikzcd}
 \]
 By~\cref{integralcartesian} the Huber pair $(B_0,B_0^+)$ is Pr\"ufer and $A \otimes_{A^+} B_0^+ \to B_0$ is bijective.
 This implies that~$(B_0,B_0^+) \to (B,B^+)$ is Cartesian.

 If~$A$ is noetherian, so is~$B_0$.
 Hence, \cref{flatness} implies that $B_0^+ \to B^+$ is flat and thus weakly surjective by \cite{KneZha}, Proposition~4.5.
 The result now follows from \cite{KneZha}, Theorem~5.10.
\end{proof}

\subsection{Laurent coverings and Zariski cohomology} \label{section_Laurent}

\begin{definition}
 Let $(A,A^+)$ be a Huber pair.
 A Laurent covering of $\Spa(A,A^+)$ is a covering by rational open subsets of the form
 \[
  \Spa(A,A^+) = \bigcup_{\alpha_i \in \{\pm 1\}} R (f_1^{\alpha_1} ,\ldots, f_n^{\alpha_n})
 \]
 with $f_1,\ldots,f_n \in A$.
\end{definition}

\begin{lemma} \label{cartesianLaurentcover}
 Let $(A,A^+)$ be a Huber pair such that $A^+ \to A$ is weakly surjective.
 Then for any $f \in A$ the Laurent covering
 \[
  R(\frac{f}{1}) \cup R(\frac{1}{f}) = \Spa(A,A^+)
 \]
 is Cartesian.

 Denote by $A^+[\frac{1}{f}]$ the subring of~$A_f$ generated by the image of~$A^+$ and~$1/f$.
 If in addition $(A,A^+)$ is Pr\"ufer and~$A$ is noetherian, $A^+[f]$ and $A^+[\frac{1}{f}]$ are integrally closed in~$A$ and~$A_f$, respectively, i.e. $(A,A^+[f])$ and $(A_f,A^+[\frac{1}{f}])$ are Huber pairs and
 \[
  R(\frac{f}{1}) = \Spa(A,A^+[f]),	\qquad	R(\frac{1}{f}) = \Spa(A_f,A^+[\frac{1}{f}]).
 \]
\end{lemma}

\begin{proof}
 We only treat $R(\frac{1}{f})$.
 The examination of $R(\frac{f}{1})$ is similar (and even easier).
 We have
 \[
  R(\frac{1}{f}) = \Spa(A_f,A^+_f),
 \]
 where~$A^+_f$ denotes the integral closure of $A^+[\frac{1}{f}]$ in~$A_f$.
 In order to show that $R(\frac{1}{f}) \to \Spa(A,A^+)$ is Cartesian we need to check that the natural homomorphism
 \[
  \varphi: A \otimes_{A^+} A^+_f \to A_f
 \]
 is an isomorphism.
 The surjectivity of~$\varphi$ is obvious.
 Consider the diagram
 \[
  \begin{tikzcd}
   A_f	\\
	& A^+_f \otimes_{A^+} A	\ar[ul,"\varphi"']							& A					\ar[l]	\ar[ull,bend right=20]	\\
	& A^+_f			\ar[uul,hookrightarrow,bend left=20,"\beta"]	\ar[u,"\alpha'"]	& A^+.	\ar[u,hookrightarrow,"\alpha"']	\ar[l]
  \end{tikzcd}
 \]
 As $\alpha$ is weakly surjective, so are~$\alpha'$ and~$\beta$ (see \cite{KneZha}, Proposition~3.10).
 Moreover,~$\alpha'$ is injective because~$\beta$ is injective.
 We conclude by \cite{KneZha}, Corollary~3.16 that~$\varphi$ is injective.

 Assume now that $(A,A^+)$ is Pr\"ufer and~$A$ is noetherian.
 As the image of~$A^+$ in~$A_f$ is Pr\"ufer in the image of~$A$ in~$A_f$ by \cite{KneZha}, Proposition~5.7, we may replace~$A^+$ and~$A$ by their images in~$A_f$ and assume henceforth that $A \to A_f$ is injective.
 The same argument as above shows that
 \[
  A \otimes_{A^+} A^+[\frac{1}{f}] \cong A_f.
 \]
 By \cref{flatness}, $A^+ \to A^+[\frac{1}{f}]$ is flat.
 Moreover, $A^+ \to A \to A_f$ is weakly surjective.
 Hence, $A^+ \to A^+[\frac{1}{f}]$ is weakly surjective by \cite{KneZha}, Proposition~4.5.
 Since $A_f$ is generated by~$A$ and $A^+[\frac{1}{f}]$, \cite{KneZha}, Theorem~5.10 implies that $A^+[\frac{1}{f}]$ is Pr\"ufer in~$A_f$.
 In particular, $A^+[\frac{1}{f}]$ is integrally closed in~$A_f$.
\end{proof}

\begin{corollary} \label{cartesianneighborhoodbasis}
 Let $(A,A^+)$ be a Pr\"ufer Huber pair.
 Then every closed point of $\Spa(A,A^+)$ has a basis of Cartesian affinoid neighborhoods.
\end{corollary}

\begin{proof}
 Let us first convince ourselves that it suffices to show that every covering of $\cX = \Spa(A,A^+)$ has a Cartesian refinement.
 Assume the latter is the case.
 Let $x \in \cX$ be a closed point and $\cU \subseteq \cX$ an open neighborhood of~$x$.
 Then
 \[
  \cX = \cU \cup (\cX \setminus \{x\})
 \]
 is a covering of~$\cX$.
 By assumption, it has a refinement by a Cartesian covering $(\cV_i \subseteq \cX)_{i \in I}$.
 So there is $i \in I$ such that $x \in \cV_i \subseteq \cU$.
 
 By \cite{Hu94}, Lemma~2.6, every open covering of $\cX := \Spa(A,A^+)$ is dominated by a rational covering, i.e., a covering of the form
 \[
  \cX = \Spa(A,A^+) = \bigcup_{j=1}^m R\big(\frac{g_1,\ldots,g_m}{g_j}\big)
 \]
 with $g_1,\ldots,g_m \in A$ such that $g_1 A + \ldots + g_m A = A$.
 Arguing as in \cite{BGR}, \S~8.2.2 we obtain the following two assertions.
 \begin{enumerate}[(i)]
  \item Let $(\cU_i \subseteq \cX)_{i \in I}$ be a rational covering.
		Then there is a Laurent covering $(\cV_j \to \cX)_{j \in J}$ such that for every $j \in J$ the covering $(\cU_i \cap \cV_j \to \cV_j)_{i \in I}$ is a rational covering generated by units in $\cO_{\cX}(\cV_j)$.
  \item Every Cartesian covering of~$\cX$ that is generated by units has a refinement which is a Laurent covering.
 \end{enumerate}
 In total we get that every covering of~$\cX$ has a refinement of the form
 \[
  (\cU_{ij} \subseteq \cU_i \subseteq \cX)_{(i,j) \in I \times J}
 \]
 such that $(\cU_i \subseteq \cX)_{i \in I}$ is a Laurent covering and for each $i \in I$ also $(\cU_{ij} \subseteq \cU_i)_{j \in J}$ is a Laurent covering.
 In particular, for every $(i,j) \in I \times J$ the inclusions $\cU_i \subseteq \cX$ and $\cU_{ij} \subseteq \cU_i$ are Cartesian by \cref{cartesianLaurentcover}.
 Hence, $\cU_{ij} \subseteq \cX$ is Cartesian and thus the covering $(\cU_{ij} \subseteq \cX)_{(i,j)}$ is Cartesian.
\end{proof}

\begin{lemma} \label{acyclicLaurentcover}
 Let $(A,A^+)$ be a complete Pr\"ufer Huber pair.
 Assume that either~$A$ is a strongly noetherian Tate ring or the topology of~$A$ is discrete and~$A$ is noetherian.
 Let~$\mathfrak{U}$ be a Laurent covering of $\cX = \Spa(A,A^+)$.
 Then the \v{C}ech cohomology groups
 \[
  \check{\mathrm{H}}^i(\mathfrak{U},\cO_{\cX}^+)
 \]
 vanish for $i \geq 1$.
\end{lemma}

\begin{proof}
 Using \cite{BGR}, 8.1.4 Corollary~4 and induction, this comes down to showing that
 \[
  0 \to A^+ \to \cO_{\cX}^+(R ( \frac{f}{1} )) \oplus \cO_{\cX}^+(R ( \frac{1}{f} )) \overset{\alpha}{\to} \cO_{\cX}^+(R ( \frac{f}{1},\frac{1}{f} )) \to 0
 \]
 is exact for every $f \in A$.
 We know already that~$\cO_{\cX}^+$ is a sheaf.
 Hence, we are left with showing the surjectivity of~$\alpha$.
 By \cref{cartesianLaurentcover} we have
 \[
  R ( \frac{f}{1} ) = \Spa(A,A^+[f]),	\quad R ( \frac{1}{f} ) = \Spa(A_f,A^+[\frac{1}{f}]),	\quad R ( \frac{f}{1},\frac{1}{f} ) = \Spa(A_f,A^+[f,\frac{1}{f}]).
 \]
 In case the topology of~$A$ is discrete, the surjectivity of~$\alpha$ is now obvious.
 In case~$A$ is a strongly noetherian Tate algebra we use the following identifications (see II.1 in the proof of Theorem~2.5 in \cite{Hu94}):
 \[
  A \langle \frac{f}{1} \rangle = A \langle X \rangle /(f-X),	\quad	A \langle \frac{1}{f} \rangle = A \langle Y \rangle /(1-fY),	\quad	A \langle \frac{f}{1},\frac{1}{f} \rangle = A \langle X,X^{-1} \rangle /(f-X).
 \]
 Then $\cO_{\cX}^+(R ( \frac{f}{1} ))$ is the closure of $A^+[f]$ in $A \langle X \rangle /(f-X)$, i.e. equal to
 \[
  \{ \sum_i b_i X^i \in A \langle X \rangle \mid b_i \in A^+ \} /(f-X).
 \]
 Similarly
 \begin{align*}
  \cO_{\cX}^+(R ( \frac{1}{f} ))				&= \{ \sum_i b_i Y^i \in A \langle Y \rangle \mid b_i \in A^+ \} /(1-fY)	\\
  \cO_{\cX}^+(R ( \frac{f}{1},\frac{1}{f} ))	&= \{ \sum_i b_i X^i \in A \langle X,X^{-1} \rangle \mid b_i \in A^+ \} /(f-X).
 \end{align*}
 Now also in this case the surjectivity of~$\alpha$ can be checked explicitly.
\end{proof}

\begin{proposition} \label{Zariskiacyclic}
 Let $(A,A^+)$ be a complete Pr\"ufer Huber pair.
 Assume that either~$A$ is a strongly noetherian Tate ring or the topology of~$A$ is discrete and~$A$ is noetherian.
 Then, setting $\cX = \Spa(A,A^+)$,
 \[
  H^i(\cX,\cO_{\cX}^+) = 0.
 \]
 for all $i > 0$.
\end{proposition}

\begin{proof}
 It suffices to prove that for any affinoid open covering of an affinoid open subspace of $\cX$, the \v{C}ech cohomology of~$\cO_{\cX}^+$ is trivial.
 As in \cite{BGR}, Chapter~8.2, Proposition~5, we further restrict to Laurent coverings of open affinoid subspaces of~$\cX$.
 But all Laurent coverings are $\cO_{\cX}^+$-acyclic by \cref{acyclicLaurentcover}.
\end{proof}

\section{Strongly \'etale cohomology} \label{stronglyetalecohomology}

If~$\cX$ is an analytic adic space, the additive group~$\GG_a$ is a sheaf for the \'etale site of~$\cX$ by \cite{Hu96}, (2.2.5).
In case~$\cX$ is a discretely ringed adic space this follows from the corresponding statement for schemes.
In particular, in both cases, $\GG_a$ is a sheaf for the strongly \'etale and the tame site.
Then, also the subpresheaf~$\GG_a^+$ of~$\GG_a$ defined by
\[
 (\cY \to \cX) \mapsto \cO_{\cY}^+(\cY)
\]
is a sheaf.

In the following we say that an adic space~$\cX$ is \emph{locally noetherian} if it is locally of the form $\Spa(A,A^+)$ such that the completion of~$A$ is noetherian.
We say that~$\cX$ is noetherian if in addition~$\cX$ is quasi-compact.

\begin{lemma}
 Let
 \[
  \varphi: \cX \to \Spa(A,A^+)
 \]
 be an \'etale covering of the complete, noetherian, Pr\"ufer affinoid adic space $\Spa(A,A^+)$.
 Then there is a morphism
 \[
  \psi:\cY \to \cX,
 \]
 which is a finite coproduct of open immersions such that $\varphi \circ \psi$ is an (affinoid) Cartesian \'etale covering.
\end{lemma}

\begin{proof}
 We may assume that $\cX = \Spa(B,B^+)$ for an \'etale homomorphism $(A,A^+) \to (B,B^+)$.
 Using Zariski's main theorem and \cite{Hu96}, Corollary~1.7.3~ii), we factor $\varphi$ as
 \[
  \Spa(B,B^+) \overset{\iota}{\longrightarrow} \Spa(C,C^+) \overset{\pi}{\longrightarrow} \Spa(A,A^+)
 \]
 with an open immersion~$\iota$ and a finite morphism~$\pi$.
 \cref{integralcartesian} implies that~$\pi$ is Cartesian and $(C,C^+)$ is Pr\"ufer.
 Now it suffices to show that every closed point~$x$ of~$\Spa(C,C^+)$ lying in the image of~$\iota$ has an open affinoid neighborhood $\Spa(C,C^+)$ coming from a Cartesian homomorphism $(A,A^+) \to (C,C^+)$.
 This follows from \cref{cartesianneighborhoodbasis}.
\end{proof}

\begin{corollary} \label{refinement}
 Every tame covering and every strongly \'etale covering of a noetherian Pr\"ufer affinoid adic space $\Spa(A,A^+)$ has a Cartesian refinement.
\end{corollary}

\begin{proposition} \label{stronglyetaleacyclic-discrete}
 Let $(A,A^+)$ be a Pr\"ufer Huber pair such that~$A$ is noetherian and equipped with the discrete topology.
 Then for all $i > 0$,
 \[
  H^i(\Spa(A,A^+)_{\set},\GG_a^+) = 0.
 \]
\end{proposition}

\begin{proof}
 Let~$\mathcal{B}$ be the category of Cartesian strongly \'etale morphisms of affinoid adic spaces
 \[
  \Spa (B,B^+) \to \Spa(A,A^+).
 \]
 It has fiber products and becomes a site by defining coverings of $\Spa(B,B^+)$ to be the Cartesian strongly \'etale coverings of $\Spa (B,B^+)$.
 By \cref{refinement} we can compute the cohomology groups $H^q(\Spa(A,A^+)_{\set},\GG_a^+)$ in~$\mathcal{B}$.

 We show that $\GG_a^+$ is flasque on~$\mathcal{B}$.
 In order to do so we prove that for every covering
 \[
  \Spa(C,C^+) \to \Spa(B,B^+)
 \]
 in~$\mathcal{B}$ the associated \v{C}ech complex for the sheaf~$\GG_a^+$ is exact.
 The fact that $\Spa(C,C^+) \to \Spa(B,B^+)$ is Cartesian implies that the diagram
 \[
  \begin{tikzcd}
   C \otimes_B \ldots \otimes_B C			& B	\ar[l]				\\
   C^+ \otimes_{B^+} \ldots \otimes_{B^+} C^+	\ar[u]	& B^+	\ar[l]	\ar[u,hookrightarrow]
  \end{tikzcd}
 \]
 is Cartesian.
 Since $\Spec C^+ \to \Spec B^+$ is an \'etale covering by \cref{stronglyetale+etale}, so is $\Spec C^+ \otimes_{B^+} \ldots \otimes_{B^+} C^+ \to \Spec B^+$.
 In particular, it is flat and thus the left vertical arrow is injective.
 Moreover, taking integral closures commutes with \'etale base change.
 Therefore, $C^+ \otimes_{B^+} \ldots \otimes_{B^+} C^+$ is integrally closed in $C \otimes_B \ldots \otimes_B C$.
 By construction of the fiber product for adic spaces, this is equivalent to saying that
 \[
  \Spa(C,C^+) \times_{\Spa(B,B^+)} \ldots \times_{\Spa(B,B^+)} \Spa(C,C^+) = \Spa(C \otimes_B \ldots \otimes_B C,C^+ \otimes_{B^+} \ldots \otimes_{B^+} C^+).
 \]
 The \v{C}ech complex for~$\GG_a^+$ thus equals the Amitsur complex
 \[
  \begin{tikzcd}
   0	\ar[r]	& B^+	\ar[r]	& C^+	\ar[r]	& C^+ \otimes_{B^+} C^+	\ar[r]	& C^+ \otimes_{B^+} C^+ \otimes_{B^+} C^+	\ar[r]	& \ldots
  \end{tikzcd}
 \]
 This complex is exact as $B^+ \to C^+$ is faithfully flat.
 Hence, $\GG_a^+$ is flasque on~$\mathcal{B}$.
 In particular,
 \[
  H^i(\Spa(A,A^+)_{\set},\GG_a^+) = 0.
 \]
 \end{proof}

%

\begin{proposition} \label{non-archimedeanfieldacyclic}
 Let $(A,A^+)$ be a complete Pr\"ufer Huber pair such that $A$ is a non-Archimedean field.
 Then
 \[
  H^i(\Spa(A,A^+)_{\set},\GG_a^+) = 0.
 \]
 for all $i \ge 1$.
\end{proposition}

\begin{proof}
 Set $\cX = \Spa(A,A^+)$.
 Note first that $(A,A^{\circ})$ (where $A^{\circ}$ denotes the power bounded elements) is henselian by Hensel's lemma for non-Archimedean fields and that $\Spa(A,A^{\circ})$ consists of a single point.
 Consider an \'etale morphism $\cY \to \cX$ with~$\cY$ affinoid.
 The base change of~$\cY$ to $\Spa(A,A^{\circ})$ is a disjoint union of affinoid adic spaces of the form $(B,B^{\circ})$ such that~$B$ is a finite separable extension of~$A$.
 Since the set of generalizations of an analytic point of an adic space is totally ordered by specialization, every connected component of~$\cY$ is of the form $(B,B^+)$ with~$B$ as above.
 In particular, $B$ is a complete, non-Archimedean field.
 Furthermore, $B^+$ is a $B$-overring of the integral closure of~$A^+$ in~$B$, hence Pr\"ufer.

 Let~$\mathcal{B}$ be the full subcategory of~$\cX_{\set}$ whose objects are the strongly \'etale morphisms $\cY \to \cX$ such that~$\cY$ is affinoid.
 We can compute the cohomology of~$\cX$ in~$\mathcal{B}$.
 We show that~$\GG_a^+$ is flasque on~$\mathcal{B}$.

 Let~$\cY \to \cX$ be in~$\mathcal{B}$ and $\cZ \to \cY$ a covering of~$\cY$.
 We may assume that~$\cY$ is the adic spectrum of a complete Pr\"ufer Huber pair $(B,B^+)$ such that $B$ is a non-Archimedean field.
 Then $\cZ = \Spa(C,C^+)$ with~$C$ finite \'etale over~$B$ and $C^+$ flat over~$B^+$ (as any torsion free module over a Pr\"ufer domain is flat).
 Since $(B,B^+) \to (C,C^+)$ is strongly \'etale, $B^+ \to C^+$ is even \'etale by \cref{comparenotionsofetale}.
 Consider the diagram
 \[
  \begin{tikzcd}
   0	\ar[r]	& B^+	\ar[r]	\ar[d,hook]	& C^+	\ar[r]	\ar[d,hook]	& C^+ \otimes_{B^+} C^+	\ar[r]	\ar[d,hook]	& C^+ \otimes_{B^+} C^+ \otimes_{B^+} C^+	\ar[r]	\ar[d,hook]	& \ldots  \\
   0	\ar[r]	& B	\ar[r]			& C	\ar[r]			& C \otimes_B C		\ar[r]			& C \otimes_B C \otimes_B C			\ar[r]			& \ldots
  \end{tikzcd}
 \]
 of exact Amitsur complexes.
 As integral closure commutes with \'etale base change, $C^+ \otimes_{B^+} \ldots \otimes_{B^+} C^+$ is integrally closed in $C \otimes_B \ldots \otimes_B C$.
 Moreover, being a finite $B$-module, $C \otimes_B \ldots \otimes_B C$ is complete and $C^+ \otimes_{B^+} \ldots \otimes_{B^+} C^+$ is an open subring.
 Therefore,
 \[
  \GG_a^+(\cZ \times_{\cY} \ldots \times_{\cY} \cZ) = C^+ \otimes_{B^+} \ldots \otimes_{B^+} C^+
 \]
 and the upper row of the above diagram is the \v{C}ech complex of $\GG_a^+$ associated with the covering $\cZ \to \cY$.
\end{proof}

\begin{corollary} \label{comparestronglyetaleZariski}
 Let~$\cZ$ be a locally noetherian adic space.
 Assume that~$\cZ$ is either discretely ringed or analytic.
 The canonical homomorphism
 \[
  H^i(\cZ,\GG_a^+) \overset{\sim}{\longrightarrow} H^i(\cZ_{\set},\GG_a^+)
  \]
 is an isomorphism for all $i \geq 0$.
\end{corollary}

\begin{proof}
  Consider the Leray spectral sequence associated with the morphism of sites
 \[
  \varphi: \cZ_{\set} \to \cZ
 \]
 We have to show that
 \[
  R^q \varphi_* \GG_a^+ = 0.
 \]
 Put differently, for every local Huber pair $(A,A^+)$ such that either~$A$ is discrete and noetherian or a non-Archimedean field we have to show that
 \[
  H^q(\Spa(A,A^+)_{\set},\GG_a^+) = 0.
 \]
 But every local Huber pair is Pr\"ufer and thus the result follows from \cref{stronglyetaleacyclic-discrete} and \cref{non-archimedeanfieldacyclic}.
\end{proof}

\section{Tame cohomology} \label{tamecohomology}

In this section we compute the tame cohomology of~$\GG_a^+$.
The main problem we face is that for a Cartesian tame morphism $\Spa(B,B^+) \to \Spa(A,A^+)$ the image of $B^+ \otimes_{A^+} B^+$ in $B \otimes_A B$ is not necessarily integrally closed.
But it turns out that the tameness condition makes the integral closure tractable.

\subsection{Computation of integral closures} \label{computationintegralclosure}

We fix a local, Cartesian, tame homomorphism $(A,A^+) \to (B,B^+)$ of strongly henselian, local, complete, Huber pairs.
Assume moreover that~$A$ is noetherian.
Since~$A$ and~$B$ are henselian, the extension $B/A$ is finite \'etale.
Let $\mid \cdot \mid$ be the valuation of~$B$ corresponding to the closed point of $\Spa(B,B^+)$.
We denote by~$\Gamma_B$ the value group of~$|\cdot|$ and by~$\Gamma_A$ the value group of the restriction of~$|\cdot|$ to~$A$.
As $A^+$ and~$B^+$ are strictly henselian and $(A,A^+) \to (B,B^+)$ is a tame morphism of complete, local Huber pairs, we can choose a presentation
\[
 B = A[T_1,\ldots,T_r]/(T_1^{m_1}-\alpha_1,\ldots,T_r^{m_r}-\alpha_r)
\]
with~$\alpha_i \in A^{\times}$ and~$m_i > 1$ prime to the residue characteristic of~$A^+$.
It induces an isomorphism
\[
 \Z/m_1\Z \times \ldots \Z/m_r\Z \to \Gamma_B/\Gamma_A, \quad (i_1,\ldots,i_r) \mapsto |T_1^{i_1}\cdot\ldots\cdot T_r^{i_r}|.
\]
For~$\gamma \in \Gamma_B/\Gamma_A$ we set
\[
 e_{\gamma} = T_1^{i_1}\cdot\ldots\cdot T_r^{i_r}
\]
with $0 \le i_k \le m_k-1$ and $|T^{i_1}\cdot\ldots\cdot T^{i_r}| \equiv \gamma \mod \Gamma_A$.
We denote the Galois group of $B/A$ by~$G$.

We write~$B_n$ for the $n$-fold tensor product of~$B$ over~$A$:
\[
 B_n = B \otimes_A \ldots \otimes_A B.
\]
Then $\{e_{\gamma_1} \otimes \ldots \otimes e_{\gamma_n}\}_{\gamma_1,\ldots,\gamma_n \in \Gamma_B/\Gamma_A}$ is a basis of~$B_n$ over~$A$.
As $(A,A^+) \to (B,B^+)$ is Cartesian and $B^+$ is flat over~$A^+$ by \cref{flatness}, the natural homomorphism $B^+ \otimes_{A^+} \ldots \otimes_{A^+} B^+ \to B_n$ is injective.
We view $B^+ \otimes_{A^+} \ldots \otimes_{A^+} B^+$ as a subring of~$B_n$ and denote its integral closure by~$B_n^+$.
Then $(B_n,B_n^+)$ is complete and $\Spa(B_n,B_n^+)$ is the $n$-fold fiber product of $\Spa(B,B^+)$ over $\Spa(A,A^+)$.
This subsection is concerned with describing~$B_n^+$ more explicitly.

\begin{proposition} \label{basistameextension}
 For an element $b = \sum_{\gamma_1,\ldots,\gamma_n} a_{\gamma_1,\ldots,\gamma_n} e_{\gamma_1} \otimes \ldots \otimes e_{\gamma_n}$ of~$B_n$ and $\delta \in \Gamma_B$
  the following are equivalent:
 \begin{enumerate}[\rm (i)]
  \item $|b(x)| \le \delta$ for all $x \in \Spa(B_n,B_n^+)$.
  \item $|a_{\gamma_1,\ldots,\gamma_n}| \le \delta |e_{\gamma_1}\cdot\ldots\cdot e_{\gamma_n}|^{-1}$ for all $\gamma_1,\ldots,\gamma_n \in \Gamma_B/\Gamma_A$.
 \end{enumerate}
\end{proposition}

\begin{proof}
 For an $(n-1)$-tuple $\underline{\sigma} = (\sigma_1,\ldots,\sigma_{n-1})$ of elements of~$G$ we define a homomorphism $m_{\underline{\sigma}} : B_n \to B$ by setting
 \[
  m_{\underline{\sigma}}(b_1 \otimes \ldots \otimes b_n) = \sigma_1 b_1 \cdot \ldots \cdot \sigma_{n-1} b_{n-1} \cdot b_n.
 \]
 Consider the isomorphism
 \begin{align*}
  \varphi:B_n	& \longrightarrow \prod_{\underline{\sigma} \in G^{n-1}} B \\
  b		& \mapsto (m_{\underline{\sigma}}(b))_{\underline{\sigma}}.
 \end{align*}
 Via~$\varphi$ the elements of $\Spa(B_n,B_n^+)$ correspond to the valuations of $\prod_{\underline{\sigma} \in G^{n-1}} B$ of the form
 \[
  |(b_{\underline{\sigma}})_{\underline{\sigma}}|' = |b_{\underline{\sigma}^0}(y)|
 \]
 for fixed $\underline{\sigma}^0$ and a valuation $y \in \Spa(B,B^+)$.
 As $\Spa(B,B^+)$ is local with closed point corresponding to~$|\cdot|$, it suffices to test condition~(i) for valuations as above with $|\cdot(y)| = |\cdot|$.
 For an element of $B_n$ of the form $b_1 \otimes \ldots \otimes b_n$ and any $\underline{\sigma} \in G^{n-1}$ we have
 \[
  |m_{\underline{\sigma}}(b_1 \otimes \ldots \otimes b_n)| = |b_1|\cdot \ldots \cdot |b_n|
 \]
 because $B$ is henselian.
 Together with the triangle inequality this proves that~(ii) implies~(i).

 Set
 \[
  C = A[T_1,\ldots,T_{r-1}]/(T_1^{m_1}-\alpha_1,\ldots,T_{r-1}^{m_{r-1}}-\alpha_{r-1}).
 \]
 This is an intermediate extension of $B/A$ and $B = C[T_r]/(T_r^{m_r}-\alpha_r)$.
 By flatness we can view $C_n = C \otimes_A \ldots \otimes_A C$ as a subalgebra of~$B_n$.
 Denote by~$\Gamma_C$ the value group of the restriction of~$|\cdot|$ to~$C$.
 Then $e_{\gamma}$ for $\gamma \in \Gamma_C/\Gamma_A \subset \Gamma_B/\Gamma_A$ form a basis of $C_n/A$.
 Moreover,
 \[
  \{T_r^{i_1} \otimes \ldots \otimes T_r^{i_n} \mid 0 \le i_1,\ldots,i_n \le m-1\}
 \]
 constitutes a basis of $B_n$ over $C_n$.
 Taking all combinations of products
 \[
  e_{\gamma} \cdot (T_r^{i_1} \otimes \ldots \otimes T_r^{i_n})
 \]
 with $i_j \in \{0,\ldots,m_r-1\}$ and $\gamma \in \Gamma_C/\Gamma_A$ yields the basis $\{e_{\gamma}\}_{\gamma \in \Gamma_B/\Gamma_A}$.
 Fix a primitive $m_r$-th root of unity~$\zeta \in A^+$ and denote by~$\sigma$ the element of~$G$ which maps~$T_r$ to $\zeta T_r$ and leaves~$C$ invariant.
 Every element of~$G$ can be written in the form $\tau \sigma^j$ for $0 \le j \le m_r-1$ and $\tau \in G$ with $\tau \zeta = \zeta$.
 For an $(n-1)$-tuple $\underline{\sigma} = (\tau_1 \sigma^{j_1},\ldots,\tau_{n-1} \sigma^{j_{n-1}})$ in $G^{n-1}$
  and an element $b = \sum_{i_1,\ldots,i_n=0}^{m_r-1} a_{i_1,\ldots,i_n} T_r^{i_1} \otimes \ldots \otimes T_r^{i_n}$ of~$B_n$ we have
 \[
  m_{\underline{\sigma}}(b) = \sum_{i_1,\ldots,i_n = 0}^{m_r-1} m_{\underline{\tau}}(a_{i_1,\ldots,i_n}) \zeta^{i_1 j_1 + \ldots i_{n-1} j_{n-1}} T_r^{i_1+\ldots +i_n}.
 \]
 As $|T_r|^k$ for $k = 0,\ldots,{m_r-1}$ represent the $m_r$ distinct elements of $\Gamma_B/\Gamma_C$, we obtain
 \[
  |m_{\underline{\sigma}}(b)| = \max_{0 \le k \le m_r-1} \big|\sum_{i_1+\ldots+i_n \equiv k \mod m_r} m_{\underline{\tau}}(a_{i_1,\ldots,i_n}\alpha_r^{(i_1+\ldots + i_n-k)/m_r}) \zeta^{i_1 j_1 + \ldots + i_{n-1} j_{n-1}}\big| \cdot |T_r|^k.
 \]
 Suppose $|b(x)| \le \delta$ for all $x \in \Spa(B_n,B_n^+)$.
 Then in particular,
 \[
  |m_{\underline{\sigma}}(b)| \le \delta
 \]
 for all $\underline{\sigma} \in G^{n-1}$.
 By the above this is equivalent to
 \begin{equation} \label{valuationinequality}
  |\sum_{i_1+\ldots+i_n \equiv k \mod m_r} m_{\underline{\tau}}(a_{i_1,\ldots,i_n}\alpha_r^{(i_1+\ldots + i_n-k)/m_r}) \zeta^{i_1 j_1 + \ldots + i_{n-1} j_{n-1}}| \le \delta |T_r|^{-k}
 \end{equation}
 for all $\underline{\sigma}$ and all $k=0,\ldots,m_r-1$.
 The following \cref{vandermonde} shows that the matrix $\left(\zeta^{i_1 j_1 + \ldots + i_{n-1} j_{n-1}}\right)$ is invertible in~$A^+$ (note that $A^+$ is a $\ZZ[1/m_r,\zeta]$-module).
 Therefore, inequality~(\ref{valuationinequality}) holds for all $j_1,\ldots,j_{n-1} = 0,\ldots,m_r-1$ if and only if
 \[
  |m_{\underline{\tau}}(a_{i_1,\ldots,i_n}\alpha_r^{(i_1+\ldots + i_n-k)/m_r})| \le \delta |T_r|^{-k}
 \]
 for all $i_1,\ldots,i_{n-1} = 0,\ldots,m_r-1$.
 The result now follows by induction on~$r$.
\end{proof}

\begin{lemma} \label{vandermonde}
 Let~$m$ and~$n$ be positive integers and~$\zeta$ a primitive $m$-th root of unity.
 Over the ring $R = \ZZ[1/m,\zeta]$ we consider the $m^{n-1} \times m^{n-1}$-matrix~$V_n$ whose rows are indexed by the $(n-1)$-tuples $(i_1,\ldots,i_{n-1}) \in \{0,\ldots,m-1\}^{n-1}$
  and whose columns by $(j_1,\ldots,j_{n-1}) \in \{0,\ldots,m-1\}^{n-1}$ (both provided with the lexicographical ordering)
  and whose entry at $(i_1,\ldots,i_{n-1},j_1,\ldots,j_{n-1})$ is $\zeta^{i_1 j_1 + \ldots + i_{n-1} j_{n-1}}$.
 Then~$V_n$ is invertible.
\end{lemma}

\begin{proof}
 Let~$M$ be a free $R$-module with basis $e_0,\ldots,e_{m-1}$ and consider the $R$-linear map~$V$ sending $e_i$ to $\sum_{j=0}^{m-1}\zeta^{ij}e_j$.
 Its transformation matrix is the Vandermonde matrix
 \[
  \begin{pmatrix}
   1		& 1				& 1					& \ldots	& 1				\\
   1		& \zeta			& \zeta^2			& \ldots	& \zeta^{m-1}	\\
   1		& \zeta^2		& \zeta^4			& \ldots	& \zeta^{2(m-1)}	\\
   \vdots	& \vdots		& \vdots			& \ddots	& \vdots		\\
   1		& \zeta^{m-1}	& \zeta^{2(m-1)}	& \ldots	& \zeta^{(m-1)^2}
  \end{pmatrix},
 \]
 whose determinant is
 \[
  \prod_{0\le i<j \le m-1} (\zeta^j - \zeta^i).
 \]
 Since $(\zeta^j - \zeta^i)$ divides~$m$, it is a unit in~$R$.
 Hence, $V$ is invertible.
 Moreover, $V_n$ is the matrix of
 \[
  V^{\otimes (n-1)} : M^{\otimes (n-1)} \to M^{\otimes (n-1)}.
 \]
 It is thus invertible as well.
\end{proof}

\begin{corollary} \label{criterionintegrality}
 The integral closure~$B_n^+$ of $B^+ \otimes_{A^+} \ldots \otimes_{A^+} B^+$ in $B_n$ is the subring generated by
 \[
  \{b_1 \otimes \ldots \otimes b_n \in B_n \mid \prod_{i = 1}^n |b_i| \le 1 \}.
 \]
 An element $\sum_{\gamma_1,\ldots,\gamma_n} a_{\gamma_1,\ldots,\gamma_n} e_{\gamma_1} \otimes \ldots \otimes e_{\gamma_n}$ is integral over $B^+ \otimes_{A^+} \ldots \otimes_{A^+} B^+$ if and only if
 \[
  |a_{\gamma_1,\ldots,\gamma_n}| \le |e_{\gamma_1} \cdot \ldots \cdot e_{\gamma_n}|^{-1}
 \]
 for all $\gamma_1,\ldots,\gamma_n \in \Gamma_B/\Gamma_A$.
\end{corollary}

\begin{proof}
 By \cite{Hu93}, Lemma~3.3, an element~$b$ of $B_n$ is contained in~$B_n^+$ if and only if $|b(x)| \le 1$ for all $x \in \Spa(B_n,B_n^+)$.
 The result thus follows by \cref{basistameextension} with $\delta = 1$.
\end{proof}

Since~$B$ is faithfully flat over~$A$ and~$B^+$ is faithfully flat over~$A^+$ by \cref{flatness}, we obtain a diagram of exact Amitsur complexes
\[
 \begin{tikzcd}
  0	\ar[r]	& A^+	\ar[r]	\ar[d,hook]	& B^+	\ar[r]	\ar[d,hook]	& B^+ \otimes_{A^+} B^+	\ar[r]	\ar[d,hook]	& B^+ \otimes_{A^+} B^+ \otimes_{A^+} B^+	\ar[r]	\ar[d,hook]	& \ldots  \\
  0	\ar[r]	& A	\ar[r]			& B	\ar[r]			& B \otimes_A B		\ar[r]			& B \otimes_A B \otimes_A B			\ar[r]			& \ldots
 \end{tikzcd}
\]
As the image of an integral element is integral, the diagram factors as
\[
 \begin{tikzcd}
  0	\ar[r]	& A^+	\ar[r]	\ar[d,dash,shift left=.5]	\ar[d,dash,shift right=.5]	& B^+	\ar[r]	\ar[d,dash,shift left=.5]	\ar[d,dash,shift right=.5]	& B^+ \otimes_{A^+} B^+	\ar[r]	\ar[d,hook]	& B^+ \otimes_{A^+} B^+ \otimes_{A^+} B^+	\ar[r]	\ar[d,hook]	& \ldots  \\
  0	\ar[r]	& A^+	\ar[r]	\ar[d,hook]							& B^+	\ar[r]	\ar[d,hook]							& (B \otimes_A B)^+	\ar[r]	\ar[d,hook]	& (B \otimes_A B \otimes_A B)^+			\ar[r]	\ar[d,hook]	& \ldots  \\
  0	\ar[r]	& A	\ar[r]									& B	\ar[r]									& B \otimes_A B		\ar[r]			& B \otimes_A B \otimes_A B			\ar[r]			& \ldots
 \end{tikzcd}
\]

\begin{proposition} \label{tameAmitsurcomplex}
 Let $(A,A^+) \to (B,B^+)$ be a local, Cartesian, tame homomorphism of strongly henselian, local, complete, Huber pairs.
 Assume moreover that~$A$ is noetherian.
 Denote by $(B \otimes_A B)^+$ the integral closure of $(B^+ \otimes_{A^+} B^+)$ in $(B \otimes_A B)$ and similarly for $(B \otimes_A B \otimes_A B)^+$ etc.
 Then the complex
 \[
  \begin{tikzcd}
   0	\ar[r]	& A^+	\ar[r]	& B^+	\ar[r]	& (B \otimes_A B)^+	\ar[r]	& (B \otimes_A B \otimes_A B)^+	\ar[r]	& \ldots
  \end{tikzcd}
 \]
 is exact.
\end{proposition}

\begin{proof}
 Consider the section $s$ of the inclusion $A \hookrightarrow B$ sending an element $\sum_{\gamma} a_{\gamma} e_{\gamma}$ of~$B$ to the coefficient~$a_1$ of~$e_1 = 1$.
 Mapping $b_1 \otimes \ldots \otimes b_n$ to $s(b_1) \cdot \ldots \cdot s(b_n)$,~$s$ induces a morphism~$\Phi$ of complexes
 \[
 \begin{tikzcd}
  0	\ar[r]	& A	\ar[r]		\ar[d]	& B	\ar[r]		\ar[d]	& B \otimes_A B	\ar[r]		\ar[d]	& B \otimes_A B \otimes_A B	\ar[r]		\ar[d]	& \ldots	\\
  0	\ar[r]	& A	\ar[r,"\id"]		& A	\ar[r,"0"]		& A		\ar[r,"\id"]		& A				\ar[r,"0"]		& \ldots
 \end{tikzcd}
\]
It is well known that~$\Phi$ is a homotopy equivalence whose inverse is the natural inclusion~$I$ of the lower complex in the upper one.
Namely, $\Phi \circ I = \id$ and $I \circ \Phi$ is homotopic to the identity by the homotopy given by
\[
 \begin{array}{rcl}
  D_i : B_n	& \longrightarrow	& B_n	\\
  (b_1 \otimes \ldots \otimes b_n)	& \mapsto		& s(b_1) \otimes \ldots \otimes s(b_{i-1}) \otimes b_i \otimes \ldots \otimes b_n.
 \end{array}
\]
In order to show that the complex in the statement of the proposition is exact, it suffices to show that~$\Phi$ restricts to homomorphisms $B_n^+ \to A^+$
 and~$D_i$ to a homomorphism $B_n^+ \to B_n^+$.

Writing~$D_i$ in terms of the basis $\{e_{\gamma}\}_{\gamma}$ we obtain:
\[
 D_i(\sum_{\gamma_1,\ldots,\gamma_n} a_{\gamma_1,\ldots,\gamma_n} e_{\gamma_1} \otimes \ldots \otimes e_{\gamma_n}) = \sum_{\gamma_i,\ldots,\gamma_n} a_{1,\ldots,1,\gamma_i,\ldots \gamma_n} 1 \otimes \ldots \otimes 1 \otimes e_{\gamma_i} \otimes \ldots \otimes e_{\gamma_n}.
\]
Therefore, \cref{criterionintegrality} assures that~$D_i$ maps~$B_n^+$ to~$B_n^+$.
The argument for~$\Phi$ is the same.
\end{proof}

\subsection{Computation of tame cohomology}

\begin{proposition} \label{strictlyetaleacyclic}
 Let $(A,A^+)$ be a strongly henselian Huber pair where~$A$ is either a strongly noetherian Tate ring or noetherian and discrete.
 Then
 \[
  H^i(\Spa(A,A^+)_t,\GG_a^+) = 0
 \]
 for all $i \geq 1$.
\end{proposition}

\begin{proof}
 Let~$\mathcal{B}$ be the category of Cartesian tame morphisms of affinoid adic spaces
 \[
  \Spa (B,B^+) \to \Spa(A,A^+).
 \]
 It has fiber products and becomes a site by defining coverings of~$\Spa(B,B^+)$ to be the Cartesian tame coverings of $\Spa (B,B^+)$.
 By \cref{refinement} we can compute the cohomology groups $H^q(\Spa(A,A^+)_t,\GG_a^+)$ in~$\mathcal{B}$.

 We show that $\GG_a^+$ is flasque on~$\mathcal{B}$.
 Let
 \[
  \Spa(C,C^+) \to \Spa(B,B^+)
 \]
 be a covering in~$\mathcal{B}$.
 We need to show that the \v{C}ech complex for~$\GG_a^+$ associated with this covering is exact.
 Using the notation of \cref{computationintegralclosure} we have
 \[
  \Spa(C,C^+) \times_{\Spa(B,B^+)} \ldots \times_{\Spa(B,B^+)} \Spa(C,C^+) = \Spa(C_n,C_n^+).
 \]
 Note that since~$B$ is henselian, $C_n$ is finite over~$B$, hence complete.
 Therefore,
 \[
  \GG_a^+(\Spa(C_n,C_n^+)) = C_n^+
 \]
 and the \v{C}ech complex for the covering $\Spa(C,C^+) \to \Spa(B,B^+)$ equals
 \[
  0 \to B^+ \to C^+ \to C_2^+ \to C_3^+ \to \ldots
 \]
 This complex is exact by \cref{tameAmitsurcomplex}.
\end{proof}

\begin{corollary} \label{comparestronglyetaletame}
 Let~$\cX$ be a locally noetherian adic space.
 Assume that~$\cX$ is either discretely ringed or analytic.
 The canonical homomorphism
 \[
  H^i(\cX_{\set},\GG_a^+) \to H^i(\cX_t,\GG_a^+)
 \]
 is an isomorphism for all $i \geq 0$.
\end{corollary}

\begin{proof}
 Consider the Leray spectral sequence associated with the morphism of sites
 \[
  \varphi: \cX_t \to \cX_{\set}.
 \]
 We have to show that
 \[
  R^q \varphi_* \GG_a^+ = 0.
 \]
 Put differently, for every strongly henselian Huber pair $(A,A^+)$ where~$A$ is either a strongly noetherian Tate ring or noetherian and discrete we have to show that
 \[
  H^q(\Spa(A,A^+)_t,\GG_a^+) = 0.
 \]
 This is true by \cref{strictlyetaleacyclic}.
\end{proof}

Combining \cref{comparestronglyetaleZariski}, \cref{comparestronglyetaletame} and \cref{compareZariskicohomology} we obtain:

\begin{theorem} \label{comparetamecohomology}
 Let~$X$ be pro-open in a regular scheme~$S$ over~$k$ such that~$X$ is dense in~$S$.
 Assume that resolution of singularities holds over~$k$.
 There is a natural isomorphism
 \[
  H^i(S,\cO_S) \cong H^i(\Spa(X,S)_t,\GG_a^+)
 \]
 for all $i \geq 0$.
\end{theorem}

\section{The Artin-Schreier sequence}

Let~$\cX$ be an adic space with $\ch(\cX) = \{p\}$.
There is an Artin-Schreier sequence
\[
 0 \to \Z/p\Z \longrightarrow \GG_a^+ \overset{F-1}{\longrightarrow} \GG_a^+ \to 0
\]
on~$\cX_t$ and on~$\cX_{\set}$, where~$F-1$ is the homomorphism $x \mapsto x^p-x$.
We can check exactness on stalks.
Let $(A,A^+)$ be strongly henselian.
Then
\[
 F-1 : A^+ \to A^+
\]
is surjective as~$A^+$ is strictly henselian.

\begin{proposition} \label{cohomologypPruefer}
 Let $(A,A^+)$ be a complete Pr\"ufer Huber pair such that~$A$ is of characteristic $p > 0$ and is either noetherian with the discrete topology or a strongly noetherian Tate ring.
 If $\Spa(A,A^+)$ is connected,
 \[
  H^i(\Spa(A,A^+)_t,\Z/p\Z) \cong H^i(\Spa(A,A^+)_{\set},\Z/p\Z) \cong \begin{cases}
                                                                      \Z/p\Z		& i = 0,	\\
                                                                      A^+/(F-1)A^+	& i = 1,	\\
                                                                      0			& i \ge 2.
                                                                     \end{cases}
 \]
\end{proposition}

\begin{proof}
 This follows from \cref{stronglyetaleacyclic-discrete}, \cref{non-archimedeanfieldacyclic} and \cref{comparestronglyetaletame} via the Artin-Schreier sequence.
\end{proof}

\begin{corollary}
 Let~$\cX$ be a locally noetherian adic space with $\ch(\cX) = \{p\}$ which is either analytic or discretely ringed.
 Then the Leray spectral sequence associated with $\cX_t \to \cX_{\set}$ induces isomorphisms
 \[
  H^i(\cX_t,\Z/p\Z) \cong H^i(\cX_{\set},\Z/p\Z)
 \]
 for all $i \ge 0$.
\end{corollary}

\begin{proposition} \label{cohomologyp}
 Let~$S$ be an affine, regular, and integral scheme of characteristic $p > 0$ and $X$ dense and pro-open in~$S$.
 Assume that resolution of singularities holds over~$S$.
 Then we have
 \[
  H^i(\Spa(X,S)_t,\Z/p\Z) \cong H^i(\Spa(X,S)_{\set},\Z/p\Z) \cong \begin{cases}
							            \Z/p\Z			& i = 0,	\\
                                                                    \cO_S(S)/(F-1)\cO_S(S)	& i = 1,	\\
                                                                    0				& i \geq 2.
                                                                   \end{cases}
 \]
\end{proposition}

\begin{proof}
 This follows from \cref{comparetamecohomology} via the Artin-Schreier sequence.
\end{proof}

\begin{corollary} \label{purity}
 Let~$S$ be a regular integral  scheme of characteristic $p > 0$ and $X$ dense and pro-open in~$S$.
 Assume that resolution of singularities holds over~$S$.
 The Leray spectral sequences associated with the morphisms of sites $\Spa(X,S)_t \to \Spa(X,S)_{\set}$ and $\Spa(X,S)_{\set} \to S_{\et}$ induce natural isomorphisms
 \[
  H^i(S_{\et},\Z/p\Z) \cong H^i(\Spa(X,S)_{\set},\Z/p\Z) \cong H^i(\Spa(X,S)_t,\Z/p\Z)
 \]
 for all $i \ge 0$.
\end{corollary}

\begin{proof}
 It suffices to show that
 \[
  H^i(\Spa(X,S)_t,\Z/p\Z) = H^i(\Spa(X,S)_{\set},\Z/p\Z) = 0
 \]
 for $i>0$ in case~$S$ is strictly henselian.
 This follows directly from the description given in \cref{cohomologyp}.
\end{proof}

\begin{corollary}[Purity] \label{purity2}
 Let~$S$ be a noetherian scheme of characteristic $p > 0$ and $X$ a regular scheme which is separated and essentially of finite type over~$S$.
 Assume that resolution of singularities holds over~$S$ and let~$\bar{X}$ be a regular compactification of~$X$ over~$S$.
 Then
 \[
  H^i(\Spa(X,S)_t,\Z/p\Z) \cong H^i(\bar{X}_{\et},\Z/p\Z).
 \]
 In particular, for any pro-open dense subscheme $U \subseteq X$ we have
 \[
  H^i(\Spa(U,S)_t,\Z/p\Z) \cong H^i(\Spa(X,S)_t,\Z/p\Z).
 \]
\end{corollary}

\begin{proof}
 This follows by applying \cref{purity} to $X \to \bar{X}$.
\end{proof}

Note that the above theorem implies in particular that $H^i(\bar{X}_{\et},\Z/p\Z)$ is independent of a choice of compactification.

\begin{corollary}[Homotopy invariance] \label{homotopyinvariance}
 Let~$S$ be a noetherian scheme of characteristic $p > 0$ and $X$ a regular scheme which is essentially of finite type over~$S$.
 Assume that resolution of singularities holds over~$S$.
 Then
 \[
  H^i(\Spa(X,S)_t,\Z/p\Z) \cong H^i(\Spa(\A^1_X,S)_t,\Z/p\Z).
 \]
\end{corollary}

\begin{proof}
 Let~$\bar{X} \to S$ be a regular compactification of~$X$.
 Then~$\P^1_{\bar{X}}$ is a regular compactification of~$\A^1_X$.
 By \cref{purity2}, we have natural isomorphisms
 \begin{IEEEeqnarray*}{rCl}
  H^i(\Spa(X,S)_t,\Z/p\Z)		& \cong	& H^i(\bar{X}_{\et},\Z/p\Z),		\\
  H^i(\Spa(\A^1_X,S)_t,\Z/p\Z)	& \cong & H^i(\P^1_{\bar{X},\et},\Z/p\Z).
 \end{IEEEeqnarray*}
 It remains to show that the two \'etale cohomology groups on the right are isomorphic.
 We consider the projection $\pi: \P^1_{\bar{X}} \to \bar{X}$ and the associated higher direct images $R^j\pi_*\Z/p\Z$.
 For $j = 0$ we directly see $\pi_*\Z/p\Z = \Z/p\Z$.
 Assume now $j > 0$.
 By proper base change, the stalk at a geometric point $\bar{x} \to \bar{X}$ is isomorphic to
 \[
  H^j(\P^1_{k(\bar{x}),\et},\Z/p\Z) = 0.
 \]
 We conclude that $R\pi_*\Z/p\Z \cong \Z/p\Z$ and hence
 \[
  H^i(\P^1_{\bar{X},\et},\Z/p\Z) \cong H^i(\bar{X}_{\et},\Z/p\Z).
 \]
 \end{proof}

\bibliographystyle{../meinStil}
\bibliography{../citations}

\end{document}